\documentclass[a4paper,12pt]{amsart}
\title[Gepner type stability conditions]{Gepner type stability conditions on 
graded matrix factorizations}
\date{}
\author{Yukinobu Toda}

\usepackage{amscd}
\usepackage{amsmath}
\usepackage{amssymb}
\usepackage{amsthm}
\usepackage{float}
\usepackage[dvips]{graphicx}

\usepackage[all,ps,dvips]{xy}

\usepackage{booktabs}

\usepackage{array}
\usepackage{amscd}
\usepackage[all]{xy}

\DeclareFontFamily{U}{rsfs}{%
\skewchar\font127}
\DeclareFontShape{U}{rsfs}{m}{n}{%
<-6>rsfs5<6-8.5>rsfs7<8.5->rsfs10}{}
\DeclareSymbolFont{rsfs}{U}{rsfs}{m}{n}
\DeclareSymbolFontAlphabet
{\mathrsfs}{rsfs}
\DeclareRobustCommand*\rsfs{%
\@fontswitch\relax\mathrsfs}

\theoremstyle{plain}
\newtheorem{thm}{Theorem}[section]
\newtheorem{prop}[thm]{Proposition}
\newtheorem{lem}[thm]{Lemma}
\newtheorem{sublem}[thm]{Sublemma}
\newtheorem{defi}[thm]{Definition}
\newtheorem{rmk}[thm]{Remark}
\newtheorem{cor}[thm]{Corollary}

\newtheorem{step}{Step}
\newtheorem{sstep}{Step}

\newtheorem{prop-defi}[thm]{Proposition-Definition}
\newtheorem{thm-defi}[thm]{Theorem-Definition}
\newtheorem{lem-defi}[thm]{Lemma-Definition}

\newtheorem{conj}[thm]{Conjecture}
\newtheorem{exam}[thm]{Example}

\newdimen\argwidth
\def\db[#1\db]{
 \setbox0=\hbox{$#1$}\argwidth=\wd0
 \setbox0=\hbox{$\left[\box0\right]$}
  \advance\argwidth by -\wd0
 \left[\kern.3\argwidth\box0 \kern.3\argwidth\right]}

\newcommand{\aA}{\mathcal{A}}
\newcommand{\bB}{\mathcal{B}}
\newcommand{\cC}{\mathcal{C}}
\newcommand{\dD}{\mathcal{D}}

\newcommand{\fF}{\mathcal{F}}

\newcommand{\hH}{\mathcal{H}}

\newcommand{\lL}{\mathcal{L}}
\newcommand{\mM}{\mathcal{M}}

\newcommand{\oO}{\mathcal{O}}
\newcommand{\pP}{\mathcal{P}}
\newcommand{\qQ}{\mathcal{Q}}

\newcommand{\sS}{\mathcal{S}}
\newcommand{\tT}{\mathcal{T}}

\newcommand{\xX}{\mathcal{X}}

\newcommand{\lr}{\longrightarrow}

\newcommand{\Hom}{\mathop{\rm Hom}\nolimits}

\newcommand{\dR}{\mathbf{R}}

\newcommand{\Pic}{\mathop{\rm Pic}\nolimits}

\newcommand{\id}{\textrm{id}}

\newcommand{\ch}{\mathop{\rm ch}\nolimits}

\newcommand{\td}{\mathop{\rm td}\nolimits}
\newcommand{\Ext}{\mathop{\rm Ext}\nolimits}

\newcommand{\rank}{\mathop{\rm rank}\nolimits}
\newcommand{\Coh}{\mathop{\rm Coh}\nolimits}

\newcommand{\cneq}{\mathrel{\raise.095ex\hbox{:}\mkern-4.2mu=}}
\newcommand{\eqcn}{\mathrel{=\mkern-4.5mu\raise.095ex\hbox{:}}}

\newcommand{\Cok}{\mathop{\rm Cok}\nolimits}

\newcommand{\Aut}{\mathop{\rm Aut}\nolimits}

\newcommand{\Cone}{\mathop{\rm Cone}\nolimits}

\newcommand{\Stab}{\mathop{\rm Stab}\nolimits}

\newcommand{\DT}{\mathop{\rm DT}\nolimits}

\newcommand{\grr}{\mathrm{gr} \mathchar`- }
\newcommand{\grrproj}{\mathrm{grproj} \mathchar`- }

\newcommand{\Imm}{\mathop{\rm Im}\nolimits}

\newcommand{\HMF}{\mathrm{HMF}^{\rm{gr}}}
\newcommand{\Ker}{\mathop{\rm Ker}\nolimits}
\newcommand{\Ree}{\mathop{\rm Re}\nolimits}

\newcommand{\ST}{\mathop{\rm ST}\nolimits}

\begin{document}

\begin{abstract}
We introduce the notion of 
Gepner type Bridgeland stability conditions 
on triangulated categories, which depends on 
a choice of an autoequivalence and a complex number. 
We conjecture the existence of Gepner type stability 
conditions on the triangulated categories of graded
matrix factorizations
of weighted homogeneous polynomials. 
Such a stability condition 
may give a natural stability condition for Landau-Ginzburg 
B-branes, 
and correspond to the Gepner point 
of the stringy K$\ddot{\rm{a}}$hler moduli space of 
a quintic 3-fold. 
The main result is to show 
our conjecture when the variety 
defined by the weighted
homogeneous polynomial is 
a complete intersection of hyperplanes in 
a Calabi-Yau 
manifold with dimension less than or equal to two. 
\end{abstract}

\maketitle

\setcounter{tocdepth}{1}
\tableofcontents

\section{Introduction}

\subsection{Motivation}
The Donaldson-Thomas (DT) invariants
enumerate 
semistable coherent sheaves on 
Calabi-Yau 3-folds, 
which 
have drawn much attention recently~\cite{Thom}. 
We are in interested in the following two problems
in DT theory: 
\begin{itemize}
\item Find constraints among DT invariants 
induced by autoequivalences of the derived category of 
coherent sheaves, e.g. Seidel-Thomas twists~\cite{ST}. 

\item Construct DT type invariants
counting B-branes on Landau-Ginzburg (LG)
models associated to a superpotential.  
\end{itemize}
As for the former problem, there are 
several predictions in string theory on 
generating series of DT invariants, 
e.g. $S$-duality conjecture, 
Ooguri-Strominger-Vafa conjecture~\cite{DM},~\cite{OSV}.
There seem to be mysterious constraints among 
DT invariants behind such predictions, 
and we hope to reveal their origins via 
symmetries in the derived category.   
We believe that a key 
step toward this problem is 
to construct
a Bridgeland stability condition
on the derived category~\cite{Brs1}
satisfying a certain symmetric property
with respect to the given autoequivalence. 
Indeed a construction of a (weak) stability condition 
which is preserved under the derived dual,  
together with wall-crossing arguments~\cite{JS},~\cite{K-S},
play crucial roles in the
proof of the rationality of the 
generating series of rank
one DT type invariants counting 
curves~\cite{BrH},~\cite{Tolim2},~\cite{Tsurvey}. 
 
As for the latter problem, 
in order to define the DT type invariants, 
we need to fix a 
stability condition 
for B-branes on LG models. 
A desired stability condition should 
be natural in some sense, 
so that it is an analogue of Gieseker 
stability on coherent sheaves. 
In a mathematical term, 
if the superpotential is given by a 
homogeneous polynomial $W$, the 
relevant B-brane category
is Orlov's triangulated category of
 graded matrix factorizations $\HMF(W)$~\cite{Orsin}. 
For instance, suppose that $W$ is
the defining polynomial of a 
quintic Calabi-Yau 3-fold 
$X \subset \mathbb{P}^4$. 
Then a desired Bridgeland 
stability condition on $\HMF(W)$
may 
correspond to the Gepner point (cf.~Figure~\ref{fig:one})
of the stringy K$\ddot{\rm{a}}$hler moduli 
space of $X$, via mirror symmetry and 
Orlov equivalence~\cite{Orsin}
\begin{align}\label{intro:Orlov}
 D^b \Coh(X) \stackrel{\sim}{\to} \HMF(W). 
\end{align}
The Gepner point is 
an orbifold point in the stringy 
K$\ddot{\rm{a}}$hler moduli space of $X$, 
with the stabilizer group
$\mathbb{Z}/5\mathbb{Z}$. 
Such an orbifold data may be translated into 
a certain symmetric property of the corresponding 
Bridgeland stability condition, which we focus and 
pursue in this paper. 

Now we have observed a common keyword
regarding the above two problems, 
that is \textit{a Bridgeland stability condition 
with a symmetric property}. 
The motivating problem of this paper, which is 
rather ambitious and not able to do at this moment, 
is 
to find a natural stability condition on $\HMF(W)$, 
and apply its symmetric property to 
obtain non-trivial constraints among DT invariants on 
$X$, via Orlov equivalence (\ref{intro:Orlov})
and wall-crossing arguments~\cite{JS},~\cite{K-S}. 
(cf.~Subsection~\ref{subsec:future} (iii).) 

\begin{figure}[htbp]
 \begin{center}
  \includegraphics[width=60mm]{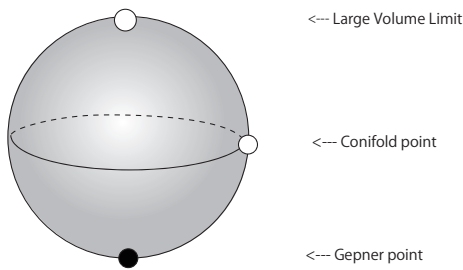}
 \end{center}
 \caption{Stringy K$\ddot{\textrm{a}}$hler moduli space
 of a quintic 3-fold}
 \label{fig:one}
\end{figure}

\subsection{Gepner type stability conditions}
As we discussed, we are interested in constructing 
a natural stability condition on $\HMF(W)$. 
So far there are few works on this problem, 
except~\cite{KST1},~\cite{Tak},~\cite{Wal}, 
which will be discussed later. 
A serious issue is that, since the  
category $\HMF(W)$ is \textit{not} a priori constructed as 
a derived category of some abelian category, 
it is not clear what is the meaning of `natural'. 
Our viewpoint 
is as follows:
rather than constructing a stability 
condition in terms of graded matrix factorizations, 
we just extract and formulate the 
symmetric property of a
desired stability condition, and try to 
find a one satisfying such a property. 
We formulate it as 
a \textit{Gepner type}
property, 
which depends on
a choice of a pair of an autoequivalence and a 
complex number, given simply as follows:  
\begin{defi}
A stability condition 
$\sigma$ on a triangulated category $\dD$
is called 
Gepner type with 
respect to $(\Phi, \lambda) \in \Aut(\dD) \times \mathbb{C}$
if the following condition holds: 
\begin{align}\label{intro:eq:gep}
\Phi_{\ast} \sigma= \sigma \cdot (\lambda). 
\end{align}
\end{defi}
Obviously any stability condition is 
Gepner type with respect to $([k], k)$
for $k\in \mathbb{Z}$, where $[k]$ is the 
$k$-times composition of the shift functor 
$[1]$.  
On the other hand, 
there are several interesting 
examples in which the relation (\ref{intro:eq:gep}) holds 
with respect to non-trivial 
pairs $(\Phi, \lambda)$. 
As the name indicates, 
if $X \subset \mathbb{P}^4$ is a 
quintic 3-fold, 
a 
Gepner type 
stability condition on $D^b \Coh(X)$
with 
respect to the pair 
\begin{align}\label{pair:ST}
(\Phi, \lambda)=
\left(\mathrm{ST}_{\oO_X} \circ \otimes \oO_X(1), \frac{2}{5} \right)
\end{align}
conjecturally corresponds to the Gepner point, 
where $\mathrm{ST}_{\oO_X}$ is the Seidel-Thomas
twist~\cite{ST} associated to $\oO_X$. 
The fraction $2/5$ appears since the five times composition 
of the autoequivalence $\Phi$ 
becomes the twice shift functor. 
This fact is best understood in
terms of $\HMF(W)$ via the equivalence
 (\ref{intro:Orlov}), as the 
autoequivalence $\Phi$ corresponds to the grade 
shift functor on the matrix factorization side. 

At this moment, 
it seems to be a difficult problem 
to construct a stability condition 
on a quintic 3-fold corresponding to the 
Gepner point:
we are not even 
able to construct a Bridgeland stability condition on 
a quintic 3-fold near the large volume limit point. 
(cf.~\cite{BMT}.)
However there is a plenty of examples of 
weighted homogeneous polynomials, which are 
more amenable than quintic polynomials, but 
enough interesting to study. 
The goal of this paper is to construct 
Gepner type stability conditions on graded
matrix factorizations in some of 
such interesting cases. 
In the quintic case, 
an attempt to construct a Gepner 
point leads to a conjectural stronger version of 
 Bogomolov-Gieseker inequality
among Chern characters of 
 stable sheaves on quintic 3-folds. 
The detail in this case will be discussed 
in a subsequent paper~\cite{TGep2}.

\subsection{Gepner type stability conditions on 
 graded matrix factorizations}
Let $W$ be a homogeneous element with 
degree $d$ in the weighted polynomial ring
\begin{align*}
W \in A \cneq 
\mathbb{C}[x_1, x_2, \cdots, x_n], \quad \deg x_i = a_i \in \mathbb{Z}
\end{align*}
such that $(W=0) \subset \mathbb{C}^n$ 
has only an isolated singularity at the origin. 
The triangulated category 
of graded matrix factorizations of $W$, denoted by 
$\mathrm{HMF}^{\rm{gr}}(W)$, 
is 
defined to be the homotopy category of the dg category whose 
objects consist of data 
\begin{align*}
P^{0} \stackrel{p^0}{\to} P^1
 \stackrel{p^1}{\to} P^0(d)
\end{align*}
where $P^i$ are 
graded free $A$-modules of finite rank, 
$p^i$ are homomorphisms of graded 
$A$-modules, $P^i \mapsto P^i(1)$ is the shift of the 
grading, satisfying the following 
condition:
\begin{align*}
\quad p^1 \circ p^0= \cdot W, \quad 
p^0(d) \circ p^1= \cdot W. 
\end{align*}
In~\cite{Orsin}, Orlov proved that 
the triangulated category $\mathrm{HMF}^{\rm{gr}}(W)$
is related to 
the derived category of coherent sheaves on 
the stack 
\begin{align}\label{W=0}
X \cneq (W=0) \subset \mathbb{P}(a_1, \cdots, a_n)
\end{align}
depending on the sign of the 
Gorenstein index
\begin{align*}
\varepsilon \cneq \sum_{i=1}^n a_i -d. 
\end{align*}
Let $\tau$ be the autoequivalence of $\HMF(W)$
induced by the grade shift functor 
$P^{\bullet} \mapsto P^{\bullet}(1)$. 
We propose the following conjecture 
on the existence of a Gepner type stability 
condition on $\mathrm{HMF}^{\rm{gr}}(W)$: 
\begin{conj}\label{conj:Gep}
There is a Gepner type stability 
condition 
\begin{align*}
\sigma_G=(Z_G, \{\pP_G(\phi)\}_{\phi \in \mathbb{R}})
\in \Stab(\mathrm{HMF}^{\rm{gr}}(W))
\end{align*}
with respect to $(\tau, 2/d)$, whose 
central charge $Z_G$ is given by 
\begin{align*}
Z_G(P^{\bullet})=\mathrm{str}(e^{2\pi \sqrt{-1}/d} \colon 
P^{\bullet} \to P^{\bullet}). 
\end{align*}
Here the $e^{2\pi \sqrt{-1}/d}$-action on 
$P^{\bullet}$ is induced by the $\mathbb{Z}$-grading 
on each $P^{i}$, and $`\mathrm{str}'$ is the supertrace
which respects the $\mathbb{Z}/2\mathbb{Z}$-grading on 
$P^{\bullet}$. 
\end{conj}
Similarly to the quintic case, the fraction 
$2/d$ appears since the $d$-times 
composition of $\tau$ coincides with $[2]$. 
We propose that, if there is a 
desired stability condition $\sigma_G$
in Conjecture~\ref{conj:Gep},
 then it may be employed as a `natural' stability condition 
for graded matrix factorizations. 
There are at least three reasons for this. 
Firstly  
the Gepner type property with respect to $(\tau, 2/d)$ 
resembles the following property
of the classical Gieseker stability
on coherent sheaves:
for a polarized variety $(X, H)$, 
a coherent sheaf $E$ on $X$ 
is $H$-Gieseker semistable 
if and only if $E \otimes \oO_X(H)$ is 
$H$-Gieseker semistable.
 In this sense, 
the desired stability condition $\sigma_G$ 
seems to be a natural analogue of the Gieseker 
stability on $\Coh(X)$ for 
graded matrix factorizations.
Secondly, the Gepner type
property with respect to $(\tau, 2/d)$ 
turns out to be a very strong constraint
for the stability conditions. 
Indeed, such a property 
characterizes the central charge $Z_G$
uniquely up 
to a scalar multiplication.
(cf.~Subsection~\ref{subsec:unique1}.)
Also in some cases, we see that
$\sigma_G$ is also
unique up to shift (cf.~Subsection~\ref{subsec:unique2})
and one may expect that this holds in general. 
Thirdly, if there is such $\sigma_G$, then 
the $\sigma_G$-semistable 
objects have a nice compatibility with 
the Serre functor on $\HMF(W)$. 
As a result, 
the moduli space of $\sigma_G$-semistable objects 
should have good local properties, e.g. 
smoothness, with a perfect obstruction theory, etc, 
depending on the given data $n, d, \varepsilon$. 
(cf.~Subsection~\ref{subsec:future} (ii).)
Therefore, whatever the way $\sigma_G$ is 
constructed, 
it seems worth studying moduli spaces of 
$\sigma_G$-semistable objects, and the enumerative invariants
defined by them. 

\subsection{Result}
Before stating our result, 
we mention the previous beautiful works by 
Takahashi~\cite{Tak} and 
Kajiura-Saito-Takahashi~\cite{KST1}. 
They study the 
triangulated category $\HMF(W)$ 
in the following cases: 
$n=a_1=1$~\cite{Tak}
and $n=3$, $\varepsilon >0$~\cite{KST1}. 
In these cases, they show that $\HMF(W)$
is equivalent to the 
derived category of representations of 
a quiver of ADE type. 
As a result, $\HMF(W)$
has only a finite number of indecomposable 
objects up to shift, 
which are completely classified. 
Using such a classification, they construct 
a stability condition on $\HMF(W)$ by 
the assignment of
a phase for each indecomposable object.
Their construction satisfies 
our Gepner type property, so
Conjecture~\ref{conj:Gep} is proved in these cases. 
(cf.~Subsection~\ref{subsec:rmk:n=3}.)  

From our motivation, we are rather interested in 
a case that there is an  
infinite number of indecomposable 
objects up to shift, which are hard to classify, and 
form non-trivial moduli spaces. 
For instance if $\varepsilon \le 0$, then 
the stack (\ref{W=0}) is either Calabi-Yau or
general type, and 
it seems hopeless to construct $\sigma_G$
via classification of indecomposable objects. 
Our main result, formulated as follows, 
contains such cases:  
(cf.~Propositions~\ref{prop:30}, \ref{prop:2-1}, 
\ref{prop:40}, \ref{prop:3-1}, \ref{prop:2-2}.)
\begin{thm}\label{thm:main}
Conjecture~\ref{conj:Gep} is true
if $n-4 \le \varepsilon \le 0$
and
the stack $X$
defined by (\ref{W=0})
does not contain stacky points. 
\end{thm}
The assumption that $X$ does not contain 
stacky points means that $X$ is indeed 
a smooth projective variety. 
The inequality $n-4 \le \varepsilon \le 0$
implies that $X$ is 
contained in a
Calabi-Yau manifold 
of dimension less than or equal to two
as a codimension $-\varepsilon$
complete intersection of hyperplanes. 
These conditions are restrictive, and 
the possible types $(a_1, \cdots, a_n, d, W, X)$ 
are completely classified. 
The classification 
with $n-\varepsilon=3, 4$ and $W$ of Fermat type 
is given in Table~1 below.

\begin{table}\label{table}
\caption{Possible types in Theorem~\ref{thm:main}}
\begin{center}

\begin{tabular}{rrrrrr} \toprule
$n$ & $\varepsilon$ & $(a_1, \cdots, a_n)$ & $d$ & $W$ & $X$ \\
\midrule
$4$ & $0$ & (1, 1, 1, 1) & 4 & $x_1^4 + x_2^4 + x_3^4 + x_4^4$ & K3 surface \\
$4$ & $0$ & (3, 1, 1, 1) & 6 & $x_1^2 + x_2^6 + x_3^6 + x_4^6$ & K3 surface \\
$3$ & $-1$ & (1, 1, 1) & 4 & $x_1^4 + x_2^4 + x_3^4$ & genus 3 curve \\
$3$ & $-1$ & (3, 1, 1) & 6 & $x_1^2 + x_2^6 + x_3^6$ & genus 2 curve \\
$2$ & $-2$ & (1, 1) & 4 & $x_1^4 + x_2^4$ & 4 points \\
$2$ & $-2$ & (3, 1) & 6 & $x_1^2 + x_2^6$ &  2 points \\
$3$ & $0$ & (1, 1, 1) & 3 & $x_1^3 + x_2^3 + x_3^3$ & elliptic curve \\
$3$ & $0$ & (2, 1, 1) & 4 & $x_1^2 + x_2^4 + x_3^4$ & elliptic curve \\
$3$ & $0$ & (3, 2, 1) & 6 & $x_1^2 + x_2^3 + x_3^6$ & elliptic curve \\
$2$ & $-1$ & (1, 1) & 3 & $x_1^3 + x_2^3$ & 3 points \\
$2$ & $-1$ & (2, 1) & 4 & $x_1^2 + x_2^4$ & 2 points \\
$2$ & $-1$ & (3, 2) & 6 & $x_1^2 + x_2^3$ & 1 point \\
\bottomrule
\end{tabular}

\end{center}

\end{table}

Our strategy proving Theorem~\ref{thm:main}
is as follows: 
by Orlov's theorem~\cite{Orsin}, 
the condition $\varepsilon \le 0$
allows us to
describe $\HMF(W)$ as a semiorthogonal 
decomposition
\begin{align}\notag
\HMF(W)=\langle \mathbb{C}(-1-\varepsilon), \cdots, 
\mathbb{C}(0), \Psi D^b \Coh(X) \rangle
\end{align}
where $\Psi$ is a fully faithful functor 
from $D^b \Coh(X)$ to $\HMF(W)$, and 
$\mathbb{C}(i)$
 is a certain exceptional object. 
We show that there is the 
heart of a bounded t-structure
on $\HMF(W)$, given by the 
extension closure
\begin{align}\notag
\aA_W \cneq 
\langle \mathbb{C}(-1-\varepsilon), \cdots, 
\mathbb{C}(0), \Psi \Coh(X) \rangle_{\rm ex}. 
\end{align} 
We describe the central charge $Z_G$ in terms of 
the generators of the heart $\aA_W$, 
and take a suitable tilting $\aA_G$
of $\aA_W$ using the description of $Z_G$. 
We then show that the pair
 $(Z_G, \aA_G)$
determines a 
stability condition $\sigma_G$
on $\HMF(W)$. 
It remains to show that $\sigma_G$ has a 
desired Gepner type property, and we reduce 
it to showing $\sigma_G$-stability of 
certain objects in $\HMF(W)$. 

We show that the above construction works 
in the situation of Theorem~\ref{thm:main}. 
However unfortunately, 
several case by case arguments are involved in 
the proof, and that prevents us to construct 
$\sigma_G$ beyond the cases in Theorem~\ref{thm:main}. 
For instance, they contain 
the proofs of the inequalities of numerical classes of 
certain objects in $\aA_W$, which are required 
in proving the axiom of Bridgeland stability for 
$(Z_G, \aA_G)$. 
They also contain checking the 
$\sigma_G$-stability of some objects in $\HMF(W)$, 
which we use to prove the Gepner type 
property of $\sigma_G$. 
The above arguments are in particular hard 
if the image of $Z_G$ is not discrete. 
In the situation of Theorem~\ref{thm:main}, 
fortunately, the image of $Z_G$
is always discrete and that makes our situation
 technically 
rather amenable. 
For instance 
the image of $Z_G$ is not discrete if $d=5$, and 
the case of $(a_1, a_2, d)=(1, 1, 5)$ is 
not included in Table~1. 
Such a case will be discussed in a subsequent paper~\cite{TGep2}, 
since it involves a subtle argument due
to the above non-discrete issue.  

Among the list in Table~1, 
the case of $(n, \varepsilon)=(3, -1)$ seems to 
be the most interesting case, since 
we observe a new phenomenon relating 
graded matrix factorizations and coherent 
systems on the smooth projective 
curve $X$. 
Recall that a coherent system on $X$ consists of data
\begin{align*}
V \otimes \oO_X \to F 
\end{align*}
where $F$ is a coherent sheaf on $X$
and $V$ is a finite dimensional $\mathbb{C}$-vector 
space.   
We show that the heart $\aA_W$ 
is equivalent to the 
abelian category of coherent systems 
on $X$. 
In the construction of $\sigma_G$, 
we see  
that a Clifford type 
bound on the dimension of $V$
is involved. 
Such a Clifford type bound 
for semistable coherent systems
is established by Lange-Newstead~\cite{LaNe}, 
and we apply their work. 
If $X$
has a higher dimension,  
the conjectural construction of $\sigma_G$ would 
predict a higher dimensional 
analogue of Clifford type bound for 
semistable coherent systems. 
In the case of a quintic surface, 
the detail will
be discussed in~\cite{TGep2}.

\subsection{Future directions of the research}\label{subsec:future}
We believe that the work of this paper 
leads to several interesting directions of
the future research. We discuss some of them. 

{\bf (i) Descriptions of $\sigma_G$-semistable objects in terms of 
graded matrix factorizations:}
As we discussed in the previous subsection, 
our construction of $\sigma_G$ relies on
Orlov's theorem, and it is not intrinsic in
terms of graded matrix factorizations. 
It would be an interesting problem to see
what kinds of 
graded matrix factorizations appear as 
$\sigma_G$-semistable objects, and compare 
them with the $R$-stability discussed in
string theory~\cite{Wal}. 
It may involve deeper understanding of
Orlov equivalence, and we are not even able 
to give a mathematically rigorous
 candidate of
the description of $\sigma_G$ purely 
in terms of 
graded matrix factorizations.

{\bf (ii) Constructing moduli spaces of $\sigma_G$-semistable
graded matrix factorizations:}
It would be an important problem to construct 
moduli spaces of $\sigma_G$-semistable 
graded matrix factorizations, and study their properties. 
We expect that, using the argument of~\cite{Tst3},
there exist Artin stacks of finite type 
\begin{align*}
\mM_G^{\rm{s}}(\gamma) \subset \mM_G^{\rm{ss}}(\gamma), \quad 
\gamma \in \mathrm{HH}_{0}(W)
\end{align*}
which parameterize $\sigma_G$-(semi)stable
graded 
matrix factorizations $P^{\bullet}$ with 
$\ch(P^{\bullet})=\gamma$. 
Here $\mathrm{HH}_{0}(W)$ is the 
zero-th
Hochschild homology group of $\HMF(W)$, 
studied in~\cite{Dyck},~\cite{PoVa2},~\cite{PoVa}.
On the other hand, 
we have the following vanishing for 
$[P^{\bullet}] \in {\mM}_G^{\rm{ss}}(\gamma)$:
 \begin{align*}
\Hom^i(P^{\bullet}, P^{\bullet})=0, \quad
i> n-2 -\frac{2\varepsilon}{d}. 
\end{align*}
The above vanishing, 
which is proved in Lemma~\ref{lem:important}, is one of the 
important properties of Gepner type stability conditions. 
Since the space $\Hom^i(P^{\bullet}, P^{\bullet})$
is responsible for the local deformation theory 
of $P^{\bullet}$, 
the moduli space $\mM_G^{\rm{ss}}(\gamma)$
would have good local properties depending on 
$n, d, \varepsilon$, e.g. 
it is smooth if $4>n-2\varepsilon/d$, 
has a perfect obstruction theory if 
$5>n-2\varepsilon/d$. 
Such properties would be important in 
constructing counting invariants of 
graded matrix factorizations, even in a
non-CY3 situation.

{\bf (iii) DT type invariants counting 
$\sigma_G$-semistable graded matrix factorizations:}
Suppose that $\HMF(W)$ is a
CY3 category. 
If Conjecture~\ref{conj:Gep}
is true, then (ii) would imply the existence
of the invariants
\begin{align*}
\DT_{G}(\gamma) \in \mathbb{Q}, \quad 
\gamma \in \mathrm{HH}_{0}(W)
\end{align*}
which count $\sigma_G$-semistable 
matrix factorizations $P^{\bullet}$
with $\ch(P^{\bullet})=\gamma$. 
The above invariants may give DT analogue 
of the Fan-Jarvis-Ruan-Witten
theory~\cite{FJRW1}
in Gromov-Witten theory.
Also,
the Gepner type property of $\sigma_G$
should yield 
an important identity
\begin{align*}
\DT_{G}(\gamma) = \DT_{G}(\tau_{\ast}\gamma).
\end{align*}
The above identity, combined with 
Orlov's equivalence
and wall-crossing argument~\cite{JS}, \cite{K-S},
may imply non-trivial constraints among the 
original sheaf counting 
DT invariants on $X$ induced by
the autoequivalence
$\mathrm{ST}_{\oO_X} \circ \otimes \oO_X(1)$. 
If the above story works, then it realizes an analogue of 
Calabi-Yau/Landau-Ginzburg
 correspondence in FJRW theory~\cite{ChRu}.

\subsection{Plan of the paper}
In Section~\ref{sec:Gep}, we introduce the notion 
of Gepner type stability conditions
and propose a conjecture on the existence of 
Gepner type stability conditions on 
the triangulated category $\HMF(W)$. 
We also discuss some examples of our 
conjecture, and their uniqueness. 
In Section~\ref{sec:T-st}, we construct 
the heart $\aA_W$ of a bounded t-structure
on 
$\HMF(W)$, and
describe it in terms of quiver 
representations or coherent systems. 
In Section~\ref{sec:const}, we explain how to 
compute the central charge in terms of 
generators of $\aA_W$, and 
propose a general recipe on a construction of 
a Gepner type stability condition. 
In Section~\ref{sec:proof}, we prove Theorem~\ref{thm:main}
by applying the strategy in Section~\ref{sec:const}.

\subsection{Acknoledgement}
The author would like to thank 
Kentaro Hori, Kyoji Saito and Atsushi 
Takahashi for valuable discussions. 
This work is supported by World Premier 
International Research Center Initiative
(WPI initiative), MEXT, Japan. This work is also supported by Grant-in Aid
for Scientific Research grant (22684002)
from the Ministry of Education, Culture,
Sports, Science and Technology, Japan.

\section{Gepner type stability conditions}\label{sec:Gep}
In this section, we recall the definition of 
Bridgeland stability conditions on triangulated categories, 
group actions, and define the notion of Gepner 
type stability conditions. We then recall Orlov's 
triangulated categories of graded matrix factorizations, and 
discuss Gepner type stability conditions on them. 

\subsection{Definitions}
Let $\dD$ be a triangulated category and $K(\dD)$ its 
Grothendieck group. 
We first recall Bridgeland's definition of 
stability conditions on it. 
\begin{defi}\label{defi:stab} \emph{(\cite{Brs1})}
A stability condition $\sigma$
on $\dD$ consists of a pair $(Z, 
\{\pP(\phi)\}_{\phi \in \mathbb{R}})$
\begin{align}\label{pair2}
Z \colon K(\dD) \to \mathbb{C}, \quad 
\pP(\phi) \subset \dD
\end{align}
where $Z$
is a group homomorphism 
(called central charge) 
and $\pP(\phi)$ is
a full subcategory (called $\sigma$-semistable objects 
with phase $\phi$)
satisfying the following conditions: 
\begin{itemize}
\item For $0\neq E \in \pP(\phi)$, 
we have $Z(E) \in \mathbb{R}_{>0} \exp(\sqrt{-1} \pi \phi)$. 
\item For all $\phi \in \mathbb{R}$, we have 
$\pP(\phi+1)=\pP(\phi)[1]$. 
\item For $\phi_1>\phi_2$ and $E_i \in \pP(\phi_i)$, we have 
$\Hom(E_1, E_2)=0$. 

\item For each $0\neq E \in \dD$, there is
 a collection of distinguished triangles 
\begin{align*}
E_{i-1} \to E_i \to F_i \to E_{i-1}[1], \quad 
E_N=E, \ E_0=0
\end{align*}
with $F_i \in \pP(\phi_i)$ and  
$\phi_1> \phi_2> \cdots > \phi_N$. 
\end{itemize}
\end{defi}
The full subcategory $\pP(\phi) \subset \dD$ is 
shown to be an abelian category, and its 
simple objects are called $\sigma$-stable. 
In~\cite{Brs1}, Bridgeland 
shows that there is a natural topology on 
the  
 set of `good' stability conditions 
$\Stab(\dD)$, and 
its each connected component 
has structure of a complex manifold. 

\begin{rmk}
The above `good' conditions are called 
`numerical property' and `support property'
in literatures. 
Although the above properties are important in 
considering the space $\Stab(\dD)$, we 
omit the detail since we focus on 
the construction of one specific stability condition. 
\end{rmk}

Let $\Aut(\dD)$
be the group of autoequivalences on $\dD$. 
There is a left $\Aut(\dD)$-action on 
the set of stability conditions on $\dD$. 
For $\Phi \in \Aut(\dD)$, it
acts on the pair (\ref{pair2}) as 
follows:  
\begin{align*}
\Phi_{\ast} (Z, \{\pP(\phi)\}_{\phi \in \mathbb{R}})
=(Z \circ \Phi^{-1}, \{ \Phi(\pP(\phi)) \}_{\phi \in \mathbb{R}}). 
\end{align*}
There is also a right $\mathbb{C}$-action on 
the set of stability conditions on $\dD$. 
For $\lambda \in \mathbb{C}$, its acts on the pair (\ref{pair2})
as follows: 
\begin{align*}
 (Z, \{\pP(\phi)\}_{\phi \in \mathbb{R}}) \cdot (\lambda)
= (e^{-\sqrt{-1}\pi \lambda} Z, \{ \pP(\phi + \Ree \lambda) \}_{\phi \in \mathbb{R}}). 
\end{align*}
The notion of Gepner type stability 
conditions is defined in terms of the above group actions. 
\begin{defi}
A stability condition 
$\sigma$ on a triangulated category $\dD$
is called 
Gepner type with 
respect to $(\Phi, \lambda) \in \Aut(\dD) \times \mathbb{C}$
if the following condition holds: 
\begin{align}\label{eq:gep}
\Phi_{\ast} \sigma= \sigma \cdot (\lambda). 
\end{align}
\end{defi}

Here we give some trivial examples: 
\begin{exam}
(i) For $k\in \mathbb{Z}$, any stability condition 
$\sigma$ on a triangulated category $\dD$ is Gepner type 
with respect to $([k], k)$. 

(ii) Let $\aA$ be an abelian category 
and $Z$ a group homomorphism 
\begin{align*}
Z \colon K(\aA) \to \mathbb{C}
\end{align*} 
such that $Z(\aA \setminus \{0\})$
is contained in $\mathbb{R}_{>0} e^{\sqrt{-1} \pi \theta}$
for some $\theta \in \mathbb{R}$. 
We set $\pP(\phi)$ for $\phi \in [\theta, \theta+1)$ to be
\begin{align*}
\pP(\theta)=\aA, \quad \pP(\phi)=\{0\} \mbox{ if }\phi \in (\theta, \theta+1).
\end{align*}
Other $\pP(\phi)$ are determined by the rule 
$\pP(\phi+1)=\pP(\phi)[1]$. 
Let $\Phi$ be an 
autoequivalence of $D^b (\aA)$
which preserves $\aA$. 
Then 
the pair 
$(Z, \{\pP(\phi)\}_{\phi \in \mathbb{R}})$ is a Gepner type stability 
condition on $D^b (\aA)$ with respect to 
$(\Phi, 0)$. 
\end{exam}

\begin{rmk}
A Gepner type stability conditions as in Example~(ii) 
appears at a point in the space of stability conditions on
$X=\omega_{\mathbb{P}^2}$
studied by~\cite{BaMa}. 
Such a point corresponds to the orbifold point in the 
stringy K$\ddot{\rm{a}}$hler moduli space of
$X$, and $\aA$ 
is the abelian category of representations of a McKay quiver. 
The autoequivalence $\Phi$ is
given by 
\begin{align*}
\Phi=
\ST_{\oO_{\mathbb{P}^2}} \circ \otimes \oO_X(1)
\end{align*} 
which induces the automorphism of the McKay quiver. 
\end{rmk}

As we will see, there will be more interesting 
examples of Gepner type stability conditions on 
triangulated categories of graded matrix factorization. 

\subsection{Triangulated categories of graded matrix factorizations}
Here we recall Orlov's construction of triangulated
categories of graded matrix factorizations~\cite{Orsin}. 
Let $A$
be a
weighted graded polynomial ring
\begin{align}\label{def:A}
A \cneq \mathbb{C}[x_1, x_2, \cdots, x_n], \quad \deg x_i= a_i
\end{align}
and $W \in A$ a homogeneous element of degree $d$. 
Throughout of this paper, we always assume that
$a_1 \ge a_2 \ge \cdots \ge a_n$, and  
$(W=0) \subset \mathbb{C}^n$ has an isolated 
singularity at the origin. 
For a graded $A$-module $P$, 
we denote by $P_i$ its degree $i$-part, 
and $P(k)$ the graded $A$-module 
whose grade is shifted by $k$, i.e. 
$P(k)_{i}= P_{i+k}$. 
\begin{defi}
A graded matrix factorization of $W$ is data
\begin{align}\label{MF}
P^0 \stackrel{p^0}{\to} P^1 \stackrel{p^1}{\to}
P^0(d)
\end{align}
where $P^i$ are graded free $A$-modules of finite rank, $p^i$ 
are homomorphisms of graded $A$-modules, satisfying the 
following conditions:
\begin{align*}
\quad p^1 \circ p^0= \cdot W, \quad 
p^0(d) \circ p^1= \cdot W.
\end{align*}
\end{defi}
The category $\HMF(W)$ 
is defined to be
the homotopy category of graded matrix factorizations of $W$. 
Its objects consist of 
data (\ref{MF}), and 
the set of morphisms are given by the commutative 
diagrams of graded $A$-modules
\begin{align*}
\xymatrix{
P^0 \ar[r]^{p^0}\ar[d]^{f^0} & P^1 \ar[r]^{p^1}\ar[d]^{f^1}
 & P^0(d) \ar[d]^{f^0(d)} \\
Q^0 \ar[r]^{q^0} & Q^1 \ar[r]^{q^1} & Q^0(d). 
}
\end{align*}
modulo null-homotopic morphisms: the above diagram is 
null-homotopic if there are homomorphisms of 
graded $A$-modules 
\begin{align*}
h^0 \colon P^0 \to Q^1(-d), \quad 
h^1 \colon P^1 \to Q^0
\end{align*}
satisfying 
\begin{align*}
f^0= q^1(-d) \circ h^0 + h^1 \circ p^0, \quad 
f^1 = q^0 \circ h^1 + h^0(d) \circ p^1. 
\end{align*}
The category $\HMF(W)$ has a structure of a triangulated category. 
The shift functor $[1]$ sends data (\ref{MF})
to
 \begin{align*}
P^1 \stackrel{-p^1}{\to} P^0(d) \stackrel{-p^0(d)}{\to} 
P^1(d)
\end{align*}
and the distinguished triangles are defined via the
usual mapping cone constructions. 
The grade shift functor
$P^{\bullet} \mapsto P^{\bullet}(1)$ 
induces the 
autoequivalence $\tau$ of $\HMF(W)$, 
which satisfies the 
 following identity: 
\begin{align}\label{taud}
\tau^{\times d} =[2]. 
\end{align}
There is also a Serre functor $\sS_W$
on $\HMF(W)$, described by $\tau$, $n$
and the \textit{Gorenstein index} $\varepsilon$. 
The number $\varepsilon$ is defined as
\begin{align*}
\varepsilon \cneq \sum_{i=1}^{n} a_i -d \in \mathbb{Z}.
\end{align*}  
The Serre functor $\sS_W$ on $\HMF(W)$ is 
given by (for instance see~\cite[Theorem~3.8]{KST2})
\begin{align}\label{Serre}
\sS_W = \tau^{-\varepsilon}[n-2]. 
\end{align}
The above description of the Serre functor will 
be used later in this paper.

\subsection{Relation to the triangulated categories of 
singularities}\label{subsec:relation}
The triangulated category  
$\HMF(W)$ is known to be 
equivalent to the derived category of 
singularities of the hypersurface 
singularity $(W=0) \subset \mathbb{C}^n$. 
This equivalence is often useful in 
doing some computations of matrix factorizations. 

Let $R$ be the graded ring 
$R=A/(W)$
and 
$\mathrm{gr} \mathchar`- R$
the abelian category of finitely generated 
graded $R$-modules. 
We denote by 
$D^b (\mathrm{grproj} \mathchar`- R)$ 
the subcategory of 
$D^b (\mathrm{gr} \mathchar`- R)$
consisting of perfect complexes of $R$-modules. 
The triangulated category of singularities 
is defined to be the 
quotient category 
\begin{align*}
D_{\rm{sg}}^{\rm{gr}}(R) \cneq D^b (\grr R)
/ D^b(\grrproj R). 
\end{align*}
The following result is proved in~\cite{Orsin}: 
\begin{thm}\emph{(\cite[Theorem~3.10]{Orsin})}
There is an equivalence of triangulated categories
\begin{align}\label{Cok}
\mathrm{Cok} \colon 
\HMF(W) \stackrel{\sim}{\to} D_{\rm{sg}}^{\rm{gr}}(R)
\end{align}
sending a matrix factorization (\ref{MF}) 
to the cokernel of $p^0$. 
\end{thm}
The cokernel of $p^0$ is easily checked to be annihilated by $W$, so 
it is $R$-module. 
Obviously the equivalence (\ref{Cok}) commutes with 
grade shift functors on both sides, so we use the 
same notation $\tau$ for the grade shift functor 
on $D_{\rm{sg}}^{\rm{gr}}(R)$. 

Let 
\begin{align}\label{mideal}
m=(x_1, \cdots, x_n) \subset R
\end{align}
be the maximal 
ideal and set the graded $R$-module 
$\mathbb{C}(j)$ to be $(R/m)(j)$.
The graded $R$-module $\mathbb{C}(j)$ determines 
an object in $D_{\rm{sg}}^{\rm{gr}}(R)$.  
The matrix factorization 
given by
\begin{align}\notag
\Cok^{-1}(\mathbb{C}(j)) \in \HMF(W)
\end{align}
plays an important role. 
By~\cite[Corollary~2.7]{Dyck}, 
it is given by the matrix factorization
of the form
\begin{align}\label{cok--}
\bigoplus_{k\ge 0}
\bigwedge^{2k+1}m/m^2 \otimes A(d&k+j) 
\stackrel{p^0}{\to}
\bigoplus_{k\ge 0}
 \bigwedge^{2k}m/m^2
 \otimes A(dk+j) \\
\notag
& \stackrel{p^1}{\to}
 \bigoplus_{k\ge 0}
\bigwedge^{2k+1}m/m^2 \otimes A(dk + k +j).
\end{align}
We omit the descriptions of the above morphisms 
$p^0$, $p^1$, 
as we will not use them. 
By an abuse of notation, we abbreviate
$\Cok^{-1}$ and denote by 
$\mathbb{C}(j)$ the matrix factorization 
given by (\ref{cok--}).

\subsection{Conjecture on the existence of Gepner type stability condition}
We are interested in constructing a natural 
stability condition on $\HMF(W)$.
We require that such a stability condition 
is preserved under the grade shift functor $\tau$ in some sense. 
This might be an analogues property for the classical 
$H$-Gieseker stability on a 
polarized variety $(X, H)$:
$H$-Gieseker semistable 
sheaves are preserved under $\otimes \oO_X(H)$.  
Because of the identity (\ref{taud}), the 
expected stability condition is not exactly preserved
by $\tau$ but the phases of semistable objects 
should be shifted by $2/d$. 
This is nothing but the Gepner type 
property with respect to $(\tau, 2/d)$. 

There is a natural construction of a central charge $Z_G$
which might give a desired stability condition, and 
already appeared in some articles~\cite{KST1},~\cite{Tak},~\cite{Wal}.   
For a graded matrix factorization (\ref{MF}), its image under $Z_G$ is 
symbolically defined by
\begin{align}\label{ZG}
\mathrm{str}(e^{2\pi \sqrt{-1}/d} \colon P^{\bullet} \to P^{\bullet}). 
\end{align}
Here $e^{2\pi \sqrt{-1}/d}$-action is given by 
the $\mathbb{C}^{\ast}$-action on $P^{\bullet}$
induced by the $\mathbb{Z}$-grading on each $P^i$,
and the $`\mathrm{str}'$ means the
super trace of $e^{2\pi \sqrt{-1}/d}$-action 
which respects the $\mathbb{Z}/2\mathbb{Z}$-grading 
of $P^{\bullet}=P^0 \oplus P^1$. 
More precisely, since $P^i$ are free $A$-modules of finite rank, 
they are written as 
\begin{align*}
P^i \cong 
\bigoplus_{j=1}^{m} A(n_{i, j}), 
\quad n_{i, j} \in \mathbb{Z}.
\end{align*}
Then (\ref{ZG}) is written as 
\begin{align*}
\sum_{j=1}^{m} \left( e^{2 n_{0, j} \pi \sqrt{-1}/d} - e^{2 n_{1, j} 
\pi \sqrt{-1} /d} \right). 
\end{align*}
\begin{exam}\label{exam:C0}
Let $\mathbb{C}(0)$ be the
matrix factorization given by (\ref{cok--}) for $j=0$. 
By the definition of $Z_G$, we have 
\begin{align*}
Z_G(\mathbb{C}(0))= &\sum_{i=1}^{n} e^{-2a_i \pi \sqrt{-1}/d} + \sum_{i_1< i_2 < i_3}e^{-2(a_{i_1}+a_{i_2}+a_{i_3}) \pi \sqrt{-1}/d} + \cdots \\
&
- 1- \sum_{i_1 < i_2} e^{-2(a_{i_1}+a_{i_2}) \pi \sqrt{-1}/d} - 
\sum_{i_1<i_2<i_3<i_4} \cdots  \\
&= -\prod_{j=1}^{n} \left( 1-e^{-2a_j \pi \sqrt{-1}/d} \right) \neq 0.  
\end{align*} 
\end{exam}

It is easy to check that (\ref{ZG}) descends to a 
group homomorphism 
\begin{align*}
Z_G \colon K(\HMF(W)) \to \mathbb{C}. 
\end{align*}
Indeed, $Z_G$ is one of the components of the 
Chern character map of $\HMF(W)$
constructed in~\cite{PoVa}. 
(cf.~Remark~\ref{rmk:Chern}.)
As we stated in the introduction, we propose the following 
conjecture: 
\begin{conj}\label{conj2:Gep}
There is a Gepner type stability 
condition 
\begin{align*}
\sigma_G=(Z_G, \{\pP_G(\phi)\}_{\phi \in \mathbb{R}})
\in \Stab(\mathrm{HMF}^{\rm{gr}}(W))
\end{align*}
with respect to $(\tau, 2/d)$, where 
$Z_G$ is given by (\ref{ZG}). 
\end{conj}
Note that the central charge $Z_G$ 
satisfies the condition (\ref{eq:gep})
with respect to $(\tau, 2/d)$
by the construction. The problem 
is to construct full subcategories 
$\pP_G(\phi) \subset \HMF(W)$ 
satisfying the desired property. 

\begin{rmk}\label{discrete}
The image of $Z_G$
is contained in $\mathbb{Z}[e^{2\pi \sqrt{-1}/d}]$, 
which may or may not be discrete depending
on $d$. For instance, it is discrete if $d=3, 4, 6$, 
but not so if $d=5$. 
\end{rmk}

\subsection{The case of $n=1$}\label{subsec:n=1}
As a toy example of Conjecture~\ref{conj2:Gep}, 
let us consider the case of $n=1$, i.e. 
\begin{align*}
W=x^d \in \mathbb{C}[x], \quad \deg x=a. 
\end{align*} 
The case of $a=1$ is worked out in~\cite{Tak}. 
In this case, the triangulated category 
$\HMF(W)$ is equivalent to the derived category of 
the path algebra of the Dynkin quiver of type $A_{d-1}$. 
As a result, we have a complete classification of 
indecomposable objects in $\HMF(W)$, given by
\begin{align*}
Q_{j, l} \cneq 
\left\{
A(j-l) \stackrel{x^{l}}{\to}
A(j) \stackrel{x^{d-l}}{\to} A(j-l+d) \right\} 
\end{align*}
for $1\le l \le d-1$ and $j\in \mathbb{Z}$. 
Note that $Q_{j, 1}$ coincides with $\mathbb{C}(j)$
given by (\ref{cok--}). 
We set $\phi(Q_{j, l}) \in \mathbb{Q}$ to
be
\begin{align*}
\phi(Q_{j, l}) \cneq -\frac{1}{2} - \frac{l}{d} + \frac{2j}{d}. 
\end{align*}
For $\phi \in \mathbb{R}$, 
we define $\pP_G(\phi) \subset \HMF(W)$
to be the subcategory consisting of 
direct sums of objects $Q_{j, l}$ with
$\phi(Q_{j, l})=\phi$. 
Then the pair 
$\sigma_G=(Z_{G}, \{\pP_G(\phi)\}_{\phi \in \mathbb{R}})$
is shown to be a desired stability condition in 
Conjecture~\ref{conj2:Gep} by~\cite{Tak}. 
The case of $a>1$ 
follows from the following lemma: 
\begin{lem}\label{coprime}
Let $A$ be the graded ring (\ref{def:A}), 
$a \in \mathbb{Z}_{\ge 1}$ the 
greatest common divisor of $(a_1, \cdots, a_n)$, 
and set $a_i'=a_i/a$. 
Let $A'$ be the graded ring defined by 
\begin{align*}
A'=\mathbb{C}[x_1', x_2', \cdots, x_n'], \quad 
\deg x_i' =a_i'. 
\end{align*}
For a homogeneous element $W \in A$
of degree $d$, we 
regard it as a homogeneous element $W' \in A'$
of degree $d'=d/a$
by the identification $x_i=x_i'$. 
Then if Conjecture~\ref{conj2:Gep} holds for $W'\in A'$, then 
it also holds for $W\in A$. 
\end{lem}
\begin{proof}
There is an obvious fully-faithful 
functor as triangulated categories
\begin{align*}
i \colon \HMF(W') \to \HMF(W)
\end{align*}
by multiplying $a$
to each grading of $A'$-modules
which appear
in the LHS. 
If $\tau'$ is the grade shift on $\HMF(W')$ and 
$Z_{G}'$ the central charge (\ref{ZG}) for $\HMF(W')$, 
we have 
\begin{align}\label{compare}
\tau^{a} \circ i= i \circ \tau', \quad 
Z_G \circ i= Z_{G}'. 
\end{align}
Also note that if there is a non-zero 
morphism of graded $A$-modules 
$A(m) \to A(n)$, 
then $m-n$ is divisible by $a$. 
This implies that
$\HMF(W)$ has the following orthogonal
decomposition:
\begin{align}\label{decom}
\langle i \HMF(W'), 
\tau i \HMF(W'), \cdots, \tau^{a-1} i \HMF(W')
 \rangle. 
\end{align}
Suppose that 
$(Z_G', \{\pP_G'(\phi)\}_{\phi \in \mathbb{R}})$
is a Gepner type stability condition on $\HMF(W')$ 
with respect to $(\tau', 2/d')$. 
We set $\pP_G(\phi)$ as follows: 
\begin{align*}
\pP_G(\phi)=
\left\{ \bigoplus_{j=0}^{a-1}
\tau^j(Q_j) :
Q_j \in i \pP_G' \left( \phi -\frac{2j}{d} \right)
\right\}.  
\end{align*}
By (\ref{compare}) and (\ref{decom}), it is easy 
to check that $(Z_G, \{ \pP_G(\phi)\}_{\phi \in \mathbb{R}})$
is a Gepner type stability condition on $\HMF(W)$
with respect to $(\tau, 2/d)$. 
\end{proof}

Combined with
the argument for $n=a_1=1$, we obtain the following corollary: 
\begin{cor}\label{cor:n=1}
Conjecture~\ref{conj2:Gep} is true if $n=1$. 
\end{cor}

\subsection{Remarks for the
 case of $n=2$}\label{subsec:rmk2}
We discuss Conjecture~\ref{conj2:Gep} 
in some more cases with small $n$. 
By Corollary~\ref{cor:n=1},
the next interesting case may be $n=2$.
In this case, the problem 
is trivial when $\varepsilon >0$. 
Indeed it is easy to check that 
\begin{align*}
\HMF(W)=\{0\}, \quad n=2, \ \varepsilon > 0
\end{align*}  
by using the equivalence (\ref{Cok}).
On the other hand, 
the triangulated category $\HMF(W)$ is non-trivial 
when $\varepsilon =0$. 
In this case, by applying the
coordinate change if necessary, 
we may assume that 
\begin{align*}
W=x_1 x_2 \in \mathbb{C}[x_1, x_2]
\end{align*} 
with $a_1$ and $a_2$ coprime by Lemma~\ref{coprime}. 
Similarly to the case of $n=1$, it turns out that there is 
only 
a finite number of indecomposable objects up to shift
in $\HMF(W)$. 
They consist of the objects
\begin{align*}
\mathbb{C}(j)[k], \quad 0\le j \le d-1, \quad
k\in \mathbb{Z}
\end{align*}
where $\mathbb{C}(j)$ is given by 
(\ref{cok--}). 
Furthermore each indecomposable objects are mutually 
orthogonal. The above fact can be easily checked, 
for instance using 
Orlov's theorem~\cite{Orsin}, 
as given in Example~\ref{exam:n=2}
below.  
A desired stability condition $\sigma_G$
in Conjecture~\ref{conj2:Gep} is constructed as follows: 
we first choose $\phi_0 \in \mathbb{R}$ so that 
$Z_G(\mathbb{C}(0)) \in \mathbb{R}_{>0} e^{2\pi i \phi_0}$
and set 
\begin{align*}
\phi(\mathbb{C}(j)[k])= \phi_0 + k + \frac{2j}{a_1+a_2}. 
\end{align*}
We define $\pP_G(\phi)$ to be the subcategory consisting of 
direct sums of indecomposable objects with $\phi(\ast)=\phi$. 
Then by the above argument, 
$\sigma=(Z_G, \{\pP_G(\phi)\}_{\phi \in \mathbb{R}})$
gives a desired stability condition. 
As a summary, we have the following: 
\begin{prop}\label{thm:n2e}
Conjecture~\ref{conj2:Gep} is true if $n=2$ and $\varepsilon \ge 0$. 
\end{prop}
The situation drastically changes 
when $n=2$ and $\varepsilon <0$. 
For instance, let us consider the case
\begin{align*}
W=x_1^d + x_2^d \in \mathbb{C}[x_1, x_2], \quad \deg x_i =1, \ d\ge 3.  
\end{align*}
Even in such a simple case, Conjecture~\ref{conj2:Gep}
seems to be not obvious, and it requires 
a deep understanding of the category $\HMF(W)$. 
The above case is treated in this paper when $d\le 4$. 
The $d=5$ case will be studied in ~\cite{TGep2}. 

\subsection{Remarks for the case of $n=3$}\label{subsec:rmk:n=3}
Suppose that $n=3$ and $\varepsilon >0$, the 
case studied by Kajiura-Saito-Takahashi~\cite{KST1}. 
In this case, $W$ is classified into the 
following ADE types~\cite{Saito}: 
\begin{align*}
W(x_1, x_2, x_3)=
\left\{ \begin{array}{ll}
x_1 x_2 + 
x_3^{l+1} & A_l \ (l \ge 1) \\
x_1^2 +x_2^2 x_3 + x_3^{l-1} & D_l \ (l\ge 4) \\
x_1^2 + x_2^3 + x_3^4 & E_6 \\
x_1^2 + x_2^3 + x_2 x_3^3 & E_7 \\
x_1^2 + x_2^3 + x_3^5 & E_8. 
\end{array}  \right. 
\end{align*}
Furthermore, the triangulated category 
$\HMF(W)$ is equivalent to 
the derived category of 
quiver representations of a Dynkin quiver of 
the corresponding ADE type. 
As a result, the category $\HMF(W)$ 
is shown to have only a finite number of indecomposable 
objects up to shift, which are completely 
classified. Similar to the case of $n=1$ 
(which is also interpreted as 
an $A_l$-case in the above
$n=3$, $\varepsilon >0$ list by 
Kn$\ddot{\rm{o}}$rrer periodicity~\cite{Knor})
they assign phases to classified indecomposable 
objects, and prove the following: 
\begin{thm}\emph{(\cite[Theorem~4.2]{KST1})}\label{thm:KST}
Conjecture~\ref{conj2:Gep} is true 
if $n=3$ and $\varepsilon >0$. 
\end{thm}
Conjecture~\ref{conj2:Gep} in the 
case of $n=3$ and $\varepsilon \le 0$ 
is not obvious, and a part of this case 
is treated later in this paper.

\subsection{Uniqueness of the central charge}\label{subsec:unique1}
It is a natural question whether 
the Gepner type property uniquely 
characterize $\sigma_G$ or not
in some sense. 
As for the central charge, this is true: 
$Z_G$ is characterized by the Gepner type 
property with respect to $(\tau, 2/d)$
up to a scalar multiplication.
Indeed we are only interested in 
central charges which factors 
through the Chern character map
\begin{align*}
\ch \colon \HMF(W) \to \mathrm{HH}_{0}(W). 
\end{align*}
The RHS is the Hochschild homology group of $\HMF(W)$, 
or more precisely of its dg enhancement. 
A general theory on Hochschild homology 
groups and Chern character maps on $\HMF(W)$ 
is available in~\cite{Dyck},~\cite{PoVa2},~\cite{PoVa}. 

Because $K(\HMF(W))$ is not finitely 
generated in general, it would be more natural 
to define the set of central charges 
on $\HMF(W)$
as the dual space $\mathrm{HH}_0(W)^{\vee}$, 
rather than the original 
one in Definition~\ref{defi:stab}. 
On the other hand, the 
autoequivalence $\tau$ on $\HMF(W)$ defines the
 linear isomorphism 
\begin{align*}
\tau_{\ast} \colon \mathrm{HH}_{\ast}(W) \stackrel{\cong}{\to}
\mathrm{HH}_{\ast}(W). 
\end{align*}
The above 
isomorphism induces the isomorphism 
$\tau_{\ast}^{\vee}$
on the space of central 
charges $\mathrm{HH}_0(W)^{\vee}$. 
The Gepner type property requires the central charge, regarded as an element in $\mathrm{HH}_{0}(W)^{\vee}$, 
to be an eigenvector with respect to $\tau_{\ast}^{\vee}$
with eigenvalue $e^{2\pi \sqrt{-1}/d}$. 
The lemma below shows that such an eigenspace is one dimensional. 
\begin{lem}\label{lem:Hoch}
The eigenspace of $\tau_{\ast}$-action on 
$\mathrm{HH}_{\ast}(W)$
with eigenvalue $e^{\pm 2\pi \sqrt{-1}/d}$
is one dimensional,
and contained in $\mathrm{HH}_0(W)$.  
\end{lem}
\begin{proof}
We consider the $G \cneq \mu_{d}$-action on 
$\mathbb{C}^n$, given by 
\begin{align*}
e^{2\pi \sqrt{-1}/d} \cdot (x_1, \cdots, x_n)=
(e^{2a_1\pi \sqrt{-1}/d} x_1, \cdots, e^{2a_n \pi \sqrt{-1}/d}x_n). 
\end{align*} 
By~\cite[Theorem~2.6.1 (i)]{PoVa2}, 
$\mathrm{HH}_{\ast}(W)$ is given by 
\begin{align}\label{H:decom}
\mathrm{HH}_{\ast}(W) \cong \bigoplus_{\gamma \in G}
\mathrm{H}(\mathbb{C}^n_{\gamma}, W_{\gamma})^{G}. 
\end{align}
Here $H(\mathbb{C}^n, W)$ is defined by 
\begin{align*}
H(\mathbb{C}^n, W) \cneq \left(
\mathbb{C}[x_1, \cdots, x_n]/( \partial_{x_1}W, \cdots, \partial_{x_n}W) \right) dx_1 \wedge \cdots \wedge dx_n
\end{align*}
and the space $\mathrm{H}(\mathbb{C}^n_{\gamma}, W_{\gamma})$ is given
by applying the above construction for 
\begin{align*}
\mathbb{C}_{\gamma}^n \cneq \{ x \in \mathbb{C}^n \colon 
\gamma(x)=x \}, \quad 
W_{\gamma} \cneq W|_{\mathbb{C}^n_{\gamma}}. 
\end{align*} 
Since $\tau^{\times d}=[2]$ on $\HMF(W)$, we have 
$\tau_{\ast}^{\times d}=\id$ on $\mathrm{HH}_{\ast}(W)$. 
This implies that $\tau_{\ast}$ generates the $\mathbb{Z}/d\mathbb{Z}$-action 
on $\mathrm{HH}_{\ast}(W)$. 
By~\cite[Theorem~2.6.1 (ii)]{PoVa2}, 
the decomposition (\ref{H:decom})
coincides with the character decomposition of $\mathrm{HH}_{\ast}(W)$
with respect to the above $\mathbb{Z}/d\mathbb{Z}$-action. 
Therefore, noting that $\mathbb{C}^n_{e^{\pm 2\pi \sqrt{-1}/d}}=\{0\}$, 
 the desired eigenspace is one dimensional by
\begin{align}\notag
\mathrm{H}(\mathbb{C}^n_{e^{\pm 2\pi \sqrt{-1}/d}}, 
W_{e^{\pm 2\pi \sqrt{-1}/d}}) \cong \mathbb{C}. 
\end{align}
By the grading of $\mathrm{HH}_{\ast}(W)$
given in~\cite[Theorem~2.6.1]{PoVa2}, 
the above eigenspace 
is contained in $\mathrm{HH}_0(W)$. 
\end{proof}
\begin{rmk}\label{rmk:Chern}
For $E \in \HMF(W)$, the 
$\gamma= e^{2\pi \sqrt{-1}/d}$-component of 
$\ch(E)$ in the decomposition (\ref{H:decom})
coincides with the central charge $Z_G$
defined by (\ref{ZG}). 
(cf.~\cite[Theorem~3.3.3]{PoVa}.)
In particular $Z_G$ is given by an 
element in $\mathrm{HH}_0(W)^{\vee}$, 
which gives a basis of the eigenspace 
of $\tau_{\ast}^{\vee}$-action with 
eigenvalue $e^{2\pi \sqrt{-1}/d}$. 
\end{rmk}
\begin{rmk}\label{rmk:action}
Obviously the set of Gepner type stability conditions 
with respect to $(\tau, 2/d)$ is preserved under the natural 
right action of 
$\mathbb{C}$ on $\Stab(\HMF(W))$. 
By Bridgeland's deformation result~\cite[Theorem~7.1]{Brs1}, 
Lemma~\ref{lem:Hoch} implies that the set of such stability conditions 
forms a discrete subset in the quotient space
$\Stab(\HMF(W))/\mathbb{C}$.  
\end{rmk}

\subsection{Uniqueness of $\sigma_G$}\label{subsec:unique2}
There are some cases in which not only $Z_G$ but also 
$\sigma_G$ is characterized by the Gepner type property. 
At least this is the case 
when all the indecomposable 
objects should become semistable. 
We first note the following lemma, 
which is an important property of 
Gepner type stability conditions: 
\begin{lem}\label{lem:important}
Suppose that 
$W \in A$ satisfies Conjecture~\ref{conj2:Gep}
and $\sigma_G=(Z_{G}, \{\pP_G(\phi)\}_{\phi \in \mathbb{R}})$
is a Gepner type stability condition with respect to 
$(\tau, 2/d)$. 
For $\phi_i \in \mathbb{R}$
with $i=1, 2$
and $k\in \mathbb{Z}$, 
suppose that the
following inequality holds: 
\begin{align}\label{ineq:phases}
\phi_1 > \phi_2 +n-k-2 -\frac{2\varepsilon}{d}. 
\end{align}
Then for any $F_i \in \pP_G(\phi_i)$, 
we have $\Hom^k(F_2, F_1)=0$. 
\end{lem}
\begin{proof}
By the Serre functor given by (\ref{Serre}), 
we have the isomorphism
\begin{align}\label{vanish:Serre}
\Hom(F_2, F_1[k]) &\cong \Hom(F_1, \tau^{-\varepsilon} (F_2)[n-k-2])^{\vee}.
\end{align}
By the Gepner type property with respect to $(\tau, 2/d)$, we have 
\begin{align*}
\tau^{-\varepsilon} (F_2)[n-k-2] \in \pP_G \left( 
\phi_2 +n-k-2 -\frac{2\varepsilon}{d} \right). 
\end{align*}
By the inequality (\ref{ineq:phases}), the phase of 
$F_1$ is bigger than that of the above object, hence 
the RHS of (\ref{vanish:Serre}) vanishes. 
\end{proof}
Using the above lemma, 
we show the following proposition: 
\begin{prop}\label{prop:unique}
In the same situation of Lemma~\ref{lem:important}, 
suppose that the following inequality holds: 
\begin{align}\label{ineq:ep}
(n-3)d\le 2\varepsilon. 
\end{align}
Then all the other stability conditions
satisfying the conditions in Conjecture~\ref{conj2:Gep}
are obtained 
as $[2m]_{\ast}\sigma_G$ for $m\in \mathbb{Z}$. 
\end{prop}
\begin{proof}
Let us take $F_i \in \pP_G(\phi_i)$, 
$i=1, 2$ 
 with 
$\phi_1>\phi_2$. 
Then Lemma~\ref{lem:important} and the assumption (\ref{ineq:ep})
show that 
$\Hom^1(F_2, F_1)=0$. 
This implies that any object in $\HMF(W)$
whose Harder-Narasimhan factors
are $F_1$, $F_2$ decomposes into the direct sum of 
$F_1$ and $F_2$. 
Repeating this argument, any object $E\in \HMF(W)$
decomposes into the direct sum of $\sigma_G$-semistable 
objects. In particular any non-zero indecomposable object 
$E \in \HMF(W)$
is $\sigma_G$-semistable, whose phase is denoted by $\phi_E$. 

Let us fix a non-zero indecomposable 
object $M \in \HMF(W)$. By the result 
of~\cite[Theorem~5.16]{BFK},
the objects $\tau^i(M), 0\le i\le d-1$
generate the triangulated category $\HMF(W)$.  
Therefore for any non-zero indecomposable object
$E \in \HMF(W)$, there is $0\le i\le d-1$ and 
$j\in \mathbb{Z}$
such that $\Hom(E, \tau^i(M)[j]) \neq 0$. 
This implies that 
\begin{align*}
\phi_E \le j + \phi_M + \frac{2i}{d}. 
\end{align*}
Also by the Serre functor (\ref{Serre}), 
we have $\Hom(\tau^i(M)[j], \sS_W(E)) \neq 0$, 
which implies that 
\begin{align*}
j+ \phi_M -n +2 +\frac{2i}{d} + \frac{2\varepsilon}{d}\le \phi_E. 
\end{align*}
Combined with the assumption (\ref{ineq:ep}), we obtain 
\begin{align}\label{ineq:phase}
\phi_E \in \left[ j+ \phi_M + \frac{2i}{d} -1, 
j+ \phi_M + \frac{2i}{d}  \right]. 
\end{align}
Now suppose that $\sigma_G'$ is another 
stability condition which satisfies the 
condition in Conjecture~\ref{conj2:Gep}. 
Then there is $m\in \mathbb{Z}$ such that 
the phase of $M$ with respect to $[-2m]_{\ast}\sigma_G'$ coincides
with $\phi_M$. 
If $\phi_E'$ is the phase of the indecomposable 
object $E$ with respect to $[-2m]_{\ast}\sigma_G'$, 
then $\phi_E'$ is also contained in the RHS of (\ref{ineq:phase}). 
Since both of the central charges of $\sigma_G$ and $[-2m]_{\ast}\sigma_G'$
are the same $Z_G$, it follows that $\phi_E'=\phi_E$. 
Therefore $\sigma_G'=[2m]_{\ast}\sigma_G$ follows. 
\end{proof}
The proof of Proposition~\ref{prop:unique}
immediately implies the following: 

\begin{cor}
In the same situation of 
Proposition~\ref{prop:unique}, there is a function $\phi(\ast)$
from the set of indecomposable objects in $\HMF(W)$
to real numbers such that 
$\pP_G(\phi)$ consists of direct sums of 
indecomposable objects $E$ with $\phi(E)=\phi$. 
\end{cor}

\begin{rmk}
The inequality (\ref{ineq:ep}) is satisfied in the 
cases of Corollary~\ref{cor:n=1}, 
Theorem~\ref{thm:n2e} and Theorem~\ref{thm:KST}. 
In the list of Table~1, it is satisfied 
except $(n, \varepsilon)=(4, 0)$ and $(3, -1)$. 
\end{rmk}

\begin{rmk}
If we believe Conjecture~\ref{conj2:Gep}, 
the proof of Proposition~\ref{prop:unique} 
predicts that $Z_G(E) \neq 0$ for any 
non-zero indecomposable object $E \in \HMF(W)$
as long as the inequality (\ref{ineq:ep}) is satisfied. 
This seems to be not an obvious property
of graded matrix factorizations. 
\end{rmk}

\section{T-structures on triangulated categories of graded
matrix factorizations}\label{sec:T-st}
In this section, we construct and study 
the 
hearts of bounded 
t-structures on $\HMF(W)$, via 
Orlov's theorem relating $\HMF(W)$ with the 
derived category of coherent sheaves on $(W=0)$. 
In what follows, we use the same notation 
in the previous section.

\subsection{Orlov's theorem}
In~\cite{Orsin}, Orlov proves 
his famous theorem relating 
the triangulated category $\HMF(W)$ with 
the derived category of coherent sheaves on the
Deligne-Mumford stack
\begin{align}\label{DMW}
X \cneq (W=0) \subset \mathbb{P}(a_1, \cdots, a_n). 
\end{align}
Using the notation in Subsection~\ref{subsec:relation}, 
Orlov's theorem is stated in the following way: 
\begin{thm}\emph{(\cite[Theorem~2.5]{Orsin})}\label{thm:Orlov}

(i) If $\varepsilon >0$, there is a
fully faithful functor 
\begin{align*}
\Phi_i \colon \HMF(W) \hookrightarrow 
D^b \Coh(X)
\end{align*}
such that we have the semiorthogonal decomposition
\begin{align*}
D^b \Coh(X)= \langle \oO_X(-i-\varepsilon +1), \cdots, 
\oO_X(-i), \Phi_i \HMF(W)  \rangle. 
\end{align*}

(ii) If $\varepsilon \le 0$, there is a fully 
faithful functor
\begin{align*}
\Psi_{i} \colon D^b \Coh(X) \hookrightarrow \HMF(W)
\end{align*}
such that we have the semiorthogonal decomposition
\begin{align}\label{sod}
\HMF(W)= \langle \mathbb{C}(-i-\varepsilon), 
\cdots, \mathbb{C}(-i+1), 
\Psi_i D^b \Coh(X) \rangle. 
\end{align}
In particular $\Psi_i$ is an equivalence if $\varepsilon =0$. 
\end{thm}
In this paper we deal with the case of 
$\varepsilon \le 0$, so we only explain the
construction of $\Psi_i$. It is 
the composition of the following functors: 
\begin{align}\label{def:Psi}
\Psi_{i} \colon 
D^b \Coh(X) \stackrel{\dR \omega_i}{\to}
D^b (\grr R) \stackrel{\pi}{\to}
 D_{\rm{sg}}^{\rm{gr}}(R)
\stackrel{\rm{Cok}^{-1}}{\to} \HMF(W)
\end{align}
where $\rm{Cok}^{-1}$ is the inverse of 
(\ref{Cok}), $\pi$ is the natural projection and 
$\dR \omega_i$ is defined by 
\begin{align}\label{def:omega}
\dR \omega_i(E) \cneq
 \bigoplus_{j\ge i} 
\dR \Hom_{\oO_X}(\oO_X, E(j)). 
\end{align}
The functor $\Psi_i$ is not compatible with 
grade shift functors. 
If $\varepsilon=0$, i.e. $X$ is a 
Calabi-Yau stack, then their 
difference is described by a Seidel-Thomas 
twist functor~\cite{ST} on $D^b \Coh(X)$
\begin{align*}
\mathrm{ST}_{E}(\ast)
 \cneq \mathrm{Cone} \left( \dR \Hom_{\oO_X}(E, \ast) \otimes E
\to \ast  \right)
\end{align*}
for a spherical object $E \in D^b \Coh(X)$, 
e.g. a line bundle. 
The following result is suggested by 
Kontsevich and proved in~\cite{BFK}:

\begin{prop}\emph{(\cite[Proposition~5.8]{BFK})}\label{prop:grade}
If $\varepsilon =0$, the
following diagram commutes: 
\begin{align*}
\xymatrix{
D^b \Coh(X) \ar[r]^{\Psi_i} \ar[d]_{F_i} & \HMF(W) \ar[d]^{\tau} \\
D^b \Coh(X) \ar[r]^{\Psi_i} & \HMF(W). 
}
\end{align*}
Here $F_i \cneq \ST_{\oO_X(-i+1)} \circ \otimes \oO_X(1)$. 
\end{prop}
\begin{rmk}
In~\cite[Proposition~5.8]{BFK}, a
comparison result similar to Proposition~\ref{prop:grade}
is obtained also for $\varepsilon \neq 0$. 
We only mention the case of $\varepsilon=0$ since we only 
use the result in this case. 
\end{rmk}

\begin{exam}\label{exam:n=2}
Let us consider the case of $n=2$ and $\varepsilon=0$. 
In the same situation as in Subsection~\ref{subsec:rmk2}, 
we have 
\begin{align*}
X \cong \left[ \mathrm{pt}/ \mathbb{Z}_{a_1}
  \right] \coprod  \left[ \mathrm{pt} / \mathbb{Z}_{a_2}  \right]. 
\end{align*}
Here $\mathbb{Z}_{a_i} \cneq \mathbb{Z}/a_i \mathbb{Z}$ acts on 
the smooth one point $\mathrm{pt}$
 trivially, and $[\ast / \ast ]$ means the quotient stack. 
Therefore we have the orthogonal decomposition
\begin{align*}
D^b \Coh(X) = \langle V_1^{0}, \cdots, V_{1}^{a_1-1}, \ 
V_2^{0}, \cdots, V_2^{a_2-1} \rangle
\end{align*}
for one dimensional 
$\mathbb{Z}_{a_i}$-representations $V_i^{j}$
with 
weight $j$. 
By Proposition~\ref{prop:grade}, the equivalence 
$\Psi_1$ identifies $\tau$ on $\HMF(W)$
with $F_1=\ST_{\oO_X} \circ \otimes 
\oO_X(1)$. 
The equivalence $F_1$ transforms $V_i^{j}$ 
in the following way: 
\begin{align*}
&V_1^{0} \mapsto V_1^{1} \mapsto \cdots \mapsto 
V_1^{a_1-1} \\
&\mapsto V_2^{0}[1] \mapsto 
\cdots \mapsto V_2^{a_2-1}[1] \mapsto V_1^{0}[2]. 
\end{align*}
In particular $\HMF(W)$ has the description 
stated in Subsection~\ref{subsec:rmk2}. 
\end{exam}

\subsection{Construction of t-structures}
In this subsection, we construct the hearts of 
bounded t-structures on $\HMF(W)$
when $\varepsilon \le 0$. 
We introduce the following notation:
for a triangulated category $\dD$ and 
a set of objects $\sS \subset \dD$, 
we denote by $\langle \sS \rangle_{\rm{ex}}$
the extension closure of $\sS$, i.e. 
the smallest extension-closed subcategory 
of $\dD$ which contains objects
in $\sS$. 
The constructions of our hearts 
are
 based on 
the semiorthogonal decomposition (\ref{sod})
and the following well-known fact: 
\begin{lem}\label{lem:t-st}
Let $\dD$ be a triangulated category and 
\begin{align*}
\dD= \langle \dD_{N}, \cdots, \dD_2, \dD_1 \rangle
\end{align*}
a semiorthogonal decomposition. Suppose that 
$\cC_i \subset \dD_i$ are hearts of bounded 
t-structures satisfying 
$\Hom^{\le 0}(\cC_j, \cC_i)=0$ for $j>i$. 
Then there is 
a bounded t-structure on $\dD$
whose heart $\cC$ is given by 
$\langle \cC_i : 1\le i \le N \rangle_{\rm{ex}}$. 
\end{lem}
\begin{proof}
The result is obviously reduced to the case of $N=2$, 
which is proved in~\cite[Lemma~2.1]{InPo}. 
\end{proof}

\begin{rmk}\label{rmk:filt}
In the situation of Lemma~\ref{lem:t-st}, 
any object $E \in \cC$ admits a filtration 
\begin{align*}
0=E_0 \subset E_1 \subset \cdots \subset E_N=E
\end{align*}
such that $E_i/E_{i-1}$ is an object in $\cC_i$. 
\end{rmk}

\begin{rmk}\label{rmk:ext}
If the abelian category $\cC_i$ is generated by 
an exceptional object $F_i \in \cC_i$, 
then the heart $\cC$ is 
the extension closure $\langle F_N, \cdots, F_1 \rangle_{\rm{ex}}$. 
In this case,  
the sequence $(F_N, \cdots, F_2, F_1)$
is called an ext-exceptional collection. 
\end{rmk}

We have the following proposition: 
\begin{prop}\label{prop:t-st}
Suppose that $\varepsilon \le 0$. 
For each $i\in \mathbb{Z}$, there is 
a bounded t-structure on $\HMF(W)$ whose heart 
$\aA_i$ is given by 
\begin{align*}
\aA_i= \langle \mathbb{C}(-i -\varepsilon), \cdots, 
\mathbb{C}(-i+1), \Psi_i \Coh(X) \rangle_{\rm{ex}}. 
\end{align*}
\end{prop}
\begin{proof}
By following Orlov's argument
in~\cite[Theorem~2.5]{Orsin}, we see the following: 
there is an admissible subcategory 
$\tT_i \subset D^b(\grr R)$ with a semiorthogonal 
decomposition
\begin{align}\label{sod2}
\tT_i= \langle \mathbb{C}(-i-\varepsilon), \cdots, \mathbb{C}(-i+1), 
\dR \omega_i D^b \Coh(X)  \rangle 
\end{align}
and an equivalence 
\begin{align*}
\tT_i \stackrel{\sim}{\to} \HMF(W)
\end{align*}
which identifies the semiorthogonal decomposition
 (\ref{sod2}) with the RHS of (\ref{sod}). 
Here $\dR \omega_i$ is the functor defined by (\ref{def:omega}). 
Therefore by Lemma~\ref{lem:t-st}, 
it is enough to show that
\begin{align}\label{vanish1}
&\Hom^{\le 0}_{\grr R}(\mathbb{C}(j), \mathbb{C}(j'))=0 \\ 
\label{vanish2}
&\Hom^{\le 0}_{\grr R}(\mathbb{C}(j), \dR \omega_i \Coh(X))=0
\end{align}
for $j, j' \in [-i+1, -i-\varepsilon]$ with $j'<j$. 
The assertion (\ref{vanish1})
is obvious since $\mathbb{C}(j)$ is
a simple object in the heart $\grr R \subset D^b(\grr R)$. 
As for the assertion (\ref{vanish2}), 
since we have 
$\dR \omega_i (F) \in D^{\ge 0}(\grr R)$
for $F \in \Coh(X)$, 
it follows that  
\begin{align}\label{vanish3}
\Hom^{\le 0}_{\grr R}(\mathbb{C}(j), \dR \omega_i (F)) \cong
\Hom^{\le 0}_{\grr R}(\mathbb{C}(j), \omega_i(F)).
\end{align} 
Here $\omega_i(F) \in \grr R$ 
is the zero-th cohomology of $\dR \omega_i(F)$. 
Since $\omega_i(F) \in \grr R$ is concentrated on degree $\ge i$ parts, 
and $\mathbb{C}(j)$ is on degree $<i$ parts, 
the vector space (\ref{vanish3}) vanishes. 
\end{proof}

In what follows, 
we always assume that $\varepsilon \le 0$. 
We only focus on 
the case $i=1$ in the above proposition:
\begin{defi} 
Suppose that $\varepsilon \le 0$. 
We define 
$\aA_W \cneq \aA_1$
and 
$\Psi \cneq \Psi_1$, i.e.
\begin{align}\label{sodA}
\aA_W= \langle \mathbb{C}(-1-\varepsilon), \cdots, 
\mathbb{C}(0), \Psi \Coh(X) \rangle_{\rm{ex}}. 
\end{align}
\end{defi}

\subsection{Description of certain objects in $\aA_W$}
This subsection is devoted to 
investigate some objects in $\aA_W$, 
which will be used later. 
By definition, we call a
closed point $x \in X$ 
\textit{stacky} if 
the stabilizer group at $x$ is non-trivial. 
Let us describe 
$\Psi(\oO_x)$ for a non-stacky 
point $x\in X$. 
Note that $X$ is a closed substack
\begin{align*}
X \subset \left[(\mathbb{C}^n \setminus \{0\} )/\mathbb{C}^{\ast} \right]
\end{align*}
where $\mathbb{C}^{\ast}$ acts on $\mathbb{C}^n$ via 
weight $(a_1, \cdots, a_n)$. 
Hence $x \in X$ is represented by 
a point $(p_1, \cdots, p_n) \in \mathbb{C}^n$. 
We define the graded $R$-module $M(x)$ to be
\begin{align}\label{def:gM}
M(x) \cneq \bigoplus_{j\ge 1} \mathbb{C}e_j
\end{align}
where $e_j$ is concentrated on degree $j$, 
and the action of $x_i$ sends $e_j$
to $p_i e_{j+a_i}$. 
Obviously if $x \in X$ is non-stacky, 
then $\dR \omega_1(\oO_x)$ is 
a graded $R$-module and isomorphic to 
$M(x)$. 
The object $\Psi(\oO_x)$ is obtained 
by applying the inverse of 
(\ref{Cok}) to $M(x)$, after regarding 
it as an object in 
$D_{\rm{sg}}^{\rm{gr}}(R)$. 
Using the above description, we have the following lemma: 
\begin{lem}\label{lem:tau:ox}
For any non-stacky point $x\in X$, we have 
the exact sequence in $\aA_W$
\begin{align}\label{tau:ox}
0 \to \Psi(\oO_x) \to \tau \Psi(\oO_x) \to \mathbb{C}(0) \to 0. 
\end{align}
\end{lem}
\begin{proof}
The result obviously follows from the 
following exact sequence as graded $R$-modules
\begin{align*}
0 \to M(x) \to M(x)(1) \to \mathbb{C}(0) \to 0. 
\end{align*}
\end{proof}

Next let us consider the 
following object
\begin{align*}
\mathbb{C}(-\varepsilon) \in \HMF(W). 
\end{align*}
The above object 
is described in terms of 
the generators of $\aA_W$. 
We have the following lemma: 
\begin{lem}\label{lem:filt:A}
We have $\mathbb{C}(-\varepsilon)[-1] \in \aA_W$. 
Furthermore there is a filtration in $\aA_W$
\begin{align*}
0 \subset E_{-1} \subset E_0 \subset \cdots \subset E_{-1-\varepsilon}=\mathbb{C}(-\varepsilon)[-1]
\end{align*}
such that the following holds
for $0\le i \le -1-\varepsilon$
\begin{align*}
E_{-1} \cong \Psi(\omega_X), \quad 
E_{i}/E_{i-1} \cong \mathbb{C}(i) \otimes R_{-i-\varepsilon}.
\end{align*}
\end{lem}
\begin{proof}
Let $m \subset R$ be the maximal ideal (\ref{mideal}).  
We have the exact sequence in $\grr R$
\begin{align*}
0 \to m(-\varepsilon) \to R(-\varepsilon) \to \mathbb{C}(-\varepsilon) \to 0 
\end{align*}
which implies $\mathbb{C}(-\varepsilon)[-1] \cong m(-\varepsilon)$
in $D_{\rm{sg}}^{\rm{gr}}(R)$. 
Let $m_{\ge i} \subset m$ be the ideal generated by 
monomials with degree greater than or equal to $i$.  
We have the following filtration in $\grr R$
\begin{align}\label{filt:m}
m_{\ge 1-\varepsilon}\subset m_{\ge -\varepsilon}
\subset \cdots \subset m_{\ge 2} \subset m_{\ge 1}
=m
\end{align} 
such that the following holds: 
\begin{align*}
\left( m_{\ge j}/m_{\ge j+1} \right)(-\varepsilon) \cong
\mathbb{C}(-j-\varepsilon) \otimes R_{j}, \quad
1\le j\le -\varepsilon. 
\end{align*}
Therefore it is enough to show that 
$m_{\ge 1-\varepsilon}(-\varepsilon)$ is isomorphic to 
$\dR \omega_1(\omega_X)$ in $\grr R$. 
Since 
$\omega_X \cong \oO_X(-\varepsilon)$ and 
$H^k(X, \omega_X(j)) \cong 0$
for $k\neq 0$ and $j\ge 1$, we have 
\begin{align*}
\dR \omega_1(\omega_X) \cong 
\bigoplus_{j\ge 1} H^0(X, \oO_X(-\varepsilon +j)). 
\end{align*}
Obviously the RHS is isomorphic to $m_{\ge 1-\varepsilon}(-\varepsilon)$
as a graded $R$-module. 
\end{proof}

The above lemma can be applied 
to do some computations
on the left adjoint of $\Psi$, 
denoted by 
\begin{align*}
\Psi^{L} \colon \HMF(W) \to D^b \Coh(X).
\end{align*}
We have the following lemma:
\begin{lem}\label{lem:adj}
The object $\Psi^{L}(\mathbb{C}(0))$ is isomorphic to $\oO_X[1]$. 
\end{lem}
\begin{proof}
Let $\Psi^R$ is the right adjoint of $\Psi$. 
Note that $\Psi^L$ and $\Psi^R$ are related by 
\begin{align*}
\Psi^L=\sS_X^{-1} \circ \Psi^R \circ \sS_W
\end{align*} 
where 
$\sS_X=\otimes \omega_X[n-2]$ is the Serre 
functor of $D^b \Coh(X)$. 
By (\ref{Serre}), we have 
\begin{align*}
\Psi^L(\mathbb{C}(0)) &\cong 
\sS_X^{-1} \circ \Psi^R (\mathbb{C}(-\varepsilon))[n-2] \\
&\cong \sS_X^{-1}(\omega_X[1])[n-2] \\
&\cong \oO_X[1]. 
\end{align*}
Here the second isomorphism follows from Lemma~\ref{lem:filt:A}. 
\end{proof}

\begin{rmk}\label{rmk:ext:uni}
By the above lemma, it follows that 
\begin{align*}
\Hom_{\HMF(W)}^1(\mathbb{C}(0), \Psi(\oO_x))
\cong \Hom_X(\oO_X, \oO_x)
\end{align*}
which is one dimensional for $x\in X$. 
Since $\tau \Psi(\oO_x)$ is indecomposable, 
the exact sequence (\ref{tau:ox}) is a unique 
non-trivial extension. 
\end{rmk}

\begin{rmk}\label{rmk:ext1}
For $F \in \Coh(X)$, suppose that 
$\dR \omega_1(F)$ is a graded $R$-module. 
Then Lemma~\ref{lem:adj} and the argument in
Proposition~\ref{prop:t-st} imply
\begin{align*}
\Ext_{\grr R}^1(\mathbb{C}(0), \dR \omega_1(F))
\cong H^0(X, F). 
\end{align*}
For $u\in H^0(X, F)$, 
we have the corresponding extension in $\grr R$
\begin{align*}
0 \to \dR \omega_1(F) \to M_u \to \mathbb{C}(0) \to 0.
\end{align*}
The graded $R$-module $M_u$ is 
described in the following way:
as a graded $\mathbb{C}$-vector space, 
it is the direct sum $\mathbb{C}(0) \oplus \dR \omega_1(F)$, 
and the action of $x_i \in R$ sends 
$1 \in \mathbb{C}(0)$ to $u \cdot x_i \in H^0(X, F(1))$. 
 \end{rmk}

\subsection{Description of $\aA_W$ via quiver representations}
In this subsection, we assume that 
$n=2$ and describe the 
heart $\aA_W$ in terms of certain
quiver representations.
In this case,  
$X$ is a smooth zero dimensional Deligne-Mumford stack, 
and $\Coh(X)$ is generated by mutually orthogonal 
exceptional objects.
Therefore by (\ref{sodA}), the abelian category $\aA_W$ is 
the extension closure of an ext-exceptional collection. 
We first compute other Hom groups between these 
exceptional objects, and then 
describe $\aA_W$ via the ext-quivers with relations. 

 For some technical reason, we 
assume that 
$X$ does not contain stacky points, 
so it consists of finite number of 
smooth points. 
Then, after applying the 
coordinate change if necessary, 
$W$ is written as 
\begin{align*}
W=x_1 W_1 + x_2 W_2
\end{align*}
for some homogeneous elements 
$W_i \in A$ 
such that $x_1$ (resp.~$x_2$)
does not divide $W_2$ (resp.~$W_1$). 
The heart $\aA_W$ is described in
the following way:  
\begin{align*}
\aA_W=\langle \mathbb{C}(d-a_1-a_2-1), \cdots, \mathbb{C}(0), 
\Psi(\oO_x) : x\in X \rangle_{\rm{ex}}.
\end{align*}
Below we calculate the $\Hom$ groups between 
the above generators. 
\begin{lem}\label{lem:compute}
For $0< j < d-a_1 -a_2$, 
we have the following:
\begin{align}\label{com:ext}
\Hom^i_{\HMF(W)}(\mathbb{C}(j), \mathbb{C}(0)) 
 \cong \left\{ \begin{array}{cl}
R_{j}^{\vee}, & (i, j)=(1, a_1), (1, a_2) \\
R_{0}^{\vee}, & (i, j)=(2, a_1+a_2) \\
0, & \mbox{ otherwise. }
\end{array} \right. 
\end{align}
Moreover the natural map
\begin{align}\label{natural}
&\Hom^2(\mathbb{C}(a_1 +a_2), \mathbb{C}(0))^{\vee} \to \\
\notag
&\quad \bigoplus_{j, j' \in \{a_1, a_2\}}
\Hom^1(\mathbb{C}(j), \mathbb{C}(0))^{\vee}
\otimes \Hom^1(\mathbb{C}(j+j'), \mathbb{C}(j'))^{\vee}
\end{align}
sends $1 \in R_{0}$ to 
$x_1 \otimes x_2 - x_2 \otimes x_1$
under the isomorphism (\ref{com:ext}). 
\end{lem} 
\begin{proof}
By the same argument of Proposition~\ref{prop:t-st}, 
we have
\begin{align}\label{HMF=ext}
\Hom_{\HMF(W)}^i(\mathbb{C}(j), \mathbb{C}(0))
\cong \Ext_{\grr R}^i(\mathbb{C}(j), \mathbb{C}(0)). 
\end{align}
Then (\ref{com:ext}) is easily obtained 
by computing the RHS of (\ref{HMF=ext}) using 
the resolution: 
\begin{align}\label{resol}
\cdots \to
R(-a_1-d) \oplus & R(-a_2-d) \stackrel{h'}{\to} 
 R(-d) \oplus R(-a_1-a_2) \\
\notag
 &\stackrel{h}{\to} R(-a_1) \oplus R(-a_2)
\stackrel{(x_1, x_2)}{\to} R  \to \mathbb{C}(0) \to 0. 
\end{align}
Here $h$ and $h'$ given by matrices
\begin{align}\label{matrix:h}
h= \left( \begin{array}{cc}
W_1 & -x_2 \\
W_2 & x_1
\end{array} \right), \quad
h'= \left( \begin{array}{cc}
x_1 & x_2 \\
-W_2 & W_1
\end{array} \right).
\end{align}
Next we consider the map (\ref{natural}). 
We write $W_1$, $W_2$ as 
\begin{align*}
W_1=x_1 W_{11} + x_2 W_{12}, \quad
W_2= x_1 W_{21} + x_2 W_{21}
\end{align*}
for homogeneous elements $W_{k, l} \in A$. 
Let $x_k^{\vee} \in R_{a_k}^{\vee}$ be the dual basis of 
$x_k \in R_{a_k}$, 
and we regard them as elements of the RHS 
of (\ref{HMF=ext})
for $(i, j)=(1, a_k)$.
Then $x_1^{\vee}$ is represented by 
the morphism of complexes
\begin{align*}
\xymatrix{
R(a_1-d) \oplus R(-a_2) \ar[r]^{h(a_1)} \ar[d]_{g_1}  &  R(0) 
\oplus R(a_1-a_2) \ar[r]^{}\ar[d]_{\pi_1} & R(a_1) \ar[d] \\
R(-a_1) \oplus R(-a_2) \ar[r]^{(x_1, x_2)} & R(0) \ar[r] & 0. 
}
\end{align*}
Similarly $x_2^{\vee}$ is represented by 
\begin{align*}
\xymatrix{
R(a_2-d) \oplus R(-a_1) \ar[r]^{h(a_2)} \ar[d]_{g_2}  &  R(a_2 -a_1) 
\oplus R(0) \ar[r]^{}\ar[d]_{\pi_2} & R(a_2) \ar[d] \\
R(-a_1) \oplus R(-a_2) \ar[r]^{(x_1, x_2)} & R(0) \ar[r] & 0. 
}
\end{align*}
Here $\pi_i$ are projections onto the $i$-th factor, and 
$g_i$ are given by matrices
\begin{align*}
g_1 = \left(\begin{array}{cc}
W_{11} & 0 \\
W_{12} & -1
\end{array}\right), \quad 
g_2=\left(\begin{array}{cc}
W_{21} & 1 \\
W_{22} & 0
\end{array}\right). 
\end{align*}
The image of $x_2^{\vee} \otimes x_{1}^{\vee}$ 
by the dual of (\ref{natural})
is computed by composing the above morphisms of complexes.
 By 
restricting the map 
$\pi_2 \circ g_1(a_2)$
to the second component of 
$R(a_1 +a_2-d) \oplus R(0)$, 
we see that $x_2^{\vee} \otimes x_1^{\vee}$
is mapped to $-1^{\vee}$, 
where $1^{\vee}$ is the dual basis of $1 \in R_{0}$. 
Similarly $x_1^{\vee} \otimes x_2^{\vee}$
is mapped to $1^{\vee}$.  
By dualizing, we obtain the result. 
\end{proof}

\begin{rmk}\label{rmk:ext2}
By (\ref{com:ext}) and (\ref{HMF=ext}), an element 
$u \in R_{a_i}^{\vee}$ 
determines the extension in $\grr R$
\begin{align*}
0 \to \mathbb{C}(0) \to M_u \to \mathbb{C}(a_i) \to 0. 
\end{align*}
The graded $R$-module $M_u$
is
isomorphic to $\mathbb{C}(0) \oplus \mathbb{C}(a_i)$
as a graded $\mathbb{C}$-vector space, and 
the action of $x_i$ is given by 
sending $1 \in \mathbb{C}(0)$ to $u(x_i) \in \mathbb{C}(a_i)$. 
\end{rmk}

Next we compute the Hom
 groups between $\mathbb{C}(j)$ and 
$\Psi(\oO_x)$ for closed points $x\in X$.
Let $M(x)$ be the graded 
$R$-module defined by (\ref{def:gM}). 
We have the following lemma: 
\begin{lem}\label{lem:compute2}
Suppose that 
$0\le j < d-a_1 -a_2$, and 
$x\in X$ is represented by $(p_1, p_2) \in \mathbb{C}^2$. 
Then we have 
\begin{align}\label{comp2}
\Hom_{\HMF(W)}^{i}(\mathbb{C}(j), \Psi(\oO_x))
\cong \left\{ \begin{array}{cl}
\mathbb{C} u_j, & i=1, j\in [0, a_2) \\
\mathbb{C} v_j, & i=2, j\in [a_1, a_1 +a_2) \\
0, & \mbox{ otherwise. }
\end{array}  \right. 
\end{align} 
Here $u_j$ and $v_j$ are regarded as elements
\begin{align*}
&u_j=p_1 e_{-j+a_1} \oplus p_2 e_{-j +a_2} 
\in M(x)_{-j+a_1} \oplus M(x)_{-j+a_2} \\
&v_j=v e_{-j +d} \oplus e_{-j + a_1 + a_2}
\in M(x)_{-j +d} \oplus M(x)_{-j + a_1 + a_2} 
\end{align*}
where $v \cneq W_2(p_1, p_2)/p_1 = -W_1(p_1, p_2)/p_2$. 
If $j\in [a_1, a_1 +a_2)$, 
the natural map
\begin{align}\label{natural2}
&\Hom^2(\mathbb{C}(j), \Psi(\oO_x))^{\vee} \to \\
\notag
&\bigoplus_{j' \in \{j-a_1, j-a_2 \}}
\Hom^1(\mathbb{C}(j), \mathbb{C}(j'))^{\vee}
\otimes \Hom(\mathbb{C}(j'), \Psi(\oO_x))^{\vee}
\end{align}
sends the dual basis $v_j^{\vee}$
to $p_2 x_1 \otimes u_{j-a_1}^{\vee}
- p_1 x_2 \otimes u_{j-a_2}^{\vee}$
under the isomorphisms (\ref{com:ext}), (\ref{comp2}). 
(Here we set $u_{j-a_1}^{\vee}=0$ if 
$j-a_1 \ge a_1$. )
\end{lem}
\begin{proof}
Similarly to the proof of Lemma~\ref{lem:compute}, 
we have the isomorphism 
\begin{align}\label{isom:x}
\Hom_{\HMF(W)}^i(\mathbb{C}(j), \Psi(\oO_x))
\cong 
\Ext_{\grr R}^i(\mathbb{C}(j), M(x)).
\end{align}
Applying $\Hom_{\grr R}(\ast, M(x)(-j))$ to the 
exact sequence (\ref{resol}), 
the RHS of (\ref{isom:x})
 is computed by the $i$-th 
cohomology 
group of the following complex
\begin{align}\label{com:M}
0 \to M(x)_{-j} &\stackrel{(p_1, p_2)}{\lr} 
M(x)_{-j+a_1} \oplus M(x)_{-j+a_2} \\ &\quad 
\stackrel{^{t}h(p_1, p_2)}{\lr} 
\notag
M(x)_{-j+d} \oplus M(x)_{-j+a_1 +a_2} \\ &
\qquad \stackrel{^{t}h'(p_1, p_2)}{\lr}
\notag
M(x)_{-j+d+a_1} \oplus M(x)_{-j+d+a_2} \to \cdots. 
\end{align}
Here $h(p_1, p_2)$, $h'(p_1, p_2)$ are the
substitution of $(x_1, x_2)=(p_1, p_2)$ to the matrices
(\ref{matrix:h}). 
Then (\ref{comp2}) easily follows by noting that 
every non-zero maps in the complex (\ref{com:M})
are rank one. The image of $v_j^{\vee}$
by the map (\ref{natural2})
is computed similarly to (\ref{natural}), so we omit the detail. 
\end{proof}

The above computations enable us to describe $\aA_W$
in terms of quiver representations with 
relations. Here we only discuss the case of 
$a_1=a_2=1$. 
Recall that, given a set of objects $(F_N, \cdots, F_2, F_1)$,
the ext-quiver $\qQ(F_{\bullet})$ is defined as follows: 
the set of vertices is 
\begin{align*}
\{1, 2, \cdots, N\}
\end{align*} and 
the number of edges from $j$ to $j'$ is 
the dimension of $\Ext^1(F_j, F_{j'})$, which 
we identify with a basis of $\Ext^1(F_j, F_{j'})^{\vee}$. 
The following lemma may be well-known, but we include the 
proof later in Subsection~\ref{subsec:lem:equiv},
because of a lack of a reference. 
\begin{lem}\label{lem:equivalence}
Let $\dD$ be a triangulated category with finite dimensional 
Hom spaces, generated by 
an ext-exceptional collection $(F_N, \cdots, F_{2}, F_1)$. 
Let $\aA$ be the heart of a bounded t-structure on $\dD$
given by the extension closure of all $F_i$ for $1\le i\le N$. 
Suppose that 
there is a partition 
\begin{align*}
\{1, \cdots, N\} =P_1 \sqcup \cdots \sqcup P_l, \ 
j'>j \mbox{ if } j\in P_k, j' \in P_{k'}, k'>k. 
\end{align*}
such 
that, by setting $\widehat{F}_k \cneq \oplus_{j\in P_k}
F_j$, the following condition holds: 
\begin{align}\label{lem:cond}
\Ext^i(\widehat{F}_{k'}, \widehat{F}_{k})=0 \mbox{ unless }
(i, k'-k)=(1, 1), (2, 2). 
\end{align}
Then $\aA$ is equivalent to the abelian 
category of $\qQ(F_{\bullet})$-representations 
with relations generated by the images of the following 
natural maps for all $1\le k\le l$: 
\begin{align}\label{relation}
\Ext^2(\widehat{F}_{k+2}, \widehat{F}_k)^{\vee} 
\to \Ext^1(\widehat{F}_{k+2}, \widehat{F}_{k+1})^{\vee}
\otimes \Ext^1(\widehat{F}_{k+1}, \widehat{F}_k)^{\vee}. 
\end{align}
\end{lem}

The following corollary directly follows from 
Lemma~\ref{lem:compute}, Lemma~\ref{lem:compute2}
and Lemma~\ref{lem:equivalence}. 

\begin{cor}\label{cor:quiver}
Suppose that $a_1=a_2=1$
and we write 
\begin{align*}
X=\{p^{(i)}=(p_1^{(i)}, p_2^{(i)}) \in \mathbb{P}^1 : 1 \le i\le d \}.\end{align*} 
Then $\aA_W$ is equivalent to the 
category of representations of the quiver of the form
\begin{align}\label{quiver}
    \xygraph{!~:{@{=}|@{>}} !~-{@{>}}
{\bullet}*+!D{v^{(d-3)}} 
: _{X_2^{(d-3)}}^{X_1^{(d-3)}}[rr]{\bullet}*+!D{}
:[r] \cdots :[r]
    {\bullet}*+!D{v^{(1)}} : _{X_2^{(1)}}^{X_1^{(1)}}[rr]{\bullet}*+!D{v^{(0)}}
        (-^{\pi^{(d)}}[ru]{\bullet}*+!D{w^{(d)}}, -^{\pi^{(j)}}[r]{\vdots},
        -_{\pi^{(1)}}[rd]{\bullet}*+!D{w^{(1)}})
    }
\end{align}
with relations given by 
\begin{align*}
X_2^{(i-1)} X_1^{(i)}=X_1^{(i-1)}X_2^{(i)}, \quad
p_2^{(j)}\pi^{(j)}X_1^{(1)} = p_1^{(j)}\pi^{(j)} X_2^{(1)}
\end{align*}
for all $2\le i\le d-3$ and $1\le j\le d$. 
The vertex $v^{(i)}$ corresponds to 
$\mathbb{C}(i)$ and $w^{(j)}$ corresponds to 
$\Psi(\oO_{p^{(j)}})$. 
\end{cor}

By investigating the filtration (\ref{filt:m}),
we are able to describe 
$\mathbb{C}(-\varepsilon)[-1]$ 
in terms of a 
representation of a quiver 
(\ref{quiver}). 
The following corollary 
is a straightforward adaptation of 
Corollary~\ref{cor:quiver}, 
Remark~\ref{rmk:ext1} and Remark~\ref{rmk:ext2}: 
\begin{cor}\label{cor:quiver2}
In the situation of Corollary~\ref{cor:quiver}, 
the object $\mathbb{C}(-\varepsilon)[-1]$
in $\aA_W$ is 
the representation of the quiver (\ref{quiver}) given 
as follows: 
\begin{align}\label{quiver2}
    \xygraph{!~:{@{=}|@{>}} !~-{@{>}}
{\bullet}*+!D{R_1} : _{x_2}^{x_1}[rr]{\bullet}*+!D{R_2}
:[r] \cdots :[r]
    {\bullet}*+!D{R_{d-3}}
: _{x_2}^{x_1}[rr]{\bullet}*+!D{R_{d-2}}
        (-^{\pi^{(d)}}[ru]{\bullet}*+!D{\mathbb{C}}, -[r]{\vdots},
        -_{\pi^{(1)}}[rd]{\bullet}*+!D{\mathbb{C}})
    }
\end{align}
Here $\pi^{(j)} \colon R_{d-2} \to \mathbb{C}$
is the evaluation at $p^{(j)}=(p_1^{(j)}, p_2^{(j)})$. 
\end{cor}

\subsection{Description of $\aA_W$ via coherent systems}
In this subsection, we assume that 
$\varepsilon=-1$ and 
describe the heart $\aA_W$
in terms of coherent systems on $X$.  
Let us recall the definition of coherent systems. 
\begin{defi}
A coherent system on a Deligne-Mumford stack $X$
is data
\begin{align*}
V\otimes \oO_X \stackrel{s}{\to} F
\end{align*}
where $V$ is a finite dimensional $\mathbb{C}$-vector 
space, $F \in \Coh(X)$ and 
$s$ is a morphism in $\Coh(X)$. 
\end{defi}
The category of coherent systems on $X$
is denoted by $\mathrm{Syst}(X)$. 
The set of morphisms is given by the 
commutative diagrams in $\Coh(X)$
\begin{align*}
\xymatrix{
V\otimes \oO_X \ar[r]^{s} \ar[d] & F \ar[d] \\
V'\otimes \oO_X \ar[r]^{s'} & F'. 
}
\end{align*}
Obviously $\mathrm{Syst}(X)$ is an abelian 
category. We have the following proposition:
\begin{prop}\label{prop:Syst}
Suppose that $\varepsilon=-1$. 
Then we have an equivalence of abelian categories
\begin{align*}
\Theta \colon 
\mathrm{Syst}(X) 
\stackrel{\sim}{\to} \aA_W. 
\end{align*}
\end{prop}
\begin{proof}
Let us take a coherent system 
$(V\otimes \oO_X \stackrel{s}{\to} F)$
on $X$.  
By Lemma~\ref{lem:adj},
the morphism $s$ is regarded as an element
\begin{align*}
s' \in \Hom_{\HMF(W)}(V \otimes \mathbb{C}(0), \Psi(F)[1]). 
\end{align*} 
The cone of $s'$ determines an object in $\aA_W$. 
The
correspondence
\begin{align*}
\Theta \colon 
(F, s) \mapsto \Cone(s')
\end{align*}
is a functor
from $\mathrm{Syst}(X)$
to $\aA_W$
because, 
as in the proof of Proposition~\ref{prop:t-st}, 
we have the vanishing 
$\Hom(\mathbb{C}(0), \Psi \Coh(X))=0$.

Conversely, 
let us take an object $E \in \aA_W$. 
There is an exact sequence in $\aA_W$
\begin{align}\label{ex:sys}
0 \to \Psi(F) \to E \to V\otimes \mathbb{C}(0) \to 0
\end{align}
for a finite dimensional vector space $V$ and 
$F \in \Coh(X)$. 
By Lemma~\ref{lem:adj}, 
the extension class $\xi$
of (\ref{ex:sys}) 
is regarded as an element 
\begin{align*}
\xi' \in \Hom(V \otimes \oO_X, F).
\end{align*}
The pair $(F, \xi')$ determines an object in 
$\mathrm{Syst}(X)$. 
The correspondence 
\begin{align*}
\Theta' \colon E \mapsto (F, \xi')
\end{align*}
is a functor from $\aA_W$
to $\mathrm{Syst}(X)$
since 
$\langle \mathbb{C}(0), \Psi D^b \Coh(X) \rangle$
is a semiorthogonal decomposition of $\HMF(W)$. 

Obviously we have 
\begin{align*}
\Theta' \circ \Theta \cong \id_{\mathrm{Syst}(X)}, \quad 
\Theta \circ \Theta' \cong \id_{\aA_W}
\end{align*}
hence $\Theta$ is an equivalence.
\end{proof}

Combined with Lemma~\ref{lem:filt:A}
and Remark~\ref{rmk:ext1}, we immediately obtain the 
following corollary:
\begin{cor}\label{cor:canonical}
Suppose that $\varepsilon=-1$. 
Then the object $\mathbb{C}(1)[-1] \in \aA_W$ is 
given by 
\begin{align*}
\mathbb{C}(1)[-1] \cong \Theta \left( H^0(X, \oO_X(1)) \otimes \oO_X
\stackrel{s}{\to} \oO_X(1) \right). 
\end{align*}
Here $s$ is the canonical evaluation morphism. 
\end{cor}

\subsection{Proof of Lemma~\ref{lem:equivalence}}\label{subsec:lem:equiv}
Finally in this section, we give a proof of Lemma~\ref{lem:equivalence}. 
The proof is straightforward, and probably well-known. 
We 
recommend the readers to skip this subsection 
at the first reading. 
\begin{proof}
We denote by $I$ the set of relations generated 
by the images of (\ref{relation}). 
Let $\mathrm{Rep}(\qQ(F_{\bullet}), I)$ be the 
category of $\qQ(F_{\bullet})$-representations 
with relation $I$. 
We divide the proof into three steps. 
\begin{step}
\end{step}
We construct the 
functor 
\begin{align*}
\Phi \colon \aA \to \mathrm{Rep}(\qQ(F_{\bullet}), I)
\end{align*}
in the following way: 
for an object $E \in \aA$, it admits a filtration
\begin{align*}
0=E_0 \subset E_1 \subset \cdots \subset E_l=E
\end{align*}
such that $E_k/E_{k-1}$ is written as 
\begin{align*}
E_k/E_{k-1} \cong \bigoplus_{j\in P_{k}} F_j \otimes V_j
\end{align*} 
for finite dimensional vector spaces $V_j$. 
By the exact sequence
\begin{align}\label{eq:E}
0 \to E_k/E_{k-1} \to E_{k+1}/E_{k-1} \to E_{k+1}/E_k \to 0 
\end{align}
we obtain the linear maps 
\begin{align}\label{linear:VV}
\phi_{j' j} \colon 
V_{j'} \otimes \Ext^1(F_{j'}, F_j)^{\vee}
\to V_j 
\end{align} 
for $j'\in P_{k+1}$, $j \in P_k$, 
which defines the $\qQ(F_{\bullet})$-representation
$\Phi(E)$. 
In order to show that 
the representation $\Phi(E)$ satisfies the relation $I$, 
consider the composition of extension classes of (\ref{eq:E}): 
\begin{align}\label{compose}
\bigoplus_{j'' \in P_{k+2}}
F_{j''} \otimes V_{j''} \to
\bigoplus_{j' \in P_{k+1}} F_{j'} \otimes V_{j'}[1] 
\to \bigoplus_{j\in P_k} F_j \otimes V_j[2]. 
\end{align}
The above composition must vanish since it coincides
with the composition
\begin{align*}
E_{k+2}/E_{k+1} \to E_{k+1}/E_{k-1}[1] \to E_{k+1}/E_{k}[1] 
\to E_{k}/E_{k-1}[2]
\end{align*} 
where the left morphism is the extension class of
\begin{align*}
0 \to E_{k+1}/E_{k-1} \to E_{k+2}/E_{k-1} \to E_{k+2}/E_{k+1} \to 0.
\end{align*} 
By applying $\Hom(F_{j''}, \ast)$ 
for $j'' \in P_{k+2}$
 to (\ref{compose}), 
taking the adjunction
and $(j'', j)$-component for $j\in P_k$, we see that the map 
\begin{align}\notag
V_{j''} \otimes \bigoplus_{j' \in P_{k+1}}
 \left( \Ext^1(F_{j''}, F_{j'})^{\vee} \otimes 
\Ext^1(F_{j'}, F_j)^{\vee} \right)  \to V_j
\end{align}
given by the sum of the 
composition
\begin{align}\label{comp:II}
\sum_{j' \in P_{k+1}} 
\phi_{j' j} \circ 
\phi_{j'' j'}
\end{align}
is zero
on $V_{j''} \otimes I_{j'', j}$, where $I_{j'', j}$ 
is the image of (\ref{relation})
restricted to $\Ext^2(F_{j''}. F_{j})^{\vee}$-component. 
This implies that $\Phi(E)$ satisfies the relation $I$, hence 
it is an object in $\mathrm{Rep}(\qQ(F_{\bullet}), I)$.  

\begin{step}
\end{step}
The correspondence $E \mapsto \Phi(E)$ obviously 
determines a fully faithful functor from $\aA$
to $\mathrm{Rep}(\qQ(F_{\bullet}), I)$, since 
$(F_N, \cdots, F_2, F_1)$ is an ext-exceptional collection. 
It remains to show that $\Phi$ is essentially 
surjective. Let us take an object 
\begin{align*}
W \in \mathrm{Rep}(\qQ(F_{\bullet}), I).
\end{align*} 
It consists of 
finite dimensional vector spaces $V_j$ for $1\le j\le N$ and 
linear maps (\ref{linear:VV}) 
whose composition (\ref{comp:II}) is zero on $V_{j''} \otimes I_{j'', j}$
for $(j'', j) \in P_{k+2} \times P_k$. 
We need to show the existence of 
$E \in \aA$ so that $\Phi(E) \cong W$. 

By the induction on $l$, we may assume that the assertion 
holds for $l-1$. 
We set full subcategories 
$\aA_{k} \subset \dD$ as follows: 
\begin{align*}
\aA_{k} \cneq \langle F_j \colon j \in P_{k'}, 
1\le k' \le k \rangle_{\rm{ex}}. 
\end{align*}
Let $\qQ(F_{\bullet}')$ be the ext-quiver 
for $\aA_{l-1}$ and define the relation $I'$
by restricting $I$ 
to $\qQ(F_{\bullet}')$. 
The category $\mathrm{Rep}(\qQ(F_{\bullet}'), I')$ 
is naturally considered as a subcategory of 
$\mathrm{Rep}(\qQ(F_{\bullet}), I)$,
and there is an exact sequence 
\begin{align}\label{ex:W}
0 \to W' \to W \to W_{l} \to 0
\end{align}
where $W' \in \mathrm{Rep}(\qQ(F_{\bullet}'), I')$
and $W_l$ is written as 
\begin{align*}
W_l \cong \bigoplus_{j\in P_l} e_j \otimes V_j.
\end{align*}
Here $e_j$ is the simple object
in $\mathrm{Rep}(\qQ(F_{\bullet}), I)$ corresponding 
to the vertex $j$. 

By the assumption of the induction, there 
is an object
$E' \in \aA_{l-1}$
 such 
that $\Phi(E') \cong W'$. 
By the exact sequence (\ref{ex:W}), 
it is enough to show that the map 
\begin{align}\label{ind:ex}
\alpha \colon 
\Ext^1_{\aA}(F_{j}, E') 
\to \Ext^1_{\mathrm{Rep}}(e_j, W')
\end{align}
induced by $\Phi$
is an isomorphism
for all $j\in P_l$. 
Here we have written $\mathrm{Rep}(\qQ(F_{\bullet}), I)$
just as $\mathrm{Rep}$ for simplicity. 
\begin{step}
\end{step}
We show that the morphism 
(\ref{ind:ex}) is an isomorphism. 
Let us consider the exact sequence in $\aA$
\begin{align}\label{ex:fin}
&0 \to E'' \to E' \to E_{l-1} \to 0
\end{align}
with $E'' \in \aA_{l-2}$ and $E_{l-1}$ is written as
\begin{align*}
E_{l-1} \cong \bigoplus_{j' \in P_{l-1}} F_{j'} \otimes V_{j'}. 
\end{align*}
By the condition (\ref{lem:cond}), 
for $j\in P_l$, 
we see that $\Ext^1_{\aA}(F_j, E'')=0$ and 
there is a natural isomorphism
\begin{align*}
\Ext^2_{\dD}(F_j, E'') \stackrel{\cong}{\to} 
\bigoplus_{j''\in P_{l-2}}
\Ext^2_{\dD}(F_j, F_{j''} ) \otimes V_{j''}.
\end{align*}
Therefore 
applying $\Hom(F_j, \ast)$
 to (\ref{ex:fin}), we obtain 
the exact sequence 
\begin{align}\label{exex1}
0 \to \Ext^1_{\aA}(F_j, E') \to 
\bigoplus_{j' \in P_{l-1}}&\Ext^1_{\aA}(F_j,  F_{j'} ) \otimes V_{j'}  \\
\notag
&\stackrel{\beta}\to \bigoplus_{j'' \in P_{l-2}}
\Ext^2_{\dD}(F_j, F_{j''}) \otimes V_{j''}. 
\end{align}
On the other hand, there is an
exact sequence in $\mathrm{Rep}(\qQ(F_{\bullet}), I)$
\begin{align*}
0 \to W'' \to W' \to \bigoplus_{j' \in P_{l-1}}
e_{j'} \otimes V_{j'} \to 0
\end{align*} 
such that $W''\cong \Phi(E'')$. 
Since $\Ext^1_{\rm{Rep}}(e_{a}, e_{a'})=0$
unless $a\in P_{k}$, $a' \in P_{k'}$ with $k-k'=1$,  
we have $\Ext^1_{\rm{Rep}}(e_{j}, W'')=0$
for any $j\in P_l$. 
Therefore we obtain the commutative diagram
\begin{align}\notag
\xymatrix{
0 \ar[r] & \Ext^1_{\aA}(F_j, E') \ar[r] \ar[d]^{\alpha} & 
\qquad
\bigoplus_{j' \in P_{l-1}} \Ext^1_{\aA}(F_j, F_{j'}) \otimes V_{j'} \ar[d]^{\gamma} \\
0 \ar[r] & \Ext^1_{\rm{Rep}}(e_j, W') \ar[r]^{\delta} & 
\qquad 
\bigoplus_{j' \in P_{l-1}} \Ext^1_{\rm{Rep}}(e_j, e_{j'}) \otimes V_{j'}. 
}
\end{align}
Here $\gamma$ is an isomorphism induced by $\Phi$.  
By the above diagram, it follows that $\alpha$ is injective.  
On the other hand, the composition 
\begin{align*}
\beta \circ \gamma^{-1} \circ \delta \colon 
\Ext^1_{\rm{Rep}}(e_j, W') \to \bigoplus_{j'' \in P_{l-2}}
 \Ext^2_{\dD}(F_j, F_{j''}) \otimes V_{j''}
\end{align*}
vanishes, since any object 
given by an extension class in the LHS 
satisfies the relation $I$. 
Therefore $\alpha$ is surjective, hence an isomorphism. 
\end{proof}

\section{Construction of Gepner type stability conditions}\label{sec:const}
In this section, we propose a general recipe 
on a construction of a desired Gepner type 
stability condition. 
We first compute the central charge $Z_G$ in terms of 
generators of $\aA_W$, and try to 
construct $\sigma_G$
via tilting of $\aA_W$. 
In what follows we assume that 
the stack $X$ in (\ref{DMW}) is a smooth 
projective variety, i.e. $X$ does not contain stacky points. 
As in the previous section, we denote 
by $\Psi \cneq \Psi_1$ the 
Orlov's fully faithful functor
from $D^b \Coh(X)$ to $\HMF(W)$, 
which is an equivalence if $\varepsilon=0$. 
\subsection{Computation of the central charge ($\varepsilon=0$ case)}
In this subsection, we explain 
how to compute $Z_G$ in the case of $\varepsilon=0$. 
If $\varepsilon=0$, $X$ is a Calabi-Yau manifold 
of dimension $n-2$, 
and $\aA_W$ is equivalent to 
$\Coh(X)$
via $\Psi$. 
 Let us consider the group homomorphism 
given by 
\begin{align}\label{ghomx}
Z_G \circ \Psi \colon K(X) \to \mathbb{C}. 
\end{align}
The above group homomorphism 
is described in terms of Chern characters 
on $K(X)$. Indeed, a fundamental theory on 
Hochschild homology groups 
implies that $\Psi$
induces the isomorphism
$\Psi_{\ast} \colon 
\mathrm{HH}_0(X) \stackrel{\cong}{\to}
\mathrm{HH}_0(W)$
such that the following diagram commutes: 
 (cf.~\cite[Section~1]{PoVa})
\begin{align*}
 \xymatrix{
D^b \Coh(X) \ar[r]^{\Psi} \ar[d]_{\ch} & \HMF(W) \ar[d]^{\ch} \\
\mathrm{HH}_0(X) \ar[r]^{\Psi_{\ast}} & \mathrm{HH}_0(W). 
}
\end{align*}
We also have the Hochschild-Kostant-Rosenberg
isomorphism 
\begin{align*}
\mathrm{HH}_0(X) \stackrel{\cong}{\to}
\mathrm{H\Omega}_0(X) \cneq \bigoplus_{j=0}^{n-2}
H^j(X, \Omega_X^j)
\end{align*}
such that its composition with 
$\ch \colon D^b \Coh(X) \to \mathrm{HH}_0(X)$
coincides with the classical Chern character map~\cite[Theorem~4.5]{Cal2}. 
Since our central charge $Z_G$ factors through the 
Chern character map on $\HMF(W)$ 
(cf.~Remark~\ref{rmk:Chern})
and the Poincar\'e pairing on $\mathrm{H\Omega}_0(X)$ is perfect, 
the group homomorphism (\ref{ghomx})
is written as
\begin{align*}
E \mapsto \sum_{j=0}^{n-2} \int_X \alpha_j \cdot \ch_j(E)
\end{align*}
for some $\alpha_j \in H^{n-2-j, n-2-j}(X)$. 
Here by an abuse of notation, we also denote by 
$\ch(E) \in \mathrm{H\Omega}_0(X)$ the classical Chern character
of $E \in K(X)$, and by $\ch_j(E)$ its $H^{j, j}(X)$-component.  

On the other hand, let us consider the 
autoequivalence $F$ of $D^b \Coh(X)$
defined by 
\begin{align*}
F \cneq \ST_{\oO_X} \circ \otimes \oO_X(1).
\end{align*}
By Proposition~\ref{prop:grade}, the 
above autoequivalence 
corresponds to the grade shift functor 
$\tau$ on $\HMF(W)$
via the equivalence $\Psi$. 
By the Riemann-Roch theorem, 
the autoequivalence $F$ acts on $\ch(E)$ for $E \in D^b \Coh(X)$
in the following way
\begin{align*}
F_{\ast} \colon 
\ch(E) \mapsto e^{H} \ch(E) -\left(\int_{X} e^{H}\ch(E) \td_X \right) \cdot 1. 
\end{align*}
Here $H$ is the first Chern class of $\oO_X(1)$. 
The above action naturally extends to the linear 
isomorphism on $\mathrm{H\Omega}_0(X)$, given by the 
composition of the matrices
\begin{align}\notag
M\cneq  
\left( \begin{array}{cccc}
1-t_{n-2} & -t_{n-3} & \cdots & -t_{0} \\
0 & 1 & \cdots & 0 \\
\vdots & \vdots & \ddots & 0 \\
0 & 0 & \cdots & 1
\end{array} \right)
\left( \begin{array}{cccc}
1 & 0& \cdots & 0 \\
H & 1 & \cdots & 0 \\
\vdots & \vdots & \ddots & 0 \\
\frac{H^{n-2}}{(n-2)!} & \frac{H^{n-3}}{(n-3)!} & \cdots & 1
\end{array} \right). 
\end{align}
Here $t_j$ is the $H^{j, j}(X)$-component of 
$\mathrm{td}_X$, and we
regard an element in $\mathrm{H\Omega}_0(X)$ as a column vector.  
The Gepner type property of the central 
charge $Z_G$ is translated into the 
following linear equation on $\alpha_i$: 
\begin{align}\label{linear:eq}
(\alpha_0, \cdots, \alpha_{n-2}) \cdot M =
e^{2\pi \sqrt{-1}/d} \cdot (\alpha_0, \cdots, \alpha_{n-2}). 
\end{align}
By Lemma~\ref{lem:Hoch}, the solution space (\ref{linear:eq})
must be one dimensional, so it determines 
the group homomorphism (\ref{ghomx}) uniquely 
up to a scalar multiplication.  

In practice, it is more convenient to work with 
a smaller subspace in $\mathrm{H\Omega}_0(X)$. 
Let $\mathbb{C} \langle H \rangle$ be the 
subspace in $\mathrm{H\Omega}_0(X)$ defined by 
\begin{align*}
\mathbb{C} \langle H \rangle 
\cneq \bigoplus_{j=0}^{n-2} \mathbb{C} H^{j}. 
\end{align*}
We have the following lemma: 
\begin{lem}
The solution space of (\ref{linear:eq})
is contained in $\mathbb{C} \langle H \rangle$. 
\end{lem}
\begin{proof}
Let $\mathbb{C}\langle H \rangle^{\perp}$
be the orthogonal complement of 
$\mathbb{C} \langle H \rangle$ in $\mathrm{H\Omega}_0(X)$
with respect to the Poincar\'e paring. We have the direct
sum decomposition
\begin{align*}
\mathrm{H\Omega}_0(X) = \mathbb{C} \langle H \rangle
\oplus \mathbb{C}\langle H \rangle^{\perp}
\end{align*}
and $F_{\ast}$ preserves the 
above direct summands. 
Since $\td_X \in \mathbb{C} \langle H \rangle$, 
$F_{\ast}$ acts on $\mathbb{C} \langle H \rangle^{\perp}$
via multiplication by $e^{H}$, 
which is unipotent. Hence all the eigenvectors
of the action $F_{\ast}$
on $\mathbb{C}\langle H \rangle^{\perp}$ have eigenvalue $1$, 
which implies that the solution space of (\ref{linear:eq}) 
is contained in 
$\mathbb{C} \langle H \rangle$.  
\end{proof}
By the above lemma, it is enough to 
solve the equation (\ref{linear:eq}) 
for $\alpha_j \in \mathbb{C}H^{n-2-j}$. 
The ambiguity of the scalar multiplication 
is fixed by the following lemma:
\begin{lem}\label{lem:fix}
We have the equality
\begin{align}\label{int:alpha}
\int_X \alpha_0 = \prod_{j=1}^{n} \left(1-e^{-2a_j \pi \sqrt{-1}/d} \right). 
\end{align}
\end{lem}
\begin{proof}
By Lemma~\ref{lem:adj}, the object 
$\Psi^{L}(\mathbb{C}(0))$ is isomorphic to 
$\oO_X[1]$. Since $\varepsilon=0$, the functor $\Psi$ is 
an equivalence, 
hence $\Psi^{L}=\Psi^{-1}$. It follows that 
$\Psi(\oO_X)$
is isomorphic to $\mathbb{C}(0)[-1]$. 
Then the equality (\ref{int:alpha}) follows by 
applying the homomorphism (\ref{ghomx}) to $\oO_X$, and 
using the computation in Example~\ref{exam:C0}. 
\end{proof}

Now the $\alpha_j \in \mathbb{C} H^{n-2-j}$ are
uniquely determined by the equation 
(\ref{linear:eq}) and the normalization (\ref{int:alpha}). 
However for our purpose, it is more convenient to 
consider a different normalization of $Z_G \circ \Psi$. 
Namely we write $Z_G \circ \Psi$ 
as a multiple of some non-zero complex number 
and a central charge on $\HMF(W)$ whose image 
of $\Psi(\oO_x)$ is $-1$.  
This is possible by the following lemma: 
\begin{lem}\label{lem:constant}
For any $x\in X$, we have $Z_G(\Psi(\oO_x))=-C_W$
where $C_W$ is given by 
\begin{align}\label{CW}
C_W \cneq -(1-e^{2\pi \sqrt{-1}/d})^{-1}
\prod_{j=1}^{n} \left(1-e^{-2a_j \pi \sqrt{-1}/d} \right)
\end{align}
which satisfies
\begin{align}\label{thetaW}
C_W \in \mathbb{R}_{>0} e^{\sqrt{-1} \pi \theta_W}, \quad
\theta_W= 
\frac{1}{2}(n-1) -
\frac{1}{d} \left( \sum_{j=1}^{n} a_j +1  \right).
\end{align}
\end{lem}
\begin{proof}
The equality (\ref{CW}) follows from 
Example~\ref{exam:C0} and Lemma~\ref{lem:tau:ox}. 
The property (\ref{thetaW}) follows from 
\begin{align*}
1-e^{-2\pi \sqrt{-1}\theta} = 2 \sin \pi \theta
\cdot e^{\left(\frac{1}{2}-\theta \right)\pi \sqrt{-1}}.
\end{align*}
\end{proof}
We summarize the result in this subsection 
as follows: 
\begin{prop}
Suppose that $\varepsilon=0$. Then for $E \in D^b \Coh(X)$, 
the central charge 
$Z_G(\Psi(E))$ is written as 
\begin{align}\label{Zdag}
Z_G(\Psi(E))=C_W \sum_{j=0}^{n-2} \int_{X} \alpha_j^{\dag} 
\cdot \ch_j(E)
\end{align}
where $(\alpha_0^{\dag}, \cdots, \alpha_{n-2}^{\dag})$
satisfies $\alpha_j^{\dag} \in \mathbb{C}H^{n-2-j}$, 
and it is the unique solution of the linear equation
\begin{align}\label{sol:unique}
(\alpha_0^{\dag}, \cdots, \alpha_{n-2}^{\dag}) \cdot M=
e^{2\pi \sqrt{-1}/d} \cdot (\alpha_0^{\dag}, \cdots, 
\alpha_{n-2}^{\dag}), \quad \alpha_{n-2}^{\dag}=-1.
\end{align}
\end{prop}

\subsection{Computation of the central charge ($\varepsilon<0$ case)}
\label{subsec:central2}
The purpose of this subsection is to 
reduce the computation of $Z_G$
in the case $\varepsilon <0$ to that of 
the case $\varepsilon=0$. 
The strategy is to embed $X$ into 
$n-2-\varepsilon$-dimensional Calabi-Yau manifold $\widehat{X}$
and relate $Z_G$ with the central charge on $\widehat{X}$. 
We set
\begin{align*}
\widehat{A} \cneq 
\mathbb{C}[x_1, \cdots, x_{n}, x_{n+1}, \cdots, x_{n-\varepsilon}]
\end{align*}
and consider the element $\widehat{W} \in \widehat{A}$ defined by
\begin{align*}
\widehat{W} \cneq W + x_{n+1}^{d} + \cdots + x_{n-\varepsilon}^d. 
\end{align*}
Since we assume that $X$ does not contain stacky points, the 
stack
\begin{align*}
\widehat{X} \cneq (\widehat{W}=0) 
\subset \mathbb{P}(a_1, \cdots, a_n, 1, \cdots, 1)
\end{align*}
also does not contain stacky points. 
The variety $\widehat{X}$ is a projective Calabi-Yau 
manifold with dimension $n-2-\varepsilon$, 
which contains $X$ as a zero locus 
$x_{n+1}= \cdots =x_{n-\varepsilon}=0$. 

Let $\widehat{R}$ be the graded ring 
$\widehat{A}/(\widehat{W})$. 
There is a natural push-forward functor
\begin{align*}
i_{\ast} \colon 
D_{\rm{sg}}^{\rm{gr}}(R) \to D_{\rm{sg}}^{\rm{gr}}(\widehat{R})
\end{align*}
by regarding a graded $R$ module as a 
graded $\widehat{R}$-module
via the surjection $\widehat{R} \twoheadrightarrow R$. 
(cf.~\cite{UedaM}.)
Combined with the equivalence (\ref{Cok})
and the functor $\Psi=\Psi_1$ in (\ref{def:Psi}), 
we obtain the diagram
\begin{align}\label{commute}
\xymatrix{
\HMF(W) \ar[r]^{i_{\ast}} & \HMF(\widehat{W}) \\
D^b \Coh(X) \ar[r]^{i_{\ast}} \ar[u]^{\Psi} & D^b \Coh(\widehat{X}) 
\ar[u]_{\widehat{\Psi}}. 
}
\end{align}
Here $i\colon X \to \widehat{X}$ is the inclusion and 
$\widehat{\Psi}$ is the equivalence,
obtained by applying the same construction of 
$\Psi$ to $\widehat{W}$
and $\widehat{X}$. 
We have the following lemma: 
\begin{lem}
The diagram (\ref{commute}) is commutative. 
\end{lem}
\begin{proof}
The result follows from 
the definitions of $\Psi$, $\widehat{\Psi}$ and the adjunction
for $E \in D^b \Coh(X)$
\begin{align*}
\bigoplus_{j\ge 1} \dR \Hom_{\widehat{X}}(\oO_{\widehat{X}}(j), i_{\ast}E)
\cong \bigoplus_{j\ge 1} \dR \Hom_X(\oO_X(j), E).
\end{align*}
\end{proof}

The top arrow $i_{\ast}$
of the diagram (\ref{commute}) 
obviously commutes with grade shift functors on both 
sides. Also by the functoriality of the Hochschild homologies, 
we have the push-forward functor
\begin{align*}
i_{\ast} \colon \mathrm{HH}_{\ast}(W) \to \mathrm{HH}_{\ast}(\widehat{W})
\end{align*}
which preserves the one dimensional 
eigenspaces 
in Lemma~\ref{lem:Hoch} on both sides. 
By Remark~\ref{rmk:Chern}, the composition
\begin{align*}
\widehat{Z}_G \circ i_{\ast} \colon K(\HMF(W)) \to K(\HMF(\widehat{W})) \to 
\mathbb{C}
\end{align*}
differs from $Z_G$ by a scalar constant, where $\widehat{Z}_G$ is the 
central charge (\ref{ZG})
on $\HMF(\widehat{W})$
 applied for $\widehat{W}$. 
Since $i_{\ast} \mathbb{C}(0)=\mathbb{C}(0)$, it 
follows that
\begin{align*}
Z_G(P)=(1-e^{2\pi \sqrt{-1}/d})^{\varepsilon} \widehat{Z}_G (i_{\ast}P)
\end{align*}
for any $P \in \HMF(W)$
by comparing $Z_G(\mathbb{C}(0))$ and 
$\widehat{Z}_G(\mathbb{C}(0))$ given in Example~\ref{exam:C0}.
In particular, using the diagram (\ref{commute}), 
it follows that 
\begin{align*}
Z_G(\Psi(\oO_x))
&=
(1-e^{2\pi \sqrt{-1}/d})^{-1}
\prod_{j=1}^{n} \left(1-e^{2a_j \pi \sqrt{-1}/d} \right) \\
&=: -C_W
\end{align*} 
where $C_W$ coincides with the one 
defined in (\ref{CW}). 
As a summary, we have the following: 
\begin{prop}
Suppose that $\varepsilon<0$. 
For $E \in D^b \Coh(X)$, 
the central charge
$Z_G(\Psi(E))$ is written as
\begin{align}\label{write:negative}
Z_G(\Psi(E))=
C_W \sum_{j=0}^{n-2-\varepsilon} \int_{\widehat{X}} 
\widehat{\alpha}_j^{\dag} \cdot
\ch_j(i_{\ast}E)
\end{align}
where $(\widehat{\alpha}_0^{\dag}, \cdots, 
\widehat{\alpha}_{n-2-\varepsilon}^{\dag})$
satisfies
\begin{align*}
\widehat{\alpha}_j^{\dag} \in \mathbb{C}\widehat{H}^{n-2-\varepsilon -j}, 
\quad 
\widehat{H} \cneq c_1(\oO_{\widehat{X}}(1))
\end{align*}
and it is the unique solution of the equation (\ref{sol:unique})
for $\widehat{X}$. 
\end{prop}
Note that $\widehat{\alpha}_j^{\dag}$ is computed 
by the argument in the previous subsection, 
since $\widehat{X}$ is Calabi-Yau. 
Later we will use the following data:
\begin{align}\label{Zdag:com}
Z_G(\Psi(\oO_x))
&=-C_W \\ \notag
&\in \mathbb{R}_{>0} e^{\sqrt{-1}\pi (\theta_W + 1)}
\\
\notag
Z_G(\mathbb{C}(j)) 
&=C_W e^{2\pi j \sqrt{-1}/d}\left( 1- e^{2\pi \sqrt{-1}/d} \right) \\
\notag
 &\in \mathbb{R}_{>0} e^{\sqrt{-1} \pi \left(\theta_W 
+ \frac{1}{d} +\frac{2j}{d}
+ \frac{3}{2} \right)}.
\end{align}
Here
$x \in X$ and $\theta_W \in \mathbb{Q}$
 is defined by (\ref{thetaW}).
The relation (\ref{Zdag:com}) is a consequence of 
the above arguments and the computation in Example~\ref{exam:C0}.

\subsection{A recipe constructing Gepner type stability conditions}
\label{subsec:strategy}
In this subsection, we explain
how desired Gepner type stability conditions
are constructed. 
We divide the construction into 3-steps:
construction of a 
slope stability on $\aA_W$, 
construction of $\sigma_G$ via tilting 
of $\aA_W$, and checking the 
Gepner type property of $\sigma_G$. 
In the next section, 
we will apply the 
recipe here in the 
case of $n-4 \le \varepsilon \le 0$. 
\begin{sstep}
\end{sstep}
Our first step is to
construct an analogue of a slope stability on $\aA_W$. 
This is a map
\begin{align}\label{mu:slope}
\mu \colon \aA_W \to \mathbb{R} \cup\{ \pm \infty\}
\end{align}
satisfying the weak seesaw property: 
for any exact sequence 
$0 \to F \to E \to G \to 0$ in 
$\aA_W$, we have either 
\begin{align*}
&\mu(F) \le \mu(E) \le \mu(G) \mbox{ or } \\
&\mu(F) \ge \mu(E) \ge \mu(G).  
\end{align*}
The above slope function defines the $\mu$-stability 
in $\aA_W$:
\begin{defi}
An object $E \in \aA_W$ is $\mu$-(semi)stable 
if for any exact sequence $0 \to F \to E \to G \to 0$
in $\aA_W$, we have the inequality
\begin{align*}
\mu(F)<(\le) \mu(G).
\end{align*}
\end{defi}
We require that $\mu$-stability 
satisfies the Harder-Narasimhan property, i.e. 
for any $E \in \aA_W$, there is a 
filtration in $\aA_W$
\begin{align*}
0=E_0 \subset E_1 \subset \cdots \subset E_N=E
\end{align*}
such that each subquotient $F_i=E_i/E_{i-1}$ is 
$\mu$-semistable with 
$\mu(F_i)>\mu(F_{i+1})$ for all $i$. 

\begin{sstep}
\end{sstep}
Suppose that there is a slope function $\mu$ as above. 
We define a pair of full 
subcategories $(\tT_{\mu}, \fF_{\mu})$ 
by 
\begin{align}\label{mu:tilting}
&\tT_{\mu} \cneq
 \langle E \in \aA_W : E \mbox{ is } \mu \mbox{-semistable with }
\mu(E)>0 \rangle_{\rm{ex}} \\
\notag
&\fF_{\mu} \cneq
 \langle E \in \aA_W : E \mbox{ is } \mu \mbox{-semistable with }
\mu(E)\le 0 \rangle_{\rm{ex}}. 
\end{align}
The existence of Harder-Narasimhan filtrations in $\mu$-stability 
implies that $(\tT_{\mu}, \fF_{\mu})$ 
is a torsion pair on $\aA_W$. (cf.~\cite{HRS}.)
We define $\aA_G$ to be the associated tilting:
\begin{align}\label{AG:tilting}
\aA_G \cneq \langle \fF_{\mu}, \tT_{\mu}[-1] 
\rangle_{\rm{ex}} \subset \HMF(W). 
\end{align}
The category $\aA_G$ is the heart of a bounded t-structure on $\HMF(W)$. 
We try to construct a desired stability 
condition from the heart $\aA_G$, 
by the following:
\begin{defi}
We say that a triple
\begin{align}\label{triple}
(Z_G, \aA_G, \theta), \quad \theta \in \mathbb{R}
\end{align}
determines a stability condition if 
the following condition holds: 
\begin{itemize}
\item For any $0\neq E \in \aA_G$, we have 
\begin{align}\label{Htheta}
Z_G(E) 
\in
\{ r e^{\sqrt{-1} \pi \phi} : r>0, \phi \in (\theta, \theta+1] \}. 
\end{align}
\item Any object in $\aA_W$ admits a Harder-Narasimhan filtration 
with respect to $Z_G$-stability. 
\end{itemize}
\end{defi}
Here $Z_G$-stability and its Harder-Narasimhan 
filtrations are defined by the same way as
in the $\mu$-stability, by replacing $\mu$
by $\arg Z_G(\ast) \in (\theta, \theta+1]$. 
If the triple (\ref{triple})
determines a stability condition, 
it associates a pair 
\begin{align}\label{sigma:dag}
\sigma_G=
(Z_G, \{\pP_G(\phi)\}_{\phi \in \mathbb{R}}), \quad 
\pP_G(\phi) \subset \HMF(W)
\end{align}
in the following way: 
we define 
$\pP_G(\phi)$ for $\phi \in (\theta, \theta+1]$
to be 
\begin{align*}
\pP_G(\phi) = \left\{ E \in \aA_W : 
Z_{G}\mbox{-semistable with }
Z_G(E) \in \mathbb{R}_{>0} e^{\sqrt{-1} \pi \phi}
 \right\} \cup \{0\}
\end{align*}
and other $\pP_G(\phi)$ are determined by the rule
\begin{align*}
\pP_G(\phi+1)=\pP_G(\phi)[1].
\end{align*} 
If $\theta=0$, the above construction is nothing but 
the one given in~\cite[Proposition~5.3]{Brs1}, 
and the same argument applies to show that 
(\ref{sigma:dag}) is a stability condition.  
Below for an interval $I \subset \mathbb{R}$, we 
set
\begin{align*}
\pP_G(I) \cneq \langle \pP_G(\phi) :
\phi \in I \rangle_{\rm{ex}}.
\end{align*}
 Note that
$\pP_G((\theta, \theta+1])$ coincides with $\aA_G$
by our construction. We require the local finiteness 
of our stability condition, i.e. for any $\phi \in
\mathbb{R}$, the quasi-abelian category $\pP_G((\phi-\delta, \phi+ \delta))$
is noetherian and artinian for $0<\delta \ll 1$.
(cf.~\cite[Definition~5.7]{Brs1}.)
It in particular implies that any object $E \in \pP_G(\phi)$
admits a Jordan-H$\ddot{\rm{o}}$lder filtration.  
\begin{rmk}
The local finiteness condition holds if the 
image of $Z_G$ is discrete. 
By Remark~\ref{discrete}, this
is always satisfied in the cases 
studied in the next section, 
so we will not take care of the 
local finiteness. 
\end{rmk}

\begin{sstep}

\end{sstep}
In this step, we assume that 
the triple (\ref{triple}) determines a 
stability condition $\sigma_G$. 
We expect that $\sigma_G$
is a Gepner type stability condition 
with respect to $(\tau, 2/d)$. 
To show this, we consider the following stability condition
\begin{align}\label{tau-1}
\tau_{\ast}^{-1}\sigma_G \left(\frac{2}{d}  \right)
=(Z_G, \{\pP_G'(\phi)\}_{\phi \in \mathbb{R}})
\end{align}
where $\pP_G'(\phi)$ is given by 
\begin{align*}
\pP_G'(\phi)= \tau^{-1} \pP_G \left(\phi+\frac{2}{d} \right). 
\end{align*}
It is enough to show that
(\ref{tau-1})
coincides with $\sigma_G$. 
This is equivalent to that
$\pP_G'((\theta, \theta+1]) = \pP_G((\theta, \theta+1])$, 
or equivalently
\begin{align}\label{A=P}
\tau(\aA_G) = \pP_G \left( \left(\theta+ \frac{2}{d}, 
\theta+ \frac{2}{d} +1 \right] \right). 
\end{align}
We show the equality
(\ref{A=P}) by investigating $\sigma_G$-stability of 
simple objects in $\aA_W$. 
When $n=2$ we have the following
lemma: 
\begin{lem}\label{lem:fur}
Suppose that $n=2$ and 
 the following 
inequality holds: 
\begin{align}\label{condition:n=2}
\theta_W
-\frac{1}{d} -\frac{2\varepsilon}{d} - \frac{1}{2}
 \le \theta < \theta_W +1. 
\end{align}
If $\tau\Psi(\oO_x)$, $\mathbb{C}(1), \cdots, \mathbb{C}(-\varepsilon)$
are $\sigma_G$-semistable 
for all $x\in X$, then the equality (\ref{A=P}) holds. 
\end{lem}
\begin{proof}
Let $\phi_x$ be the phase of $\tau \Psi(\oO_x)$
for $x\in X$ and 
$\phi_j$ the phase of $\mathbb{C}(j)$
for $1\le j\le -\varepsilon$. 
Since $\tau \Psi(\oO_x)$ and $\mathbb{C}(j)$
for $1\le j\le -1-\varepsilon$ 
are objects in $\aA_W$, and $\aA_G$ is obtained as a 
tilting of $\aA_W$, the phases 
$\phi_x$, $\phi_j$ are contained in $(\theta, \theta+2]$
for $1\le j\le -1-\varepsilon$. 
On the other hand, the
 condition (\ref{condition:n=2}) implies that
\begin{align}\label{ineqs}
\theta < \theta_W + 1 < 
\theta_W + 1+ \frac{2}{d} < 
 \theta_W + \frac{1}{d} + \frac{3}{2} < \cdots \\
\notag
\cdots < \theta_W + \frac{1}{d} + \frac{2(-1-\varepsilon)}{d} + \frac{3}{2} 
\le \theta +2. 
\end{align}
By comparing (\ref{ineqs}) with (\ref{Zdag:com}), 
we obtain
\begin{align}\label{phases}
\phi_x= \theta_W + 1 + \frac{2}{d}, \quad 
\phi_j= \theta_W + \frac{1}{d} + \frac{2j}{d} + \frac{3}{2}
\end{align}
for $1\le j\le -1-\varepsilon$. 
We show that (\ref{phases}) also holds for $j=-\varepsilon$. 
By Lemma~\ref{lem:filt:A}, we have 
$\mathbb{C}(-\varepsilon)[-1] \in \aA_W$.
Therefore if $\mathbb{C}(-\varepsilon)$
 is $\sigma_G$-stable, then 
we have either
\begin{align}\label{former}
&\mathbb{C}(-\varepsilon) \in \aA_G[1], \quad \phi_{-\varepsilon} \in (\theta+1, \theta+2] \quad \mbox{ or }\\
\label{latter}
&\mathbb{C}(-\varepsilon) \in \aA_G[2], \quad \phi_{-\varepsilon} \in 
(\theta+2, \theta+3]. 
\end{align}
In the case of (\ref{former}), 
the equality 
(\ref{phases}) also holds for $j=-\varepsilon$ by the inequalities
(\ref{ineqs}). 
In the case of (\ref{latter}), 
 we need to exclude the case of
\begin{align*}
\phi_{-\varepsilon}=\phi_{-1-\varepsilon}+2+\frac{2}{d}.
\end{align*}
If this happens, then 
$\phi_{-1-\varepsilon} < \theta+1$, and Lemma~\ref{lem:filt:A}
and (\ref{ineqs}) imply that
\begin{align*}
Z_G(\mathbb{C}(-\varepsilon)[-1]) \in \{ \mathbb{R}_{>0}
e^{\sqrt{-1} \pi \phi} : \phi \in (\theta, \theta+1) \}.
\end{align*}
This contradicts to
that $\phi_{-\varepsilon} \in (\theta+2, \theta+3]$, hence 
(\ref{phases}) also holds for $j=-\varepsilon$. 

Note that $\tau(\aA_W)$ is generated by 
$\tau\Psi(\oO_x)$ for all $x\in X$
and $\mathbb{C}(j)$ for $1\le j\le -\varepsilon$, 
whose phases are given by (\ref{phases}). 
By the inequality (\ref{condition:n=2}), we have
\begin{align*}
\theta + \frac{2}{d} \ge 
\theta_W + \frac{1}{d} -\frac{2\varepsilon}{d} + 
\frac{3}{2}-j = \phi_{-\varepsilon} -j
\end{align*}
for $j\ge 2$. 
Noting that there is no non-trivial 
homomorphism from $\pP(\phi)$ to $\pP(\phi')$ if $\phi>\phi'$, 
it follows that 
\begin{align}\label{vvanish:1}
\Hom^{<-1}\left( \pP_G \left(\left( \theta + \frac{2}{d}, 
\theta + \frac{2}{d} +1 \right]\right), \tau(E) \right)=0
\end{align}
for any $E\in \aA_W$. 
Similarly the inequality (\ref{condition:n=2}) implies
\begin{align*}
\phi_x = \theta_W + 1  + \frac{2}{d}>
\theta + \frac{2}{d} + 1-j
\end{align*}
for $j\ge 1$. It follows that 
\begin{align}
\label{vvanish:2}
\Hom^{<0} \left( \tau(E), \pP_G \left(\left( \theta + \frac{2}{d}, 
\theta + \frac{2}{d} +1 \right]\right) \right)=0
\end{align}
for any $E \in \aA_W$. 
The above vanishing (\ref{vvanish:1}), (\ref{vvanish:2})
imply that the RHS of (\ref{A=P})
is obtained as a tilting of $\tau(\aA_W)$. 
Hence the result follows from Lemma~\ref{lem:tilting} below:  
\end{proof}
We have used the following lemma: 
\begin{lem}\label{lem:tilting}
Let $\dD$ be a triangulated category,
$\aA \subset \dD$ the heart of a bounded t-structure on $\dD$, 
and $Z \colon K(\dD) \to \mathbb{C}$ a group homomorphism. 
Suppose that there are torsion pairs $(\tT_k, \fF_k)$, $k=1, 2$
on $\aA$ such that, for $\bB_k=\langle \fF_k, \tT_k[-1] \rangle_{\rm{ex}}$
the associated tilting, both of the triples
\begin{align}\label{triples}
(Z, \bB_1, \theta), \quad (Z, \bB_2, \theta)
\end{align}
determine stability conditions. Then 
$\bB_1=\bB_2$. 
\end{lem}
\begin{proof}
We may assume that $\theta=0$. Let us take an object $E \in \tT_1$. 
Since $(\tT_2, \fF_2)$ is a torsion pair, there is an exact 
sequence in $\aA$
\begin{align*}
0 \to F \to E \to G \to 0
\end{align*}
such that $F \in \tT_2$ and $G \in \fF_2$. 
Also since $(\tT_1, \fF_1)$ is a torsion pair, 
objects in $\tT_1$ are closed under quotients, 
hence $G \in \tT_1$.
Suppose that $G\neq 0$. 
Then since (\ref{triples}) determine stability conditions for $\theta=0$, 
it follows that 
$\Imm Z(G) =0$. 
Hence the condition $G \in \fF_2$ implies 
$\Ree Z(G) \in \mathbb{R}_{<0}$ but the 
condition $G \in \tT_2$ implies 
$\Ree Z(G) \in \mathbb{R}_{>0}$, which is a 
contradiction. Therefore $G=0$, and 
$\tT_1 \subset \tT_2$ follows. 
Similarly $\tT_2 \subset \tT_1$ also holds, 
hence $\tT_1=\tT_2$ holds. Because $\fF_k$ is an 
orthogonal complement of $\tT_k$ in $\aA$, we have 
$\fF_1=\fF_2$, hence $\bB_1=\bB_2$ holds. 
\end{proof}

Next we discuss the case of $n=3$. 
In this case, we need to add an extra 
check of $\sigma_G$-stability. 
\begin{lem}\label{lem:con:n=3}
Suppose that 
$n=3$ and the following inequality 
holds:
\begin{align}\label{condition:n=3}
\theta_W
-\frac{1}{d} -\frac{2\varepsilon}{d} - \frac{1}{2}
\le \theta \le \theta_W. 
\end{align}
Suppose furthermore that 
$\tau\Psi(\oO_x)$ is $\sigma_G$-stable, 
$\mathbb{C}(1), \cdots, \mathbb{C}(-\varepsilon)$
are $\sigma_G$-semistable and the following holds
for all $x\in X$: 
\begin{align}\label{check:add}
\tau^{1-\varepsilon}\Psi(\oO_x) \in 
\pP_G\left( 1+ \theta_W + \frac{2(1-\varepsilon)}{d}  \right). 
\end{align}
Then the equality (\ref{A=P}) holds.
\end{lem}
\begin{proof}
By following the proof of Lemma~\ref{lem:fur}, 
we have the vanishing (\ref{vvanish:1}), (\ref{vvanish:2})
for $E=\Psi(\oO_x)$, $\mathbb{C}(0), \cdots, \mathbb{C}(-1-\varepsilon)$
with $x\in X$. 
Below we show the above vanishing also holds 
for $E=\Psi(F)$
for any $F \in \Coh(X)$. Then the RHS of (\ref{A=P})
is shown to be a tilting of $\tau(\aA_W)$, hence 
Lemma~\ref{lem:tilting} is applied to give the result. 

Let us take an object
\begin{align}\label{cond:A}
A \in \tau^{-1} \pP_G \left( \left( \theta + \frac{2}{d}, 
\theta + \frac{2}{d} +1 \right] \right)
\end{align}
and let
\begin{align*}
\Psi^R \colon \HMF(W) \to D^b \Coh(X)
\end{align*}
 be the right adjoint functor of $\Psi$. 
We claim that $\Psi^{R}(A) \in D^b \Coh(X)$
satisfies 
\begin{align}\label{vanish:coho}
\hH^j(\Psi^{R}(A))=0, \quad \mbox{ for } j\neq 0, 1.
\end{align}
The above property
will be proved in Sublemma~\ref{sublem:vanish} below, and 
we continue the proof assuming this fact.  
Let us take the distinguished triangle
in $\HMF(W)$
\begin{align}\label{tr:AB}
\Psi \Psi^R(A) \to A \to B
\end{align}
where $B$ satisfies
\begin{align*}
B \in 
\langle \mathbb{C}(-1-\varepsilon), \cdots, \mathbb{C}(0) \rangle.
\end{align*}
For a coherent sheaf
$F$ on $X$, we apply $\Hom(\Psi(F), \ast)$ to the 
distinguished triangle (\ref{tr:AB}). 
Since $B$ is right orthogonal to $\Psi D^b \Coh(X)$, 
and $\Psi^R(A)$
satisfies the condition (\ref{vanish:coho}), we have
\begin{align*}
\Hom^{<0}(\tau\Psi(F), \tau(A)) &\cong \Hom^{<0}(F, \Psi^R(A)) \\
 &\cong 0
\end{align*}
which proves the vanishing 
(\ref{vvanish:2}) for $E=\Psi(F)$. 

Similarly applying $\Hom(\ast, \Psi(F))$ to the triangle (\ref{tr:AB}), 
and noting 
that 
\begin{align*}
\Hom^{<-1}(\Psi \Psi^R(A), \Psi(F)) &\cong  \Hom^{<-1}(\Psi^R(A), F) \\
&\cong 0
\end{align*}
by the property (\ref{vanish:coho}), 
we see that the vanishing (\ref{vvanish:1})
for $E=\Psi(F)$ is equivalent to 
\begin{align}\label{vanish:B}
\Hom^{<-1}(B, \Psi(F)) \cong 0.  
\end{align}
To show (\ref{vanish:B}), note that 
the vanishing 
(\ref{vvanish:1}) 
for $E=\mathbb{C}(j)$,
$0\le j\le -1-\varepsilon$
and the triangle (\ref{tr:AB})
imply 
\begin{align*}
\Hom^{<-1}(B, \mathbb{C}(j)) \cong 0, \quad j=0, \cdots, -1-\varepsilon. 
\end{align*}
Therefore, if we denote by $\hH_{\aA_W}^{i}(B) \in \aA_W$ the 
$i$-th cohomology with respect to the t-structure on $\HMF(W)$
with heart $\aA_W$, then 
we have $\hH_{\aA_W}^i(B)=0$ for 
$i>1$. Since $\Psi(F) \in \aA_W$, this implies that
 (\ref{vanish:B}) holds. 
Therefore the vanishing (\ref{vvanish:1}) 
for $E=\Psi(F)$ holds. 
\end{proof}

We have used the following sublemma: 

\begin{sublem}\label{sublem:vanish}
The condition (\ref{vanish:coho})
holds. 
\end{sublem}
\begin{proof}
We investigate the vanishing 
\begin{align}\label{vanish:inv}
\Hom^{j}(\Psi^{R}(A), \oO_x)=0
\end{align}
for $x\in X$ and $j\in \mathbb{Z}$.
By the Serre duality on $X$, adjunction
and applying $\tau$, 
the vanishing (\ref{vanish:inv})
is equivalent to 
\begin{align}\label{vanish:inv2}
\Hom(\tau \Psi(\oO_x)[j-1], \tau(A))=0. 
\end{align} 

We first show that (\ref{vanish:inv2})
holds for any $j<-1$ and $x\in X$. 
Applying the Serre functor $\sS_W=\tau^{-\varepsilon}[1]$
 on $\HMF(W)$, 
the vanishing (\ref{vanish:inv2}) is equivalent to 
\begin{align}\label{vanish:dual}
\Hom(\tau(A)[-j], \tau^{1-\varepsilon}\Psi(\oO_x))=0. 
\end{align}
On the other hand, by the assumption 
(\ref{condition:n=3}), we have 
the inequality
\begin{align*}
\theta + \frac{2}{d} -j \ge 
1 + \theta_W 
+ \frac{2(1-\varepsilon)}{d}
\end{align*}
for $j<-1$. 
Therefore by our assumptions (\ref{check:add})
and (\ref{cond:A}), 
the vanishing (\ref{vanish:dual}) holds for $j<-1$. 

Next, we have the inequality
\begin{align*}
\theta_W + 1 + \frac{2}{d} + j-1 >
\theta + \frac{2}{d} +1
\end{align*}
for $j\ge 2$ by (\ref{condition:n=3}), 
hence the vanishing 
(\ref{vanish:inv2}) holds 
for $j\ge 2$. 
Moreover the above inequality, 
hence the vanishing (\ref{vanish:inv2}),  
also holds for $j=1$
unless 
$\theta=\theta_W$ holds. 
Suppose that $\theta=\theta_W$
holds, and let $P$ be the 
$\sigma_G$-Harder-Narasimhan 
factor of $\tau(A)$
with the maximum phase. 
Note that $\tau\Psi(\oO_x)$ 
is $\sigma_G$-stable with
phase $\theta_W + 1+ 2/d$, 
which 
 is bigger than or equal to the phase of $P$. 
Therefore 
the vanishing of  
(\ref{vanish:inv2}) for $j=1$
does \textit{not} hold 
only if we have
\begin{align*}
P \in \pP_G\left( \theta_W + 1+ \frac{2}{d}  \right)
\end{align*}
and
$\tau \Psi(\oO_x)$
is one of the Jordan-H$\ddot{\rm{o}}$lder factors 
of $P$. 
It follows that, by 
taking the Jordan-H$\ddot{\rm{o}}$lder
filtration of $P$, 
there is a distinguished triangle
\begin{align}\label{dist:Q}
\tau\Psi(Q) \to \tau(A) \to \tau(A')
\end{align}
where $Q$ is a zero dimensional coherent sheaf on
$X$, $A'$ is an object in the RHS of (\ref{cond:A}), 
such that the vanishing (\ref{vanish:inv2})
holds for $j\ge 1$ after replacing $A$ by $A'$. 

Applying $\Psi^R \circ \tau^{-1}$
to (\ref{dist:Q}), we have the distinguished triangle
\begin{align*}
Q \to \Psi^{R}(A) \to \Psi^R(A'). 
\end{align*}
The above argument shows that, 
after replacing $A$ by $A'$,  
the vanishing (\ref{vanish:inv})
holds unless $j=-1, 0$. 
It follows that 
$\Psi^R(A')$ is a two term complex of vector 
bundles on $X$, whose cohomologies 
are concentrated on degrees $0$ and $1$. 
Since $Q \in \Coh(X)$, 
we conclude that $\hH^j(\Psi^R(A))=0$ for 
$j\neq 0, 1$.  
\end{proof}

By the above results, the problem is reduced 
to showing $\sigma_G$-stability of 
some objects
in $\HMF(W)$. 
The following lemma is useful 
in checking the $\sigma_G$-stability
of these objects.  
The proof is obvious, and we omit it. 
\begin{lem}\label{lem:cmin}
Let $\dD$ be a triangulated category, 
$Z \colon K(\dD) \to \mathbb{C}$ 
a group homomorphism
and $\aA \subset \dD$ the heart 
of a bounded t-structure on $\dD$.  
Suppose that 
the triple
$(Z, \aA, \theta=0)$
determines a stability condition on $\dD$, and the 
following condition holds: 
\begin{align*}
c_{\rm{min}} \cneq \mathrm{inf}
\{ \Imm Z(E)>0 : E \in \aA \} >0. 
\end{align*}
Then an object $E \in \aA$ with $\Imm Z(E)=c_{\rm{min}}$
is $\sigma_G$-stable if and only if
$\Hom(P, E)=0$ for any $P \in \aA$ with 
$\Imm Z(P)=0$. 
\end{lem}

\section{Proof of Theorem~\ref{thm:main}}\label{sec:proof}
In this section, 
 we apply the strategy in the previous 
section and prove Theorem~\ref{thm:main}. 
Below we use the same notation in
the previous section. In particular
the constants $C_W \in \mathbb{C}^{\ast}$
and $\theta_W \in \mathbb{Q}$ are 
defined as in Lemma~\ref{lem:constant}, 
both in $\varepsilon=0$ and $\varepsilon<0$ case. 
The goal is to prove Conjecture~\ref{conj2:Gep}
when  
\begin{align*}
n-4 \le \varepsilon \le 0
\end{align*}
and $X$ does not contain stacky points. 
Since 
we already discussed the case with $n=1$ and 
$n=2$, $\varepsilon=0$, 
the following 
five possibilities are left: 
\begin{align*}
(n, \varepsilon)=(3, 0), \ (2, -1), \ (4, 0), \ 
(3, -1), \ (2, -2). 
\end{align*}
We divide the proof into the five 
subsections, so that each subsection 
corresponds to one of the above types. 
We repeat similar arguments in these subsections, 
so we recommend the readers to follow only one or two 
cases, e.g. $(n, \varepsilon)=(4, 0)$
or $(3, -1)$, at the first reading of this paper.

\subsection{The case of $n=3$, $\varepsilon=0$}
In this subsection, we study the case of 
$n=3$ and $\varepsilon=0$. 
In this case, 
$X$ is a smooth elliptic curve, and 
the heart $\aA_W$ is given by 
\begin{align*}
\aA_W  = \Psi \Coh(X). 
\end{align*}
Furthermore 
possible data $(a_1, a_2, a_3, d, \int_X H)$ 
are classified into the 
three types~\cite[Table~2]{Saito}
\begin{align}\label{cla:n=3}
\left(a_1, a_2, a_3, d, \int_X H \right)
= \left\{ \begin{array}{l}
 (1, 1, 1, 3, 3) \\
(2, 1, 1, 4, 2)  \\
(3, 2, 1, 6, 1).
\end{array}   \right. 
\end{align}
The central charge $Z_G$ is described as follows:
\begin{lem}\label{Zdag:n=3}
Suppose that $n=3$
and $\varepsilon=0$. 
 For any $E \in D^b \Coh(X)$, we have 
\begin{align}\label{cent:n=3}
&Z_G(\Psi(E)) = \\ \notag
&C_W \left\{
-d(E) + r(E) \left( \cos \frac{2\pi}{d} -1  \right)
+ r(E) \sin \frac{2\pi}{d} \sqrt{-1}
 \right\}. 
\end{align}
Here we set $(r(E), d(E))=(\rank(E), \deg(E))$. 
\end{lem}
\begin{proof}
The equation (\ref{sol:unique}) becomes
\begin{align}\notag
(\alpha_0^{\dag}, -1)\left(\begin{array}{cc}
1-e^{2\pi \sqrt{-1}/d} -\int_X H & -1 \\
H & 1-e^{2\pi \sqrt{-1}/d}
\end{array}  \right)=0,
\end{align}
giving $\int_X \alpha_0^{\dag}=e^{2\pi \sqrt{-1}/d}-1$. 
\end{proof}

We set
the slope function $\mu$ in (\ref{mu:slope})
to be the constant function $\mu=-1$, 
so that the resulting heart $\aA_G$ in (\ref{AG:tilting}) 
coincides with $\aA_W$. 
We have the following result: 
\begin{prop}\label{prop:30}
Suppose that $n=3$ and $\varepsilon=0$.
Then the triple
\begin{align}\label{triple:30}
(Z_G, \aA_W, \theta=\theta_W)
\end{align} determines a Gepner 
type stability condition $\sigma_G$
on $\HMF(W)$
with respect to 
$(\tau, 2/d)$. 
\end{prop}
\begin{proof}
Note that, since $X$ is an elliptic curve, 
the space of stability conditions on $D^b \Coh(X)$ is 
completely described in~\cite[Section~9]{Brs1}. 
Since $\aA_W= \Psi \Coh(X)$, 
and the central charge $Z_G$ is given by (\ref{cent:n=3}), 
it follows that the triple (\ref{triple:30}) satisfies 
the condition (\ref{Htheta}).
Then the same argument of~\cite[Example~5.4]{Brs1}
shows that the triple (\ref{triple:30})
determines a stability condition $\sigma_G$
on $\HMF(W)$. 
Now we are going to apply Lemma~\ref{lem:con:n=3}. 
Note that the inequality (\ref{condition:n=3}) is satisfied in
this case. For a point $x\in X$, we have 
\begin{align*}
\tau \Psi(\oO_x) \cong \Psi(\oO_X(-x)[1])
\end{align*} 
by Proposition~\ref{prop:grade}. 
Since the above object is $\sigma_G$-stable
with phase $1+ \theta_W + 2/d$,
Lemma~\ref{lem:con:n=3} implies that 
$\sigma_G$ is a Gepner type stability condition 
with respect to $(\tau, 2/d)$. 
\end{proof}

\subsection{The case of $n=2$, $\varepsilon=-1$}
In this subsection, we study 
the case of $n=2$, $\varepsilon=-1$. 
In this case, 
$X$ is a finite number of
smooth points, and 
the heart $\aA_W$ is given by 
\begin{align*}
\aA_W= \langle \mathbb{C}(0), \Psi(\oO_x) : 
x\in X \rangle_{\rm{ex}}. 
\end{align*}
The following lemma immediately follows from (\ref{Zdag:com}): 
\begin{lem}\label{ZG:2-1}
If we write the 
K-theory class of an object $E \in \aA_W$
as 
\begin{align*}
[E] = v_0 [\mathbb{C}(0)] + \sum_{x\in X} w_x [\Psi(\oO_x)]
\end{align*}
and set $w \cneq \sum_{x\in X} w_x$, 
then $Z_G(E)$ is given by
\begin{align*}
Z_G(E)= C_W \left\{ -w + v_0 \left( 1-\cos \frac{2\pi}{d} \right)
-v_0 \sin \frac{2\pi}{d} \sqrt{-1}  \right\}. 
\end{align*}
\end{lem}

In the notation used in Subsection~\ref{subsec:central2}, 
the polynomial $\widehat{W}$ is of type $(a_1, a_2, 1, d)$, which 
must belong to one of the classifications in (\ref{cla:n=3}).
Furthermore, applying
 Lemma~\ref{lem:compute2} and Lemma~\ref{lem:equivalence}, 
we see that $\aA_W$ is equivalent to
the abelian category of representations
 of a certain quiver $\qQ$. 
By the above arguments, 
the possible types of 
$(a_1, a_2, d, \sharp X, \qQ)$ are 
classified as follows:  
\begin{align*}
&
(a_1, a_2, d, \sharp X)=(1, 1, 3, 3), \quad 
\qQ=    \xygraph{
        {\bullet}*+!D{}
        (:^{\pi^{(3)}}[ru]{\bullet}*+!D{}, 
:^{\pi^{(2)}}[r]{\bullet}*+!D{},      
        :_{\pi^{(1)}}[rd]{\bullet}*+!D{})
    } \\
&
(a_1, a_2, d, \sharp X)=(2, 1, 4, 2), \quad
\qQ= \xygraph{
       {\bullet}*+!D{}
        (:^{\pi^{(2)}}[ru]{\bullet}*+!D{},      
        :_{\pi^{(1)}}[rd]{\bullet}*+!D{})
    } \\
&
(a_1, a_2, d, \sharp X)=(3, 2, 6, 1), \quad
\qQ= \xygraph{
       {\bullet}*+!D{}
        :_{\pi^{(1)}}[r]{\bullet}*+!D{}     
            }
\end{align*}
Here the left vertex of the quiver $\qQ$
corresponds to $\mathbb{C}(0)$, and the right
vertices correspond to $\Psi(\oO_x)$ for $x\in X$. 
In all the above cases, 
we set $\mu=-1$ so that the 
heart (\ref{AG:tilting}) coincides with $\aA_W$. 
We have the following result: 
\begin{prop}\label{prop:2-1}
Suppose that $n=2$ and $\varepsilon=-1$. 
Then the triple 
\begin{align}\label{triple:2-1}
(Z_G, \aA_W, \theta), \quad \frac{5}{6} + \theta_W 
\le \theta < 1 + \theta_W
\end{align}
determines a Gepner type stability condition $\sigma_G$
on $\HMF(W)$ with respect to $(\tau, 2/d)$. 
\end{prop}
\begin{proof}
By Lemma~\ref{ZG:2-1},
it follows that the triple (\ref{triple:2-1})
satisfies the condition (\ref{Htheta}). 
Since $\aA_W$ is noetherian and artinian,
the triple (\ref{triple:2-1}) satisfies the Harder-Narasimhan 
property
(cf.~\cite[Proposition~2.4]{Brs1}) hence it determines a 
stability
 condition $\sigma_G$ on $\HMF(W)$. 
In order to show the Gepner type property of $\sigma_G$, 
it is enough to check the
assumption of Lemma~\ref{lem:fur}. 
Since the inequality (\ref{condition:n=2})
is satisfied, 
it remains to show the $\sigma_G$-stability 
of $\tau \Psi(\oO_x)$ and $\mathbb{C}(1)$
for all $x\in X$, which 
we prove in the next lemma. 
\end{proof}

\begin{lem}\label{lem:30:check}
The objects $\tau \Psi(\oO_x)$ and $\mathbb{C}(1)$
are $\sigma_G$-stable. 
\end{lem}
\begin{proof}
As for $\tau\Psi(\oO_x)$, it fits into a unique 
non-trivial 
extension in $\aA_W$
by Lemma~\ref{lem:tau:ox} and Remark~\ref{rmk:ext:uni}
\begin{align*}
0 \to \Psi(\oO_x) \to \tau\Psi(\oO_x) \to \mathbb{C}(0) \to 0. 
\end{align*}
Therefore $\Psi(\oO_x)$
is the only non-trivial subobject
of $\tau\Psi(\oO_x)$ in $\aA_W$. 
Comparing the argument of $Z_G(\ast)$, we see that 
$\tau\Psi(\oO_x)$ is $Z_G$-stable. 

As for $\mathbb{C}(1)$, note 
that the object $\mathbb{C}(1)[-1]$
is contained in $\aA_W$ by Lemma~\ref{lem:filt:A}. 
Let $0\neq F \in \aA_W$ be a proper
subobject of $\mathbb{C}(1)[-1]$
in $\aA_W$. As in Corollary~\ref{cor:quiver2}, 
the inclusion $F \subset \mathbb{C}(1)[-1]$
is represented by the following inclusion of 
quiver representations:
\begin{align*}
F=
 \xygraph{!~:{@{>}} !~-{@{}}
        {\bullet}*+!D{V^{(0)}}
        (:^{}[ru]{\bullet}*+!D{W^{(\sharp X)}}, 
-[r]{\vdots}*+!D{}, 
        :_{}[rd]{\bullet}*+!D{W^{(1)}})
    }
\quad 
\subset 
\quad
\mathbb{C}(1)[-1]=
 \xygraph{!~:{@{>}} !~-{@{}}
        {\bullet}*+!D{R_1}
        (:^{\pi^{(\sharp X)}}[ru]{\bullet}*+!D{\mathbb{C}}, 
-[r]{\vdots}*+!D{},      
        :_{\pi^{(1)}}[rd]{\bullet}*+!D{\mathbb{C}})
    }
\end{align*}
Here $R_1$ is the space of degree one 
elements in $R=\mathbb{C}[x_1, x_2]$, whose dimension is $\sharp X -1$. 
 The morphisms $\pi^{(i)} \colon R_1 \to \mathbb{C}$ are 
evaluations at the points in $X$, and 
$V^{(0)}$, $W^{(i)}$ are finite dimensional
vector spaces with $V^{(0)} \subset R_1$, 
$\dim W^{(i)} \le 1$. 
Let $I$ be the subset of $i \in 
\{1, \cdots, \sharp X\}$ 
satisfying $W^{(i)}=0$. Then we have 
\begin{align*}
V^{(0)} \subset \bigcap_{i \in I} \Ker(\pi^{(i)})
\end{align*}
hence $\dim V^{(0)} \le \sharp X -1 - \lvert I \rvert$. 
Since we have
\begin{align*}
Z_G(\mathbb{C}(1)[-1]) &=
C_W \left\{ -\sharp X +(\sharp X -1) \cdot 
(1-e^{2\pi \sqrt{-1}/d})  \right\} \\ 
Z_G(F) &=C_W \left\{ \lvert I \rvert - \sharp X +
\dim V^{(0)} \cdot (1-e^{2\pi \sqrt{-1}/d}) \right\}
 \end{align*}
it follows that 
the argument of $Z_G(F)$ is less than 
that of $Z_G(\mathbb{C}(1)[-1])$ in 
$(\pi \theta, \pi \theta + \pi]$.
Therefore $\mathbb{C}(1)[-1]$ is $\sigma_G$-stable. 
\end{proof}

\subsection{The case of $n=4$, 
$\varepsilon=0$}
In this subsection, we study the 
case of $(n, \varepsilon)=(4, 0)$. 
In this case, 
 $X$ is a smooth projective K3 surface, 
and the heart $\aA_W$ is given by 
\begin{align*}
\aA_W= \Psi \Coh(X). 
\end{align*}
By Theorem~\ref{thm:Orlov}, the triangulated
 category $\HMF(W)$ is equivalent to 
$D^b \Coh(X)$ via the equivalence $\Psi$. 
On the other hand, the  
 spaces of stability conditions on 
K3 surfaces are 
studied by Bridgeland~\cite{Brs2}. 
Combining the techniques in~\cite{Brs2}
with the arguments in Subsection~\ref{subsec:strategy}, 
we construct a desired Gepner type 
stability condition on $\HMF(W)$. 

Let us first describe the matrix $M$
which appears in the equation (\ref{linear:eq}).
By writing $\int_X H^2 =2m$
for $m\in \mathbb{Z}_{\ge 1}$, 
the matrix $M$ is given by 
\begin{align}\label{matrix:M}
M=
\left( \begin{array}{ccc}
-1-m& -H & -1 \\
H & 1 & 0 \\
m & H & 1
\end{array}  \right). 
\end{align}
By using the above description,  
we give a classification of 
possible types. 
\begin{lem}\label{lem:n=4}
If $n=4$, the possible data 
$(a_1, a_2, a_3, a_4, d, \int_{X} H^2)$
is classified into the two types:
\begin{align}\label{types:n=4}
\left(a_1, a_2, a_3, a_4, d, \int_{X} H^2 \right)
=\left\{ \begin{array}{l}
(1, 1, 1, 1, 4, 4) \\
(3, 1, 1, 1, 6, 2). 
\end{array}  \right. 
\end{align}
\end{lem}
\begin{proof}
By (\ref{matrix:M}), we have
\begin{align*}
\det (M-\lambda \cdot \id)=
-(\lambda+1)\{ \lambda^2 +(m-2)\lambda +1\}. 
\end{align*}
Since $\det (M-\lambda \cdot \id)=0$
for $\lambda= e^{\pm 2\pi \sqrt{-1}/d}$, we have
\begin{align*}
2 \cos \left(2\pi/d \right)=m-2.
\end{align*} 
This is only possible for $(m, d)=(2, 4), (1, 6)$. 
In the latter case, the possibility 
$(a_1, a_2, a_3, a_4)=(2, 2, 1, 1)$ is excluded 
since the stack $X$ always contains a stacky
point. Therefore we obtain the 
classification (\ref{types:n=4}).
\end{proof}

\begin{rmk}\label{K3:cla}
In the type $(1, 1, 1, 1, 4, 4)$
case, $X$ is a quartic K3 surface. 
In the type $(3, 1, 1, 1, 6, 2)$ case, 
$X$ is a double cover of $\mathbb{P}^2$.  
\end{rmk}

In order to describe the central charge $Z_G$, 
it is convenient to use the twisted 
Mukai vector. 
For $B \in H^2(X, \mathbb{R})$, we set 
\begin{align}\label{twist}
v^{B}(E) \cneq e^{-B} \ch(E) \sqrt{\mathrm{td}_X}. 
\end{align}
We denote by $v_i^B(E)$ the $H^{i, i}(X)$-component of 
$v^B(E)$. 
The central charge $Z_G$ is computed in the following way: 
\begin{lem}\label{lem:ZG:4}
By setting $B=-H/2$, the following 
holds for $E \in D^b \Coh(X)$: 
\begin{align}\label{expand}
Z_G(&\Psi(E))=
\\\notag        
&C_W \left( 
-v_2^B(E) + \frac{d}{8} v_0^B(E) + \right. 
\left.
\frac{1}{2}\sqrt{\frac{d}{H^2}} H \cdot v_1^B(E) \sqrt{-1} \right). 
\end{align}
\end{lem}
\begin{proof}
By (\ref{matrix:M}), 
the equation (\ref{sol:unique})
becomes 
\begin{align*}
(\alpha_0^{\dag}, \alpha_1^{\dag}, -1)
\left( \begin{array}{ccc}
-1-m-e^{2\pi \sqrt{-1}/d} & -H & -1 \\
H & 1-e^{2\pi \sqrt{-1}/d} & 0 \\
m & H & 1-e^{2\pi \sqrt{-1}/d}
\end{array}  \right) =0
\end{align*}
which gives
\begin{align*}
\int_X \alpha_0^{\dag}=e^{2\pi \sqrt{-1}/d}-1, \quad
\alpha_1^{\dag}=
\frac{e^{2\pi \sqrt{-1}/d}}{1-e^{2\pi \sqrt{-1}/d}}H. 
\end{align*}
Applying the classification in Lemma~\ref{lem:n=4}, 
a straightforward computation shows the result. 
\end{proof}
\begin{rmk}\label{rmk:int:40}
The equality 
(\ref{expand}) is written as the integral
\begin{align}\label{central:K3}
Z_G(\Psi(E))=
-C_W \int_X e^{-\sqrt{-1} \omega} v^{-H/2}(E)
\end{align}
where $\omega \in \mathbb{R}_{>0} H$ satisfies 
$\int_X \omega^2 = d/4$.
\end{rmk}

By setting $B=-H/2$, we
consider the slope function 
$\mu$ on 
$\aA_W=\Psi\Coh(X)$, defined by 
\begin{align*}
\mu(\Psi(E)) &= \frac{v_1^B (E) \cdot H}{v_0^B(E)}.
\end{align*}
Here $\mu(\Psi(E))$ is defined to be $\infty$ if 
$E$ is a torsion sheaf. 
The above $\mu$-stability 
coincides with the classical slope stability condition
on $\Coh(X)$ via $\Psi$. 
As we discussed in Subsection~\ref{subsec:strategy}, 
the slope function $\mu$ defines a torsion pair (\ref{mu:tilting}) 
on $\aA_W$, and 
the associated tilting $\aA_G$
is given by (\ref{AG:tilting}).
We have the following result: 
\begin{prop}\label{prop:40}
Suppose that $n=4$ and $\varepsilon =0$. 
Then the triple
\begin{align}\label{triple:40}
(Z_G, \aA_G, \theta=\theta_W -1)
\end{align}
determines a Gepner type stability condition
$\sigma_G$ 
on $\HMF(W)$ with respect to $(\tau, 2/d)$. 
\end{prop}
\begin{proof}
Instead of the triple (\ref{triple:40}), 
we consider the triple
\begin{align}\label{triple:400}
(Z_G^{\dag} \cneq Z_G/C_W, \aA_G[1], \theta=0).
\end{align}
Obviously, the 
triple (\ref{triple:40}) determines a stability 
condition if and only if the triple (\ref{triple:400})
determines a stability condition. 
If the latter holds, then
by Remark~\ref{rmk:int:40},  
the
resulting stability condition is $\Psi_{\ast}$ of 
one of  
stability conditions on $D^b \Coh(X)$
constructed in~\cite[Section~6]{Brs2}.  
Therefore, applying~\cite[Lemma~6.2]{Brs2}, 
it is enough to show the following: 
for any spherical torsion free
sheaf $E$ on $X$, one has 
$Z_G^{\dag}(\Psi(E)) \notin \mathbb{R}_{\le 0}$. 
The proof of this fact requires some more
arguments, so we leave the proof to Lemma~\ref{lem:spherical}
below. We continue the proof assuming this fact. 

By the above argument and 
Lemma~\ref{lem:spherical},  
the triple (\ref{triple:40})
determines a stability condition $\sigma_G$
on $\HMF(W)$. 
It remains to show that $\sigma_G$ is Gepner 
type with respect to $(\tau, 2/d)$. 
To show this, let us consider the 
objects $\tau \Psi(\oO_x[-1])$ for $x\in X$. 
By Proposition~\ref{prop:grade}, we have
\begin{align*}
\tau \Psi(\oO_x[-1]) \cong \Psi(I_x) \in \aA_G[1] 
\end{align*}
where $I_x \subset \oO_X$ is the ideal sheaf 
which defines $x$. 
In Lemma~\ref{lem:Ix:stable} below, 
we see that $\Psi(I_x)$ is $\sigma_G$-stable, 
hence it lies in $\pP_G(\theta_W + 2/d)$. 
If we set $\sigma_G'$ to be the stability 
condition defined as in (\ref{tau-1}), this implies that 
\begin{align*}
\Psi(\oO_x[-1]) \in \tau^{-1} \pP_G \left( \theta_W + \frac{2}{d}
  \right) =\pP_G'(\theta_W)
\end{align*}
 for any $x\in X$, and it is 
$\sigma_G'$-stable. 
Since $X$ is a K3 surface, this implies that 
$\pP_G'((\theta_{W}-1, \theta_W])$
is obtained as a tilting of $\aA_W=\Psi(\Coh(X))$, 
by the argument of~\cite[Lemma~10.1]{Brs2}.
Therefore
$\pP_G'((\theta_W-1, \theta_W])$
coincides with $\aA_G$
by Lemma~\ref{lem:tilting}, 
which shows
 that $\sigma_G$ is Gepner type with respect to $(\tau, 2/d)$. 
\end{proof}

\begin{lem}\label{lem:spherical}
For any spherical torsion free sheaf $E$ on 
$X$, one has $Z_G^{\dag}(\Psi(E)) \notin \mathbb{R}_{\le 0}$. 
\end{lem}
\begin{proof}
This is equivalent to saying that, for any 
spherical torsion free 
sheaf $E$ on $X$ with 
$v_1^{B}(E) \cdot H=0$, one has 
\begin{align}\label{spherical}
-v_2^B(E) + \frac{d}{8}  v_0^B(E) >0. 
\end{align}
Since $E$ is a spherical 
sheaf, we have $v^{B}(E)^2=-2$, where the 
square is defined in the Mukai lattice of $X$.  
(cf.~\cite[Lemma~5.1]{Brs2}.)
Combined with
 $v_1^{B}(E) \cdot H=0$ and the Hodge 
index theorem, we have
\begin{align*}
0 \ge v_1^{B}(E)^2 = 2v_0^B(E) v_2^B(E) -2. 
\end{align*}
Noting that $v_0^{B}(E)>0$, we obtain 
$v_2^B(E) \le 1/v_0^{B}(E)$, hence
\begin{align}\label{ineq:vector}
-v_2^B(E) + \frac{d}{8}  
v_0^B(E) \ge \frac{1}{v_0^B(E)}
\left\{ \frac{d}{8} v_{0}^{B}(E)^2  -1 \right\}. 
\end{align}
The RHS of (\ref{ineq:vector}) is positive if $v_0^B(E) \ge 2$. 
Therefore we may assume that $v_0^B(E)=1$, i.e. 
$E$ is rank one. 
Since $E$ is a spherical sheaf, this implies that 
$E$ is a line bundle. 

Let $l$ be the first Chern class of $E$. Since $v_1^B(E) \cdot H =0$, 
we have $l \cdot H = -H^2 /2$, and $v_2^B(E)$ is written as
\begin{align*}
v_2^B(E)= \frac{l^2}{2} -\frac{H^2}{8} + 1. 
\end{align*}
Suppose by a contradiction that (\ref{spherical})
does not hold for a line bundle $E$. 
Then, combining with the Hodge index theorem, we obtain
\begin{align*}
\frac{H^2}{4} -2 + \frac{d}{4} \le l^2 \le \frac{H^2}{4}. 
\end{align*} 
Since we have only two cases, $(d, H^2)=(4, 4), (6, 2)$, 
and $l^2$ is an even integer, it follows that $l^2=0$. 
Also since $-l \cdot H$ is either $2$ (when $X$ is a quartic K3 surface)
or $1$ (when $X$ is a double cover of $\mathbb{P}^2$)
it follows that 
the linear system $|-l|$ defines an elliptic fibration
$X \to \mathbb{P}^1$.  
In particular, its general fiber is
 a smooth elliptic curve $C \in |-l|$ in 
$X$. 
However this is a contradiction: when $X$ is a quartic K3 
surface,  
we have the Castelnuovo inequality
\begin{align*}
g(C) \le \frac{1}{2} (H \cdot C-1)(H\cdot C -2) =0
\end{align*}
which contradicts to $g(C)=1$. 
When $X$ is a double cover of $\mathbb{P}^2$, 
if $\pi \colon X \to \mathbb{P}^2$ is the double 
cover, then $H \cdot C=1$ implies that 
$\pi|_{C}$ is an isomorphism between $C$ and a line in 
$\mathbb{P}^2$, which contradicts to that $C$ is an elliptic 
curve. Therefore (\ref{spherical}) also holds for a line bundle $E$. 
\end{proof}

\begin{lem}\label{lem:Ix:stable}
The object $\Psi(I_x) \in \aA_G[1]$ is $\sigma_G$-stable. 
\end{lem}
\begin{proof}
It is enough to show the $Z_G^{\dag}$-stability of 
$\Psi(I_x)$ with respect to the triple (\ref{triple:400}). 
By Lemma~\ref{lem:ZG:4}, we have the following:
\begin{align}\label{note:imm}
\Imm Z_G^{\dag} (\aA_G[1]) \subset
\left\{ \frac{1}{2} \sqrt{\frac{d}{H^2}} \times \mathbb{Z}_{\ge 0} \right\}. 
\end{align}

We first consider the case $(d, H^2)=(6, 2)$. In this case, 
the imaginary part of $Z_G^{\dag}(\Psi(I_x))$ is $\sqrt{3}/2$, that 
is the smallest positive value of the RHS of (\ref{note:imm}). 
Therefore by Lemma~\ref{lem:cmin}, 
the $Z_G^{\dag}$-stability of $\Psi(I_x)$ follows from 
\begin{align}\label{vanish:P}
\Hom(P, \Psi(I_x))=0
\end{align}
for any $P \in \aA_G[1]$ with $\Imm Z_G^{\dag}(P)=0$. 
The above vanishing holds since $\Psi^{-1}(P)$ is 
given by an iterated extensions of zero dimensional sheaves and 
objects $U[1]$ for torsion free $\mu$-stable sheaves with $\mu(U)=0$. 

Next we consider the case $(d, H^2)=(4, 4)$. 
In this case, $Z_G^{\dag}(\Psi(I_x))=\sqrt{-1}$, 
whose imaginary part is the twice of 
the smallest positive value $1/2$
of the RHS of 
(\ref{note:imm}). 
Therefore, besides the vanishing (\ref{vanish:P}), we need to 
show the following: 
if there is an exact sequence in $\aA_G[1]$
\begin{align*}
0 \to P_1 \to \Psi(I_x) \to P_2 \to 0
\end{align*}
with $\Imm Z_G^{\dag}(P_i)=1/2$, 
then we have the inequality in $(0, \pi]$
\begin{align*}
\arg Z_G^{\dag}(P_1) < \arg Z_G^{\dag}
(\Psi(I_x)) = \frac{\pi}{2}. 
\end{align*}
In order to show this, we first observe that 
$\Psi(\oO_X) \in \aA_G[1]$ is $Z_G^{\dag}$-stable 
by~\cite[Proposition~7.4.1]{BMT}. 
Since $\Psi(I_x)$ is a subobject of $\Psi(\oO_X)$
in $\aA_G[1]$, 
and $Z_G^{\dag}(\Psi(\oO_X))= -1 + \sqrt{-1}$, 
we have 
\begin{align*}
\arg Z_G^{\dag}(P_1) < \arg Z_G^{\dag}(\Psi(\oO_X)) = \frac{3}{4} \pi. 
\end{align*}
Since the image of $Z_G^{\dag}$ is contained in 
$\frac{1}{2}\mathbb{Z} + \frac{\sqrt{-1}}{2} \mathbb{Z}$, 
and $\Imm Z_G^{\dag}(P_1)=1/2$, 
the above inequality implies that 
$\arg Z_G^{\dag}(P_1) \le \pi /2$. 
It remains to exclude the following case:
\begin{align*}
Z_G^{\dag}(P_1)=Z_G^{\dag}(P_2)= \frac{\sqrt{-1}}{2}.
\end{align*}
Indeed, we see that there is no object $P \in \aA_G[1]$
such that $Z_G^{\dag}(P)= \sqrt{-1}/2$. 
Suppose that such an object $P$ exists, and 
let
\begin{align*}
\ch(\Psi^{-1}P) \sqrt{\td_X} =(r, l, s) \in H^0(X) \oplus H^2(X)
\oplus H^4(X)
\end{align*}
be the (untwisted) Mukai vector of $\Psi^{-1}P$. 
Then the condition $Z_G^{\dag}(P) =\sqrt{-1}/2$
is equivalent to 
\begin{align*}
l \cdot H + 2r=1, \quad l \cdot H + 2s =0. 
\end{align*}
This is a contradiction since both of $r$ and $s$ are
integers. 

\end{proof}

\subsection{The case of $n=3$,
$\varepsilon=-1$}
In this subsection, we study the case of $(n, \varepsilon)=(3, -1)$. 
This case seems to be the most interesting case studied in 
this paper, as a construction of $\sigma_G$ has to do 
with the study of coherent systems on the 
smooth projective curve $X$. 
Indeed
in Proposition~\ref{prop:Syst}, 
we constructed an equivalence  
 $\Theta$
between the category of coherent systems on $X$
and the heart $\aA_W$, given by 
\begin{align*}
\aA_W = \langle \mathbb{C}(0), \Psi \Coh(X) \rangle_{\rm{ex}}. 
\end{align*} 
Below we abbreviate $\Theta$ and 
regard any coherent system $(\oO_X^{\oplus R} \to F)$
on $X$
as an object in $\aA_W$.  

In the notation of Subsection~\ref{subsec:central2},
the polynomial $\widehat{W}$ 
is of type $(a_1, a_2, a_3, 1, d)$, which must 
belong to
one of the classifications (\ref{types:n=4}).
There are two possibilities:
\begin{align*}
(a_1, a_2, a_3, d)= \left\{ \begin{array}{ll}
(1, 1, 1, 4), & X \mbox{ is a genus 3 curve }\\
(3, 1, 1, 6), & X \mbox{ is a genus 2 curve. }
\end{array} \right. 
\end{align*} 
The Calabi-Yau manifold $\widehat{X}$ is
a K3 surface which is either a quartic surface or a 
double cover of $\mathbb{P}^2$, and it 
contains $X$ as an element in the linear system 
$X \in \lvert \widehat{H} \rvert$. 
The central charge $Z_G$ is described 
in terms of coherent systems as follows: 
\begin{lem}\label{lem:ZG:3}
For any coherent system $(\oO_X^{\oplus R} \to F)$
on $X$, we have 
\begin{align}\notag
&Z_G(\oO_X^{\oplus R} \to F) = \\
\notag
&C_W \left\{
-d(F) + R \left( 1- \cos \frac{2\pi}{d}  \right)
+ \frac{\sqrt{\widehat{H}^2 d}}{2} \left( r(F) - \frac{R}{2}
 \right) \sqrt{-1} \right\}. 
\end{align} 
Here we set $(r(F), d(F))=(\rank(F), \deg(F))$. 
\end{lem}
\begin{proof}
The result follows from (\ref{write:negative}), (\ref{Zdag:com}), 
Lemma~\ref{lem:ZG:4} and
an easy computation. 
\end{proof}
By setting $Z_G^{\dag} \cneq Z_G/C_W$,
we define 
the slope function $\mu$ on $\aA_W$ by 
\begin{align}\notag
\mu(\oO_X^{\oplus R} \to F) & \cneq
\frac{- \Imm Z_G^{\dag}(\oO_X^{\oplus R} \to F)}{R} \\
\label{slope:mu}
&= \frac{\sqrt{\widehat{H}^2 d}}{2} \cdot \left( \frac{1}{2} - 
\frac{r(F)}{R} \right). 
\end{align}
Here we set $\mu(\ast)=-\infty$ if $R=0$. 
\begin{lem}
Any object in $\aA_W$ admits a Harder-Narasimhan filtration 
with respect to $\mu$-stability. 
\end{lem} 
\begin{proof}
Although $\aA_W$ is a
noetherian abelian category, we have to take 
 a little care since the 
condition in~\cite[Proposition~2.12]{Tcurve1} is not satisfied in 
this case. 
Instead, we apply the argument
used in~\cite[Theorem~2.29]{Tolim}. 
Let $\cC \subset \aA_W$ be the 
subcategory consisting of objects
$(\oO_X^{\oplus R} \stackrel{s}{\to} F)$
such that $s$ is surjective. 
Note that the right orthogonal complement 
$\cC^{\perp}$ consists of objects 
of the form $(0 \to F')$. 
Any object $(\oO_X^{\oplus R} \stackrel{s}{\to} F) \in \aA_W$
fits into the exact sequence
\begin{align}\label{ex:system}
0 \to (\oO_X^{\oplus R} \stackrel{s}{\to} \Imm s) \to 
(\oO_X^{\oplus R} \stackrel{s}{\to} F) \to (0 \to \Cok(s)) \to 0  
\end{align}
showing that $(\cC, \cC^{\perp})$ is a torsion pair 
on $\aA_W$. Furthermore, since 
$\mu(\ast)=-\infty$
on $\cC^{\perp}$, 
we can easily see the following:
 an object $E \in \cC$ is $\mu$-semistable if and only 
if for any exact sequence $0 \to E_1 \to E \to E_2 \to 0$
in $\cC$, we have $\mu(E_1) \le \mu(E_2)$. 
(cf.~\cite[Lemma~2.27]{Tolim}.)
Since $\cC$ is a noetherian and artinian 
quasi-abelian category, an argument similar to~\cite[Theorem~2.29]{Tolim}
shows that any object in $\cC$ admits a $\mu$-Harder-Narasimhan 
filtration. Combined with (\ref{ex:system}), 
any object in $\aA_W$ admits a $\mu$-Harder-Narasimhan filtration. 
\end{proof}

\begin{rmk}\label{rmk:surj}
By the proof of the above lemma, 
we see the following: if an object $(\oO_X^{\oplus R} \stackrel{s}{\to} F)$
is $\mu$-semistable, then either $s$ is surjective or 
$R=0$. 
\end{rmk}

Before discussing 
the construction of $\sigma_G$, we 
review Clifford type theorem for coherent systems 
established by Newstead-Lange~\cite{LaNe}. 
For a smooth projective curve $C$ and $\alpha \in \mathbb{R}_{>0}$, 
recall that the $\alpha$-stability 
on $\mathrm{Syst}(C)$ is defined by the 
slope function
\begin{align}\label{slope:alpha}
(\oO_C^{\oplus R} \to F) \mapsto 
\frac{d(F) + \alpha \cdot R}{r(F)}. 
\end{align}
Here the above slope function is set to be $\infty$
if $r(F)=0$. 
\begin{thm}\emph{(\cite[Theorem~2.1, Remark~2.3]{LaNe})}\label{thm:Cli}
Let $C$ be a smooth projective curve of genus $g(C) \ge 2$, 
and $\alpha \in \mathbb{R}_{>0}$. 
Then for any $\alpha$-stable coherent system 
$(\oO_C^{\oplus R} \to F)$ with 
$0 \le d(F) < 2g(C) \cdot r(F)$, we have 
\begin{align}\label{ineq:Cli}
R \le \frac{d(F)}{2} + r(F). 
\end{align}
Moreover if $C$ is non-hyperelliptic, 
then the equality holds in (\ref{ineq:Cli})
only if $(\oO_X^{\oplus R} \to F)$ is isomorphic to either
\begin{align}\label{exclude}
H^0(\oO_C)\otimes \oO_C \to \oO_C \  \mbox{ or } \
H^0(\omega_C) \otimes \oO_{C} \to \omega_C. 
\end{align}
\end{thm}
Using the above result, we have the 
following lemma, which 
plays a crucial role in constructing 
a Gepner type stability condition on $\HMF(W)$:
\begin{lem}\label{lem:crucial}
Let $(\oO_X^{\oplus R} \stackrel{s}{\to} F)$ be a $\mu$-stable object
in $\aA_W$ such that $R=2r(F)>0$. Then we have 
\begin{align}\label{dRc}
d(F)> R\left( 1- \cos \frac{2\pi}{d} \right). 
\end{align}
\end{lem}
\begin{proof}
By the $\mu$-stability, the morphism $s$ 
is surjective (cf.~Remark~\ref{rmk:surj})
hence $d(F) \ge 0$.
Also the RHS of (\ref{dRc}) is $R$ if $d=4$ and $R/2$ if $d=6$. 
In both cases, they are smaller than $2g(X) \cdot r(F)=g(X) \cdot R$, 
so we may assume that $d(F) < 2g(X) \cdot r(F)$. 
Comparing the $\mu$-stability in
(\ref{slope:mu}) and the $\alpha$-stability in (\ref{slope:alpha}), 
we see that the $\mu$-stable object
$(\oO_X^{\oplus R} \to F)$ is $\alpha$-stable for 
$\alpha \gg 0$, 
since the set of quotient sheaves of $F$ with bounded 
above degrees is bounded. 
 Therefore, 
applying Theorem~\ref{thm:Cli}, we obtain 
the inequality $d(F) \ge R$. 
It remains to check that the equality $d(F)=R$ does 
not hold if $d=4$. In this case, $X$ is a quartic 
curve, so it is not hyperelliptic. Also 
$H^0(\oO_C)$ is one dimensional, $H^0(\omega_C)$ is three dimensional, 
hence the objects (\ref{exclude}) do not satisfy our 
assumption $R=2r(F)$. Therefore 
the case $d(F)=R$ is excluded. 
\end{proof}

As we discussed in Subsection~\ref{subsec:strategy}, 
the slope function $\mu$ defines a torsion pair (\ref{mu:tilting}) 
on $\aA_W$, and 
the associated tilting $\aA_G$
is given by (\ref{AG:tilting}).
We have the following result:
\begin{prop}\label{prop:3-1}
Suppose that $n=3$ and $\varepsilon=-1$. Then 
the triple
\begin{align}\label{triple:3-1}
(Z_G, \aA_G, \theta=\theta_W)
\end{align}
determines a Gepner type stability condition $\sigma_G$
on 
$\HMF(W)$ with respect to $(\tau, 2/d)$. 
\end{prop}
\begin{proof}
Obviously 
the triple (\ref{triple:3-1}) determines a 
stability condition if and only if the triple
\begin{align}\label{triple:3-11}
(Z_G^{\dag}=Z_G/C_W, \aA_G, \theta=0)
\end{align}
determines a stability condition. 
By the construction of $\aA_G$, any non-zero 
object
$E \in \aA_G$ satisfies 
$\Imm Z_G^{\dag}(E) \ge 0$. 
Moreover
Lemma~\ref{lem:ZG:3} and Lemma~\ref{lem:crucial}
imply that if 
$\Imm Z_G^{\dag}(E)=0$, 
then $\Ree Z_G^{\dag}(E)<0$ holds. 
Therefore the triple (\ref{triple:3-11})
satisfies the condition (\ref{Htheta}). 
Also by Lemma~\ref{lem:ZG:3}, the image of 
$Z_G^{\dag}$ is discrete, which enables us 
to apply the same argument of~\cite[Proposition~7.1]{Brs2} 
to prove the Harder-Narasimhan property
of the triple (\ref{triple:3-11}). 
Therefore the triple (\ref{triple:3-11}), hence (\ref{triple:3-1}), 
determines a stability condition. 
Let 
\begin{align*}
\sigma_G=(Z_G, \{\pP_G(\phi)\}_{\phi \in \mathbb{R}})
\end{align*}
be the stability condition on $\HMF(W)$
determined by the triple (\ref{triple:3-1}). 
We need to show that $\sigma_G$ is Gepner type with respect to 
$(\tau, 2/d)$. 
Note that, in our situation, the triple (\ref{triple:3-1})
satisfies the inequality in (\ref{condition:n=3}).
 Therefore, by Lemma~\ref{lem:con:n=3}, 
it is enough to check the $\sigma_G$-stability 
of $\tau\Psi(\oO_x)$, $\mathbb{C}(1)$ and 
$\tau^2 \Psi(\oO_x)$. The stabilities of these 
objects are proved in Lemma~\ref{lem:tauPsi}, Lemma~\ref{lem:check:C1}
and Lemma~\ref{lem:tau2Psi} below. 
\end{proof}
Below we check the $\sigma_G$-stability of 
the objects $\tau \Psi(\oO_x)$, $\mathbb{C}(1)$
and $\tau^2 \Psi(\oO_x)$. 
Let $\sigma_G^{\dag}$ be the stability condition 
on $\HMF(W)$ determined by the triple (\ref{triple:3-11}). 
It differs from $\sigma_G$ by an action of $\mathbb{C}$, 
so it is enough to check the $\sigma_G^{\dag}$-stability
of these objects. 
\begin{lem}\label{lem:tauPsi}
For any $x\in X$, the object $\tau \Psi(\oO_x)$ is $\sigma_G^{\dag}$-stable. 
\end{lem}
\begin{proof}
By Lemma~\ref{lem:tau:ox}, the object $\tau \Psi(\oO_x)$
is an object in $\aA_W$, given by 
the coherent system 
$(\oO_X \to \oO_x)$. 
It is $\mu$-stable with $\mu(\ast)>0$, hence 
$\tau \Psi(\oO_x)[-1] \in \aA_G$. 
On the other hand, we have
\begin{align}\label{image:3-1}
\Imm Z_G^{\dag}(\aA_G) \subset \left\{
\frac{\sqrt{\widehat{H}^2 d}}{4} \times \mathbb{Z}_{\ge 0}  \right\}. 
\end{align}
The imaginary part of $Z_G^{\dag}(\tau\Psi(\oO_x)[-1])$ is
$\sqrt{\widehat{H}^2 d}/4$, which is the 
smallest positive number in the RHS of (\ref{image:3-1}). 
By Lemma~\ref{lem:cmin}, it is enough to check that
there is no non-zero 
morphism from any object $P \in \aA_G$
with $\Imm Z_G^{\dag}(P)=0$ to 
$\Psi(\oO_x)[-1]$. 
Since $\tau\Psi(\oO_x)[-1] \in \aA_W[-1]$, 
and $P \in \aA_W$
by Sublemma~\ref{Im=0} below, 
there is no non-zero 
morphism from $P$ to $\tau \Psi(\oO_x)[-1]$. 
 \end{proof}
We have used the following sublemma. 
The proof is obvious from the construction of $\aA_G$, 
and we omit it. 
\begin{sublem}\label{Im=0}
A non-zero object $P \in \aA_G$ satisfies $\Imm Z_G^{\dag}(P)=0$
if and only if $P \in \aA_W$, and it is 
given by an iterated extensions of 
$\mu$-stable coherent systems 
$(\oO_X^{\oplus R} \to F)$ with 
$\mu(\oO_X^{\oplus R} \to F)=0$, and coherent systems 
of the form $(0 \to \oO_y)$ for $y\in X$. 
\end{sublem}

Next we check the stability of $\mathbb{C}(1)$. 
\begin{lem}\label{lem:check:C1}
The object $\mathbb{C}(1)$ is $\sigma_G^{\dag}$-stable. 
\end{lem}
\begin{proof}
By 
Corollary~\ref{cor:canonical}, 
we have
the isomorphism (again we abbreviate $\Theta$)
\begin{align*}
\mathbb{C}(1)[-1] \cong 
(\oO_X \otimes R_1 \stackrel{s}{\to} \oO_X(1))
\end{align*}
such that $s$ is the natural evaluation map. 
In particular $H^0(s)$ is an isomorphism, and 
$s$ is surjective since $\oO_X(1)$ is globally generated.
Then we apply 
Sublemma~\ref{sublem:line} below
to show that $\mathbb{C}(1)[-1]$ is $\mu$-stable. 
The slope $\mu(\mathbb{C}(1)[-1])$ equals to $1/6$ if $d=4$ and $0$
if $d=6$. 
In the former case, 
$\mathbb{C}(1)[-2] \in \aA_G$, and 
the imaginary part
of $Z_G^{\dag}(\mathbb{C}(1)[-2])$ is $1$, which is the 
smallest positive value of the RHS of (\ref{image:3-1}). 
Hence $\sigma_G^{\dag}$-stability of $\mathbb{C}(1)[-2]$
follows from the same argument of 
Lemma~\ref{lem:tauPsi}. 
In the latter case, 
$\mathbb{C}(1)[-1] \in \aA_G$ and 
the imaginary 
part of $Z_G^{\dag}(\mathbb{C}(1)[-1])$ is $0$.
Hence
the $\sigma_G^{\dag}$-stability of $\mathbb{C}(1)[-1]$
follows from 
its $\mu$-stability
and Sublemma~\ref{Im=0}. 
\end{proof}
\begin{sublem}\label{sublem:line}
For an object $(\oO_X^{\oplus R} \stackrel{s}{\to} \lL) \in \aA_W$
with $R>0$ and 
$\lL \in \Pic(X)$, it is $\mu$-stable if
$s$ is surjective and 
$H^0(s)$ is injective. 
\end{sublem}
\begin{proof}
Suppose that $s$ is surjective and $H^0(s)$ is injective. Let
\begin{align*}
0 \to (\oO_X^{\oplus R_1} \to \lL_1)
\to (\oO_X^{\oplus R} \to \lL) \to (\oO_X^{\oplus R_2} \to \lL_2) \to 0
\end{align*}
be an exact sequence of coherent systems. 
It is enough to show the inequality
\begin{align}\label{check:stable}
\mu(\oO_X^{\oplus R} \to \lL) < 
\mu(\oO_X^{\oplus R_2} \to \lL_2). 
\end{align}
By our assumption, $\lL_1 \neq 0$ and $R_2 \neq 0$, 
hence $\lL_2$ is a zero dimensional sheaf. 
Therefore the RHS of (\ref{check:stable}) equals to $1/2$, while 
the LHS of (\ref{check:stable}) is less than $1/2$. 
Hence (\ref{check:stable}) holds. 
\end{proof}
It remains to prove the following lemma: 
\begin{lem}\label{lem:tau2Psi}
The object $\tau^2 \Psi(\oO_x)$ is
an object in $\pP_G(\theta_W + 1+4/d)$. 
\end{lem}
\begin{proof}
Applying $\tau$ to the exact sequence (\ref{tau:ox}), 
we obtain the distinguished triangle in $\HMF(W)$
\begin{align*}
\tau \Psi(\oO_x) \to \tau^2 \Psi(\oO_x) \to \mathbb{C}(1). 
\end{align*}
Combined with Lemma~\ref{lem:tau:ox}
and Corollary~\ref{cor:canonical}, 
the object $\tau^2 \Psi(\oO_x)$ is obtained 
as a cone of the morphism
of coherent systems:
\begin{align*}
\xymatrix{
\oO_X \otimes R_1 \ar[r]^{s} \ar[d]_{\gamma_1} & \oO_X(1) \ar[d]^{\gamma_2} \\
\oO_X \ar[r] & \oO_x. 
}
\end{align*}
Here $s$ is the evaluation map and 
$(\gamma_1, \gamma_2)$ is the morphism of coherent systems. 
Since $(\gamma_1, \gamma_2)$ is non-zero, both of 
$\gamma_1$, $\gamma_2$ are non-zero, hence
they are surjective. 
Therefore the object
$\tau^2 \Psi(\oO_x)[-1]$ is given by the coherent system
\begin{align}\label{tau:2}
(\oO_X \otimes R_{1, x} \stackrel{s_x}{\to} \oO_X(1) \otimes I_x).  
\end{align}
Here $R_{1, x}$ is the subspace of $R_1$ which vanishes at $x$, 
and $I_x$ is the ideal sheaf which defines $x$. 

First suppose that $d=4$. In this case, $R_{1, x}$ is two dimensional, 
and $s_x$ is surjective since any of two lines in $\mathbb{P}^2$
determined by generic two elements in $R_{1, x}$
intersect only at $x$ transversally. 
Also $H^{0}(s_x)$ is injective since $H^0(s)$ is an isomorphism. 
Therefore the coherent system (\ref{tau:2}) is 
$\mu$-stable by Sublemma~\ref{sublem:line}. 
Since $\mu=0$ for the coherent system (\ref{tau:2}), 
we have $\tau^2 \Psi(\oO_x)[-1] \in \aA_G$, and 
it is $\sigma_G$-stable. 
In particular $\tau^2 \Psi(\oO_x)$ 
is an object in $\pP_G(\theta_W +2)$. 

Next suppose that $d=6$. In this case, $R_{1, x}$ is one 
dimensional, and $s_x$ is not surjective at $x$. 
In particular the coherent system (\ref{tau:2}) 
is not $\mu$-semistable. 
However we can show the $Z_G^{\dag}$-stability, 
hence $\sigma_G^{\dag}$-stability, of (\ref{tau:2}) in the 
following way: 
the sheaf $\oO_X(1) \otimes I_x$ is 
isomorphic to $\oO_X(x')$ for another point $x'\in X$, and 
(\ref{tau:2}) is
isomorphic to the coherent system 
\begin{align*}
(\oO_X \stackrel{s'}{\to} \oO_X(x'))
\end{align*}
where $s'$ is a natural inclusion. 
There is an exact sequence of coherent systems
\begin{align}\label{ex:system2}
0 \to (\oO_X \stackrel{\id}{\to} \oO_X) 
\to (\oO_X \to \oO_X(x')) \to (0 \to \oO_{x'}) \to 0. 
\end{align}
Both of the objects $(\oO_X \stackrel{\id}{\to} \oO_X)$
and $(0 \to \oO_{x'})$ are $\mu$-stable with 
negative slopes, hence the object (\ref{tau:2}) 
is an object in $\aA_G$. 
Also the imaginary part of $Z_G^{\dag}(\oO_X \to \oO_X(x'))$
is $\sqrt{3}/2$, which is the smallest
positive value of the RHS of (\ref{image:3-1}). 
Therefore by Lemma~\ref{lem:cmin}, 
it is enough to check that there is 
no non-zero morphism from $P \in \aA_G$ with 
$\Imm Z_G^{\dag}(P)=0$ to the object $(\oO_X \to \oO_X(x'))$. 
This follows from Sublemma~\ref{Im=0}, since both of 
$(\oO_X \stackrel{\id}{\to} \oO_X)$
and $(0 \to \oO_{x'})$ are $\mu$-stable with 
negative slopes, and the exact sequence (\ref{ex:system2}) does not split. 
\end{proof}

\subsection{The case of $n=2$, 
$\varepsilon=-2$}
Finally in this subsection, 
we study the case of $(n, \varepsilon)=(2, -2)$. 
In this case, $X$ is a finite number of smooth points, 
represented by points 
$p^{(j)}=(p_1^{(j)}, p_2^{(j)}) \in \mathbb{C}^2$
for $1\le j\le \sharp X$. 
The heart $\aA_W$ is given by 
\begin{align*}
\aA_W = \langle \mathbb{C}(1), \mathbb{C}(0), 
\Psi(\oO_x) : x \in X \rangle_{\rm{ex}}. 
\end{align*}
By Lemma~\ref{lem:compute2} and Lemma~\ref{lem:equivalence}, 
the heart $\aA_W$ is equivalent to
the abelian category of representations of a certain 
quiver $\qQ$ with relations. 
As in the previous subsection, 
we have the following possibilities:
 \begin{align*}
&
(a_1, a_2, d, \sharp X)=(1, 1, 4, 4), \quad 
\qQ=    \xygraph{!~-{@{=}|@{>}} !~:{@{>}}
{\bullet}*+!D{} -^{X_1}_{X_2}[rr]
        {\bullet}*+!D{}
        (:^{\pi^{(4)}}[ru]{\bullet}*+!D{}, 
:_{\pi^{(3)}}[rru]{\bullet}*+!D{},      
        :_{\pi^{(1)}}[rd]{\bullet}*+!D{}, 
:^{\pi^{(2)}}[rrd]{\bullet}*+!D{})
    } \\
&
(a_1, a_2, d, \sharp X)=(3, 1, 6, 2), \quad
\qQ= \xygraph{
{\bullet}*+D{} :^{X_2}[rr]
       {\bullet}*+!D{}
        (:^{\pi^{(2)}}[ru]{\bullet}*+!D{},      
        :_{\pi^{(1)}}[rd]{\bullet}*+!D{})
    } 
\end{align*}
Also by Corollary~\ref{cor:quiver}, 
in the $d=4$ case, we put the 
the following relation:
\begin{align}\label{Q:relation}
p_2^{(j)} \pi^{(j)} X_1 = p_1^{(j)} \pi^{(j)} X_2, \quad 
1\le j \le 4. 
\end{align}
There is no relation in the $d=6$ case. 
By (\ref{Zdag:com}) and the above 
classification, 
the central charge $Z_G$ is given as follows: 
\begin{lem}\label{ZG:2-2}
If we write the K-theory class of 
$E \in \aA_W$ as 
\begin{align}\label{K-th}
[E]= v_1 [\mathbb{C}(1)] + v_0[ \mathbb{C}(0)] + \sum_{j=1}^{\sharp X}
w_j
[\Psi(\oO_{p^{(j)}})]
\end{align}
for $v_j, w_j \in \mathbb{Z}_{\ge 0}$, 
then we 
have 
\begin{align*}
Z_G(E)= C_W &\left\{  -w + \left( 1- \cos \frac{2\pi}{d} \right) v_0 
+ v_1 \right. \\
&+\left. \left\{ -\sin \frac{2\pi}{d} v_0 + \left( \sin \frac{2\pi}{d} -
\sin \frac{4\pi}{d}  \right)v_1 \right\} \sqrt{-1} \right\}. 
\end{align*}
Here we set 
$w \cneq \sum_{j=1}^{\sharp X} w_j$. 
\end{lem}

By setting $Z_G^{\dag} \cneq Z_G/C_W$, we
define the slope function $\mu$ on $\aA_W$ by 
\begin{align*}
\mu(E) & \cneq \frac{\Ree Z_G^{\dag}(E)}{w} \\
&=\frac{1}{w} \left(v_1 + \left(1-\cos \frac{2\pi}{d} \right)v_0  \right) -1. 
\end{align*}
Here we set $\mu(E)=\infty$ if $w=0$. 
The above slope function defines the $\mu$-stability 
on $\aA_W$. 
Since $\aA_W$ is noetherian and artinian, 
any object in $\aA_W$ admits a $\mu$-Harder-Narasimhan filtration.
(cf.~\cite[Proposition~2.12]{Tcurve1}.)
We prepare the following lemma:
\begin{lem}\label{lem:2-2}
For any non-zero $\mu$-stable object $E \in \aA_W$ with 
$\mu(E)=0$, we have $\Imm Z_G^{\dag}(E)<0$. 
\end{lem}
\begin{proof}
In the case of $d=6$, 
we have $\Imm Z_G^{\dag}(E)=-\sqrt{3}v_0/2$, 
which is non-positive. If $v_0=0$, 
the condition $\mu(E)=0$ implies $v_1 =w \neq 0$. 
Therefore $E$ decomposes into direct sums, which 
contradicts to the $\mu$-stability of $E$. 

In the case of $d=4$, the claim is equivalent to that 
if $E$ is $\mu$-stable with 
$v_0+v_1=w \neq 0$, then $v_1< v_0$. 
Let us represent $E$ 
as a representation of $\qQ$:
\begin{align}\label{rep:Q}
  \xygraph{!~-{@{=}|@{>}} !~:{@{>}}
{\bullet}*+!D{V^{(1)}} -^{X_1}_{X_2}[rr]
        {\bullet}*+!D{V^{(0)}}
        (:^{\pi^{(4)}}[ru]{\bullet}*+!D{W^{(4)}}, 
:_{\pi^{(3)}}[rru]{\bullet}*+!D{W^{(3)}},      
        :_{\pi^{(1)}}[rd]{\bullet}*+!D{W^{(0)}}, 
:^{\pi^{(2)}}[rrd]{\bullet}*+!D{W^{(2)}})
    } 
\end{align}
where $V^{(i)}$ are $v_i$-dimensional and $W^{(i)}$ 
are $w_i$-dimensional. 
Suppose by a contradiction that $v_1 \ge v_0$. 
By the relation (\ref{Q:relation}), we have the linear maps
for $1\le i\le 4$
\begin{align}\label{map:V}
p_1^{(i)} X_2 -p_2^{(i)}X_1 \colon V^{(1)} \to \Ker(\pi^{(i)}). 
\end{align}
Because of the condition $w\neq 0$, there is 
$1\le i\le 4$ such that $w^{(i)} \neq 0$. 
Also the $\mu$-stability of $E$ implies 
that $\pi^{(i)}$ is surjective. 
Therefore the assumption $v_1 \ge v_0$
implies that $\dim V^{(1)} > \dim \Ker(\pi^{(i)})$, 
and there is a non-zero element $v\in V^{(1)}$ 
which is mapped to zero by (\ref{map:V}). 
Let us consider the sub quiver representation of 
(\ref{rep:Q}) generated by $v$. 
By taking account the relation (\ref{Q:relation})
into consideration, the corresponding 
subobject $F \subset E$ in $\aA_W$ has the K-theory class 
either
\begin{align*}
[\mathbb{C}(1)] \ \mbox{ or } \ [\mathbb{C}(1)] + [\mathbb{C}(0)] 
\ \mbox{ or } \
[\mathbb{C}(1)] + [\mathbb{C}(0)] + [\Psi(\oO_{p^{(i)}})]. 
\end{align*}
Hence we have $\mu(F)>0$, which contradicts to the $\mu$-stability 
of $E$ with $\mu(E)=0$. 
\end{proof}

As before, 
the slope function $\mu$ defines a torsion pair (\ref{mu:tilting}) 
on $\aA_W$, and 
the associated tilting $\aA_G$
given by (\ref{AG:tilting}).
We have the following result:
\begin{prop}\label{prop:2-2}
Suppose that $n=2$ and $\varepsilon=-2$. 
Then the triple 
\begin{align}\label{triple:2-2}
\left(Z_G, \aA_G, \theta=\theta_W + \frac{1}{2} \right)
\end{align}
determines a Gepner type stability condition $\sigma_G$
on $\HMF(W)$
with respect to $(\tau, 2/d)$. 
\end{prop}
\begin{proof}
The triple (\ref{triple:2-2}) determines a stability 
condition if and only if the triple
\begin{align}\label{triple:2-20}
\left( Z_G^{\dag}=Z_G/C_W, \aA_G, \theta=\frac{1}{2} \right)
\end{align}
determines a stability condition. 
By Lemma~\ref{ZG:2-2}, 
Lemma~\ref{lem:2-2} and the construction 
of $\aA_G$, 
the triple (\ref{triple:2-20})
satisfies the condition (\ref{Htheta}). 
Also since the image of $Z_G^{\dag}$ is discrete, the 
same argument of~\cite[Proposition~7.1]{Brs2}
shows the Harder-Narashiman property of (\ref{triple:2-20}). 
Therefore the triples (\ref{triple:2-2}), (\ref{triple:2-20})
determine stability conditions $\sigma_G$, $\sigma_G^{\dag}$
respectively. 
Since they only differ by a $\mathbb{C}$-action, 
by Lemma~\ref{lem:fur}, 
it is enough to check the $\sigma_G^{\dag}$-stability of 
$\tau\Psi(\oO_x)$, $\mathbb{C}(1)$ and $\mathbb{C}(2)$. 
These are checked in Lemma~\ref{check:1}, Lemma~\ref{check:2}
and Lemma~\ref{check:3} below. 
\end{proof}

\begin{lem}\label{check:1}
For any $x\in X$, the
 object $\tau \Psi(\oO_x)$ is $\sigma_G^{\dag}$-stable. 
\end{lem}
\begin{proof}
By Lemma~\ref{lem:tau:ox}, we have 
$\tau \Psi(\oO_x) \in \aA_W$. 
By the exact sequence (\ref{tau:ox}),
the object $\tau\Psi(\oO_x)$ is $\mu$-stable with non-positive slope. 
Therefore $\tau\Psi(\oO_x)$ is an object in $\aA_G$.  
In the $d=4$ case, we have $Z_G^{\dag}(\tau\Psi(\oO_x))=-\sqrt{-1}$. 
Since the image of $Z_G^{\dag}$ is $\mathbb{Z} + \mathbb{Z} \sqrt{-1}$, 
this immediately implies that $\tau \Psi(\oO_x)$ is $\sigma_G^{\dag}$-stable. 
In the $d=6$ case, 
note that $\Ree Z_G^{\dag}(\ast)$ on $\aA_G$ is contained in 
$\mathbb{Z}_{\le 0} \times 1/2$. 
Since $\Ree Z_G^{\dag}(\tau\Psi(\oO_x))=-1/2$,
by Lemma~\ref{lem:cmin}, 
it is enough to check that $\Hom(P, \tau\Psi(\oO_x))=0$
for any $P \in \aA_G$ with $\Ree Z_G^{\dag}(P)=0$.
By our construction of $\aA_G$, we have $P \in \aA_W$, 
and it is $\mu$-semistable 
with $\mu(P)=0$. 
Since $\tau\Psi(\oO_x)$ is $\mu$-stable with negative slope, 
we have $\Hom(P, \tau\Psi(\oO_x))=0$.  
\end{proof}

\begin{lem}\label{check:2}
The object $\mathbb{C}(1)$ is $\sigma_G^{\dag}$-stable. 
\end{lem}
\begin{proof}
Since $\mathbb{C}(1)$ is $\mu$-stable with 
$\mu(\mathbb{C}(1))=\infty$, we have 
$\mathbb{C}(1)[-1] \in \aA_G$. 
In the $d=4$ case, note that $\Ree Z_G^{\dag}(\ast)$ is 
contained in $\mathbb{Z}_{\le 0}$. 
Since
$\Ree Z_G^{\dag}(\mathbb{C}(1)[-1])=-1$, 
by Lemma~\ref{lem:cmin}, 
it is enough to check that $\Hom(P, \mathbb{C}(1)[-1])=0$
for any $P \in \aA_G$ with $\Ree Z_G^{\dag}(P)=0$, 
which follows since $P$ is an object in $\aA_W$. 
In the $d=6$ case, 
let $E$ be a non-trivial 
quotient of $\mathbb{C}(1)[-1]$ in $\aA_G$. 
Since $Z_G^{\dag}(\mathbb{C}(1)[-1])=-1$, 
it is enough to check that 
$\Imm Z_G^{\dag}(E)<0$. 
To check this, note that we have $E\in \aA_W$
 since $\mathbb{C}(1)$ is a simple 
object in $\aA_W$. 
Let us write the K-theory class of $E$ as in (\ref{K-th}). 
Since $\Imm Z_G^{\dag}(E)= -\sqrt{3}v_0/2$ is non-positive, 
it is enough 
to exclude the case of $v_0=0$. 
If $v_0=0$, then $E$ is a direct sum of objects
$\mathbb{C}(1)$ and $\Psi(\oO_x)$ for $x\in X$.
 Then there is a non-zero 
morphism from $\mathbb{C}(1)[-1]$ to $\Psi(\oO_x)$, 
which contradicts to Lemma~\ref{lem:compute2}. 
\end{proof}

\begin{lem}\label{check:3}
The object $\mathbb{C}(2)$ is $\sigma_G^{\dag}$-stable. 
\end{lem}
\begin{proof}
We first discuss the case of $d=4$. In this case, by 
Corollary~\ref{cor:quiver2}, the object $\mathbb{C}(2)[-1]$
is an object in $\aA_W$ represented by
the following quiver representation
\begin{align}\label{rep:Q2}
  \xygraph{!~-{@{=}|@{>}} !~:{@{>}}
{\bullet}*+!D{R_1} -^{x_1}_{x_2}[rr]
        {\bullet}*+!D{R_2}
        (:^{\pi^{(4)}}[ru]{\bullet}*+!D{\mathbb{C}}, 
:_{\pi^{(3)}}[rru]{\bullet}*+!D{\mathbb{C}},      
        :_{\pi^{(1)}}[rd]{\bullet}*+!D{\mathbb{C}}, 
:^{\pi^{(2)}}[rrd]{\bullet}*+!D{\mathbb{C}})
    } 
\end{align}
Here $R_1$ is two dimensional and $R_2$ is three dimensional. 
We first show that $\mathbb{C}(2)[-1]$ is $\mu$-stable. 
Let $E$ be a subobject of $\mathbb{C}(2)[-1]$ in $\aA_W$
given by a representation (\ref{rep:Q}), with 
$v^{(j)} \cneq \dim V^{(j)}$ and $w\cneq \sum_{j=1}^{4} \dim W^{(j)}$. 
Let $I$ be the subset of $i \in \{1, 2, 3, 4\}$
such that $W^{(i)}=0$. 
Similarly to the proof of Lemma~\ref{lem:30:check}, we have
the inequality
$v^{(0)} \le 3- \sharp I$.
Also we may assume that
$v^{(1)} \le 1$, since otherwise
$F$ coincides with $\mathbb{C}(2)[-1]$. 
Therefore we obtain 
the inequality
\begin{align*}
v^{(0)} + v^{(1)} \le 4-\lvert I \rvert =w
\end{align*} 
which
implies that $\mu(E) \le 0< \mu(\mathbb{C}(2)[-1])=1/4$.
Therefore  
$\mathbb{C}(2)[-1]$ is $\mu$-stable with positive slope, 
hence we have $\mathbb{C}(2)[-2] \in \aA_G$. 
Since $Z_G^{\dag}(\mathbb{C}(2)[-2]) =-1 + \sqrt{-1}$, 
and the image of $Z_G^{\dag}$ is $\mathbb{Z} + \mathbb{Z}\sqrt{-1}$, 
by Lemma~\ref{lem:cmin}, 
the $\sigma_G^{\dag}$-stability of $\mathbb{C}(2)[-2]$ follows if we 
check that $\Hom(P, \mathbb{C}(2)[-2])=0$
for any $P \in \aA_G$ with $\Ree Z_G^{\dag}(P)=0$. 
This follows since $P\in \aA_W$ and $\mathbb{C}(2)[-2] \in \aA_W[-1]$. 

Next we discuss the case of $d=6$. By Lemma~\ref{lem:filt:A}, the 
object $\mathbb{C}(2)[-1] \in \aA_W$ is represented by 
the following quiver representation
\begin{align*}
\xygraph{
{\bullet}*+!D{0} :[rr]
       {\bullet}*+!D{\mathbb{C}}
        (:^{\id}[ru]{\bullet}*+!D{\mathbb{C}},      
        :_{\id}[rd]{\bullet}*+!D{\mathbb{C}})
    } 
\end{align*}
The above object is obviously 
$\mu$-stable. Since $\mu(\mathbb{C}(2)[-1])=-3/4$, 
it follows that $\mathbb{C}(2)[-1] \in \aA_G$. 
Also we have $Z_G^{\dag}(\mathbb{C}(2)[-1])=-(1/2 + 3\sqrt{-3}/2)$, 
and $\Ree Z_G^{\dag}(\ast)$ is contained in $\mathbb{Z}_{\le 0} \times 1/2$
on $\aA_G$. Applying Lemma~\ref{lem:cmin}, 
the $\sigma_G^{\dag}$-stability of $\mathbb{C}(2)[-1]$
follows if we check that 
$\Hom(P, \mathbb{C}(2)[-1])=0$ for 
any $P \in \aA_G$ with $\Ree Z_G^{\dag}(P)=0$. 
Since such $P$ is a $\mu$-semistable 
object in $\aA_W$ with $\mu(P)=0$, 
and $\mathbb{C}(2)[-1]$ 
is $\mu$-stable with negative slope, 
it follows that 
we have $\Hom(P, \mathbb{C}(2)[-1])=0$. 
\end{proof}

Kavli Institute for the Physics and 
Mathematics of the Universe, University of Tokyo,
5-1-5 Kashiwanoha, Kashiwa, 277-8583, Japan.

\textit{E-mail address}: yukinobu.toda@ipmu.jp

\end{document}